\documentclass[11pt,leqno]{amsart}
\usepackage{a4wide}
\usepackage{amssymb}
\usepackage{amsmath}
\usepackage{amsthm}
\usepackage{mathrsfs}
\usepackage{euscript}
\usepackage[notref,notcite]{showkeys}
\usepackage[all]{xypic}
\usepackage[usenames]{color}
\usepackage{xr-hyper}
\usepackage{hyperref}

\theoremstyle{plain}

\newcommand{\id}{\operatorname{id}}
\newcommand{\im}{\operatorname{im}}

\newcommand{\pr}{\operatorname{pr}}

\newcommand{\Hom}{\operatorname{Hom}}

\newcommand{\Rep}{\operatorname{Rep}}

\newcommand{\Aut}{\operatorname{Aut}}

\newcommand{\GL}{\operatorname{GL}}

\newcommand{\Gal}{\operatorname{Gal}}

\newcommand{\bA}{\mathbf{A}}

  \newcommand{\bE}{\mathbf{E}}
 
 \newcommand{\bQ}{\mathbf{Q}}
 
 \newcommand{\bZ}{\mathbf{Z}}

\newcommand{\rA}{\mathscr{A}}

 \newcommand{\vp}{\varphi}

\newtheorem{theorem}{Theorem}[section]
\newtheorem{corollary}[theorem]{Corollary}
\newtheorem{lemma}[theorem]{Lemma}
\newtheorem{remark}[theorem]{Remark}
\newtheorem{proposition}[theorem]{Proposition}

\newtheorem{definition}[theorem]{Definition}

\theoremstyle{remark}

\title[Coates-Wiles homomorphisms and Iwasawa cohomology]{\textbf{Coates-Wiles homomorphisms and Iwasawa cohomology for Lubin-Tate extensions}}
\author{Peter Schneider and Otmar Venjakob}
\address{Universit\"{a}t M\"{u}nster,  Mathematisches Institut,  Einsteinstr. 62,
48291 M\"{u}nster,  Germany,
 http://www.uni-muenster.de/math/u/schneider/ }%
\email{pschnei@uni-muenster.de }%

\address{Universit\"{a}t Heidelberg,  Mathematisches Institut,  Im Neuenheimer Feld 288,  69120
Heidelberg,  Germany,
 http://www.mathi.uni-heidelberg.de/$\,\tilde{}\,$venjakob/}
\email{venjakob@mathi.uni-heidelberg.de}

\subjclass{11Sxx}%
\keywords{ $p$-adic Hodge theory, explicit reciprocity law, Coates-Wiles homomorphisms, Lubin-Tate formal groups, ramified Witt vectors, Artin-Schreier pairing, Schmidt-Witt residue formula, Coleman power series}%

\dedicatory{To John Coates on the occasion of his 70th birthday}

\begin{document}

\begin{abstract}
 For the $p$-cyclotomic tower of $\mathbb{Q}_p$  Fontaine established a description of local Iwasawa cohomology with coefficients in a local Galois representation $V$ in terms of the $\psi$-operator  acting on the attached etale $(\varphi,\Gamma)$-module $D(V)$. In this article we generalize  Fontaine's result to the case of arbitratry Lubin-Tate towers $L_\infty$ over finite extensions $L$ of $\mathbb{Q}_p$ by using the Kisin-Ren/Fontaine equivalence of categories between Galois representations and $(\varphi_L,\Gamma_L)$-module and extending parts of \cite{Her}, \cite{Sch}. Moreover, we prove a kind of explicit reciprocity law which calculates the  Kummer map over $L_\infty$ for the multiplicative group twisted with the dual of the Tate module $T$ of the Lubin-Tate formal group in terms of Coleman power series and the attached $(\varphi_L,\Gamma_L)$-module. The proof is based on a generalized Schmid-Witt residue formula. Finally, we extend the explicit reciprocity law of Bloch and Kato \cite{BK} Thm.\ 2.1 to our situation  expressing the Bloch-Kato exponential map for $L(\chi_{LT}^r)$ in terms of generalized Coates-Wiles homomorphisms, where the Lubin-Tate characater $\chi_{LT}$ describes the Galois action on $T.$
\end{abstract}

\maketitle

\section{Introduction}

The invention of Coates-Wiles homomorphisms and Coleman \cite{Col} power series - practically at the same time, actually the Coleman power series for elliptic units show already up in \cite{CW} - were the starting point of a range of new and important developments in arithmetic and they have not lost their impact and fascination up to now. In order to recall it we shall first introduce some notation.

Consider a finite extension $L$ of $\mathbb{Q}_p$ and fix a Lubin-Tate formal group $LT$ over the integers $o_L$ (with uniformizer $\pi_L$). By $t_0=(t_{0,n})$ we denote a generator of the Tate module $T$ of $LT$ as $o_L$-module. The $\pi_L^n$-division points generate a tower of Galois extensions $L_n=L(LT[\pi_L^n])$ of $L$ the union of which we denote by $L_\infty$ with Galois group $\Gamma_L.$ Coleman assigned to any norm compatible unit $u=(u_n)\in \varprojlim_n L_n^\times$ a Laurent series $g_{u,t_0}\in o_L((Z))^\times$ such that $g_{u,t_0}(t_{0,n})=u_n$ for all $n.$ If $\partial_{\mathrm{inv}}$ denotes the invariant derivation with respect to $LT$, then, for $r\geq 1,$ the $r$th Coates-Wiles homomorphism is given by
\begin{equation*}
  \psi_{CW}^r:\varprojlim_n L_n^\times \to L(\chi_{LT}^r),\;\;\; u \mapsto  \frac{1}{r!}\partial_{\mathrm{inv}}^r \log g_{u,t_0}(Z)|_{Z=0} := \frac{1}{r!}\partial_{\mathrm{inv}}^{r-1}\frac{\partial_{\mathrm{inv}}g_{u,t_0}(Z)}{ g_{u,t_0}(Z)}|_{Z=0} \ ,
\end{equation*}
it is Galois invariant and satisfies at least heuristically - setting $t_{LT}=\log_{LT}(Z)$ - the equation
\begin{equation*}
  \log g_{u,t_0}(Z)=\sum_r \psi_{CW}^r(u)t_{LT}^r
\end{equation*}
the meaning of which in $p$-adic Hodge theory has been crucially exploited, e.g.\ \cite{Fo2}, \cite{Co5}, \cite{CC}.

Explicitly or implicitly this mysterious map plays  - classically  for the multiplicative group over $\mathbb{Q}_p$ - a crucial role in the Bloch-Kato (Tamagawa number) conjecture \cite{BK}, in the study of special $L$-values \cite{CS}, in explicit reciprocity laws \cite{Kat}, \cite{Wi} and even in the context of the cyclotomic trace map from $K$-theory into topological cyclic homology for $\mathbb{Z}_p$ \cite{BM}.

In this context one motivation for the present work is to understand Kato's (and hence Wiles's) explicit reciprocity law in terms of $(\varphi_L,\Gamma_L)$-modules. Since in the classical situation a successful study of explicit reciprocity laws has been achieved by  Colmez, Cherbonnier/Colmez, Benois and Berger using Fontaine's work on $(\varphi,\Gamma)$-modules and Herr's calculation of Galois cohomology by means of them, the plan for this article is to firstly use Kisin-Ren/Fontaine's equivalence of categories (recalled in section \ref{sec:KR}) to find a description of Iwasawa cohomology $H_{Iw}^i(L_\infty,V)$ for the tower $L_\infty$ and a (finitely generated $o_L$-module) representation $V$  of $G_L$ in terms of a $\psi$-operator acting on the etale  $(\varphi_L,\Gamma_L)$-module $D_{LT}(V)$. To this aim we have to generalize parts of \cite{Her} in section \ref{sec:etale}, in particular the residue pairing, which we relate to Pontrjagin duality. But instead of using Herr-complexes (which one also could define easily in this context) we use local Tate-duality $H_{Iw}^i(L_\infty,V)\cong H^{2-i}(L_\infty,V^\vee(1))$ for $V$ being an $o_L$-module of finite length and an explicit calculation of the latter groups in terms of $(\varphi_L,\Gamma_L)$-modules inspired by \cite{Sch} and \cite{Fo1}, \cite{Fo2}. Using the key observation that $D_{LT}(V)^\vee\cong D_{LT}(V^\vee(\chi_{LT}))$ (due to the residue pairing involving differentials $\Omega^1\cong D_{LT}(o_L)(\chi_{LT})$ and the compatibility of inner Homs under the category equivalence) we finally establish in Theorem \ref{psi-version} the following exact sequence
\begin{equation*}
  0 \longrightarrow H^1_{Iw}(L_\infty,V) \longrightarrow D_{LT}(V(\tau)) \xrightarrow{\psi_L -1} D_{LT}(V(\tau)) \longrightarrow H^2_{Iw}(L_\infty,V) \longrightarrow 0
\end{equation*}
as one main result of this article, where the twist by $\tau=\chi_{LT}\chi_{cyc}^{-1}$ is a new phenomenon (disappearing obviously in the cyclotomic case) arising from the joint use of Pontrjagin and local Tate duality. The second main result is the explicit calculation of the twisted (by the $o_L$-dual $T^*$ of $T$) Kummer map
\begin{equation*}
    \varprojlim_{n} L_n^\times \otimes_{\mathbb{Z} }T^* \xrightarrow{\kappa \otimes_{\mathbb{Z}_p} T^*} H^1_{Iw}(L_\infty,\mathbb{Z}_p(1)) \otimes_{\mathbb{Z}_p }T^* \cong H^1_{Iw}(L_\infty,o_L(\tau)) \
\end{equation*}
in terms of Coleman series (recalled in section \ref{sec:Coleman}) and $(\varphi_L,\Gamma_L)$-modules, see Theorem \ref{Kummer-commutative}, which generalizes the explicit reciprocity laws of Benois and Colmez. Inspired by \cite{Fo1}, \cite{Fo2} we reduce its proof to an explicit reciprocity law, Proposition \ref{reduction-to-char-p}, in characteristic $p$, which in turn is proved  by the Schmid-Witt residue formula which we generalize to our situation in section \ref{sec:Witt}, see Theorem \ref{schmid-witt}.
In section \ref{sec:CoatesWiles} we generalize the approach sketched in \cite{Fo2} to prove in Theorem \ref{BK} a generalization of the explicit reciprocity law of Bloch and Kato \cite{BK} Thm.\ 2.1: again in this context the Bloch-Kato exponential map is essentally given by the Coates-Wiles homomorphism. As a direct consequence we obtain as Corollary \ref{Kato} a new proof for a special case of Kato's explicit reciprocity law for Lubin-Tate formal groups.

It is a great honour and pleasure to dedicate this work to John Coates who has been a source of constant inspiration for both of us.

We thank L.\ Berger for pointing Lemma \ref{FilBmaxexact}.i out to us and for a discussion about Prop.\ \ref{id-component} and G.\ Kings for making available to us a copy of \cite{Koe}. We also acknowledge support by the DFG Sonderforschungsbereich 878 at M\"unster and by the DFG research unit 1920 ``Symmetrie, Geometrie und Arithmetik'' at Heidelberg/Darmstadt.

\section*{Notation}

Let $\mathbb{Q}_p \subseteq L \subset \mathbb{C}_p$ be a field of finite degree $d$ over $\mathbb{Q}_p$, $o_L$ the ring of integers of $L$, $\pi_L \in o_L$ a fixed prime element, $k_L = o_L/\pi_L o_L$ the residue field, and $q := |k_L|$. We always use the absolute value $|\ |$ on $\mathbb{C}_p$ which is normalized by $|\pi_L| = q^{-1}$.

We fix a Lubin-Tate formal $o_L$-module $LT = LT_{\pi_L}$ over $o_L$ corresponding to the prime element $\pi_L$. We always identify $LT$ with the open unit disk around zero, which gives us a global coordinate $Z$ on $LT$. The $o_L$-action then is given by formal power series $[a](Z) \in o_L[[Z]]$. For simplicity the formal group law will be denoted by $+_{LT}$.

The power series $\frac{\partial (X +_{LT} Y)}{\partial Y}_{|(X,Y) = (0,Z)}$ is a unit in $o_L[[Z]]$ and we let $g_{LT}(Z)$ denote its inverse. Then $g_{LT}(Z) dZ$ is, up to scalars, the unique invariant differential form on $LT$ (\cite{Haz} \S5.8). We also let $\log_{LT}(Z) = Z + \ldots$ denote the unique formal power series in $L[[Z]]$ whose formal derivative is $g_{LT}$. This $\log_{LT}$ is the logarithm of $LT$ (\cite{Lan} 8.6). In particular, $g_{LT}dZ = d\log_{LT}$. The invariant derivation $\partial_\mathrm{inv}$ corresponding to the form $d\log_{LT}$ is determined by
\begin{equation*}
  f' dZ = df = \partial_\mathrm{inv} (f) d\log_{LT} = \partial_\mathrm{inv} (f) g_{LT} dZ
\end{equation*}
and hence is given by
\begin{equation}\label{f:inv}
  \partial_\mathrm{inv}(f) = g_{LT}^{-1} f' \ .
\end{equation}
For any $a \in o_L$ we have
\begin{equation}\label{f:dlog}
  \log_{LT} ([a](Z)) = a \cdot \log_{LT} \qquad\text{and hence}\qquad ag_{LT}(Z) = g_{LT}([a](Z))\cdot [a]'(Z)
\end{equation}
(\cite{Lan} 8.6 Lemma 2).

Let $T$ be the Tate module of $LT$. Then $T$ is
a free $o_L$-module of rank one, and the action of
$G_L := \Gal(\overline{L}/L)$ on $T$ is given by a continuous character $\chi_{LT} :
 G_L \longrightarrow o_L^\times$. Let $T'$ denote the Tate module of the $p$-divisible group Cartier dual to $LT$, which again is a free $o_L$-module of rank one. The Galois action on $T'$ is given by the continuous character $\tau := \chi_{cyc}\cdot\chi_{LT}^{-1}$, where
$\chi_{cyc}$ is the cyclotomic character.

For $n \geq 0$ we let $L_n/L$ denote the extension (in $\mathbb{C}_p$) generated by the $\pi_L^n$-torsion points of $LT$, and we put $L_\infty := \bigcup_n L_n$. The extension $L_\infty/L$ is Galois. We let $\Gamma_L := \Gal(L_\infty/L)$ and $H_L := \Gal(\overline{L}/L_\infty)$. The Lubin-Tate character $\chi_{LT}$ induces an isomorphism $\Gamma_L \xrightarrow{\cong} o_L^\times$.

\section{Coleman power series}\label{sec:Coleman}

We recall the injective ring endomorphism
\begin{align*}
  \varphi_L : o_L[[Z]] & \longrightarrow o_L[[Z]] \\
                     f(Z) & \longmapsto f([\pi_L](Z)) \ .
\end{align*}
In order to characterize its image we let $LT_1$ denote the group of $\pi_L$-torsion points of $LT$. According to
\cite{Col} Lemma 3 we have
\begin{equation*}
  \im(\varphi_L) = \{f \in o_L[[Z]] : f(Z) = f(a +_{LT} Z)\ \text{for any $a \in LT_1$}\}.
\end{equation*}
This leads to the existence of a unique $o_L$-linear endomorphism $\psi_{Col}$ of $o_L[[Z]]$ such that
\begin{equation*}
  \varphi_L \circ \psi_{Col} (f)(Z) = \sum_{a \in LT_1} f(a +_{LT} Z)  \qquad\text{for any $f \in o_L[[Z]]$}
\end{equation*}
(\cite{Col} Thm.\ 4 and Cor.\ 5) as well as of a unique multiplicative map $\mathscr{N} : o_L[[Z]] \longrightarrow o_L[[Z]]$ such that
\begin{equation*}
  \varphi_L \circ \mathscr{N} (f)(Z) = \prod_{a \in LT_1} f(a +_{LT} Z)  \qquad\text{for any $f \in o_L[[Z]]$}
\end{equation*}
(\cite{Col} Thm.\ 11).

The group $\Gamma_L$ acts continuously on $o_L[[Z]]$ via
\begin{align}\label{f:Gamma-action}
  \Gamma_L \times o_L[[Z]] & \longrightarrow o_L[[Z]] \\
               (\gamma, f) & \longmapsto f([\chi_{LT}(\gamma)](Z))
\end{align}
(\cite{Col} Thm.\ 1).

\begin{remark}\phantomsection\label{Z-inverse}
\begin{itemize}
  \item[i.] $\psi_{Col} \circ \varphi_L = q$.
  \item[ii.] $\psi_{Col}([\pi_L] \cdot f) = Z \psi_{Col}(f)$ for any $f \in o_L[[Z]]$.
  \item[iii.] $\mathcal{N}([\pi_L]) = Z^q$.
\end{itemize}
\end{remark}
\begin{proof}
Because of the injectivity of $\varphi_L$ it suffices in all three cases to verify the asserted identity after applying $\varphi_L$. i. We compute
\begin{align*}
  \varphi_L \circ \psi_{Col} \circ \varphi_L (f) & = \sum_{a \in LT_1} (\varphi_L f)(a +_{LT} Z) = \sum_{a \in LT_1} f([\pi_L](a +_{LT} Z)) \\
   & = \sum_{a \in LT_1} f([\pi_L](Z)) = \varphi_L(qf) .
\end{align*}
ii. We compute
\begin{align*}
  (\varphi_L \circ \psi_{Col})([\pi_L]f) & = \sum_{a \in LT_1} [\pi_L](a +_{LT} Z) f(a +_{LT} Z) \\
   & = [\pi_L](Z) \sum_{a \in LT_1} f(a +_{LT} Z) = \varphi_L(Z) (\varphi_L \circ \psi_{Col})(f) \\
   & = \varphi_L(Z \psi_{Col}(f)) \ .
\end{align*}
ii. We omit the entirely analogous computation.
\end{proof}

We observe that for any $f \in o_L((Z)) = o_L[[Z]][Z^{-1}]$ there is an $n(f) \geq 1$ such that $[\pi_L]^{n(f)} \cdot f \in o_L[[Z]]$. The above remark therefore allows to extend $\psi_{Col}$ to an $o_L$-linear endomorphism
\begin{align*}
  \psi_{Col} : o_L((Z)) & \longrightarrow o_L((Z)) \\
   f & \longmapsto Z^{-n(f)} \psi_{Col}([\pi_L]^{n(f)} f)
\end{align*}
and to extend $\mathcal{N}$ to a multiplicative map
\begin{align*}
  \mathcal{N} : o_L((Z)) & \longrightarrow o_L((Z)) \\
   f & \longmapsto Z^{-qn(f)} \mathcal{N}([\pi_L]^{n(f)} f) \ .
\end{align*}

We choose an $o_L$-generator $t_0$ of $T$. This is a sequence of elements $t_{0,n} \in \pi_{L_n} o_{L_n}$ such that $[\pi_L](t_{0,n+1}) = t_{0,n}$ for $n \geq 1$, $[\pi_L](t_{0,1}) = 0$, and $t_{0,1}\neq 0$.

\begin{theorem}[Coleman]\label{Coleman}
For any norm-coherent sequence $u = (u_n)_n \in \varprojlim_n L_n^\times$ there is a unique Laurent series $g_{u,t_0} \in (o_L((Z))^\times)^{\mathcal{N} = 1}$ such that $g_{u,t_0}(t_{0,n}) = u_n$ for any $n \geq 1$. This defines a multiplicative isomorphism
\begin{align*}
  \varprojlim_n L_n^\times & \xrightarrow{\; \cong \;} (o_L((Z))^\times)^{\mathcal{N} = 1} \\
  u & \longmapsto g_{u,t_0} \ .
\end{align*}
\end{theorem}
\begin{proof}
See \cite{Col} Thm.\ A and Cor.\ 17.
\end{proof}

\begin{remark}\phantomsection\label{change-generator}
\begin{itemize}
  \item[i.] The map $(o_L((Z))^\times)^{\mathcal{N} = 1} \xrightarrow{\cong} k_L((Z))^\times$ given by reduction modulo $\pi_L$ is an isomorphism; hence
\begin{align*}
  \varprojlim_n L_n^\times & \xrightarrow{\;\cong\;} k_L((Z))^\times \\
  u & \longmapsto g_{u,t_0} \bmod \pi_L
\end{align*}
      is an isomorphism of groups.
  \item[ii.] If $t_1 = c t_0$ is a second $o_L$-generator of $T$ then $g_{u,t_1}([c](Z)) = g_{u,t_0}(Z)$ for any $u \in \varprojlim_n L_n^\times$.
\end{itemize}
\end{remark}
\begin{proof}
i. \cite{Col} Cor.\ 18. ii. This is immediate from the characterizing property of $g_{u,t_0}$ in the theorem.
\end{proof}

We now introduce the ``logarithmic'' homomorphism
\begin{align*}
  \Delta_{LT} : o_L[[Z]]^\times & \longrightarrow o_L[[Z]] \\
                  f & \longmapsto \frac{\partial_{\mathrm{inv}}(f)}{f} = g_{LT}^{-1} \frac{f'}{f} \ ,
\end{align*}
whose kernel is $o_L^\times$.

\begin{lemma}\phantomsection\label{Delta-N}
\begin{itemize}
  \item[i.]  $\Delta_{LT} \circ \varphi_L = \pi_L \varphi_L \circ \Delta_{LT}$.
  \item[ii.] $\psi_{Col} \circ \Delta_{LT} = \pi_L \Delta_{LT} \circ \mathcal{N}$.
\end{itemize}
\end{lemma}
\begin{proof}
We begin with a few preliminary observations. From \eqref{f:dlog} we deduce
\begin{equation}\label{f:pre1}
  g_{LT} = \frac{[\pi_L]'}{\pi_L} \varphi_L(g_{LT}) \ .
\end{equation}
Secondly we have
\begin{equation}\label{f:pre2}
  \tfrac{d}{dZ} \varphi_L(f(Z)) = \tfrac{d}{dZ}f([\pi_L](Z)) = f'([\pi_L](Z)) [\pi_L]'(Z) = \varphi_L(f') [\pi_L]'
\end{equation}
for any $f \in o_L[[Z]]$. Finally, the fact that $g_{LT}(Z)dZ$ is an invariant differential form implies that
\begin{equation}\label{f:pre3}
  g_{LT} = \tfrac{d}{dZ} \log_{LT} (a +_{LT} Z) \qquad\text{for any $a \in LT_1$}.
\end{equation}
For i. we now compute
\begin{equation*}
  \Delta_{LT} \circ \varphi_L (f) = \tfrac{1}{g_{LT}} \tfrac{\tfrac{d}{dZ} \varphi_L(f(Z))}{\varphi_L(f)} = \tfrac{[\pi_L]'}{g_{LT}} \tfrac{\varphi_L(f')}{\varphi_L(f)}
   = \tfrac{\pi_L}{\varphi_L(g_{LT})} \tfrac{\varphi_L(f')}{\varphi_L(f)} = \pi_L \varphi_L \circ \Delta_{LT}(f) \ ,
\end{equation*}
where the second, resp.\ the third, identity uses \eqref{f:pre2}, resp.\ \eqref{f:pre1}. For ii. we compute
\begin{align*}
  \varphi_L \circ \psi_{Col} \circ \Delta_{LT}(f) & = \sum_{a \in LT_1} \tfrac{1}{g_{LT}(a +_{LT} Z)} \tfrac{f'}{f}(a +_{LT} Z) \\
   & = \sum_{a \in LT_1} \tfrac{1}{g_{LT}(a +_{LT} Z)} \tfrac{\tfrac{d}{dZ} f(a +_{LT} Z)}{f(a +_{LT} Z)} \tfrac{1}{\tfrac{d}{dZ} (a +_{LT} Z)}  \\
   & =  \sum_{a \in LT_1} \tfrac{1}{\tfrac{d}{dZ} \log_{LT}(a +_{LT} Z)} \tfrac{\tfrac{d}{dZ} f(a +_{LT} Z)}{f(a +_{LT} Z)} \\
   & = \sum_{a \in LT_1} \Delta_{LT}(f(a +_{LT} Z)) = \Delta_{LT} (\prod_{a \in LT_1} f(a +_{LT} Z)) \\
   & = \Delta_{LT} \circ \varphi_L \circ \mathcal{N}(f) = \pi_L \varphi_L \circ \Delta_{LT} \circ \mathcal{N}(f) \\
   & = \varphi_L(\pi_L \Delta_{LT} \circ \mathcal{N}(f)) \ ,
\end{align*}
where the fourth, resp. the seventh, identity uses \eqref{f:pre3}, resp. part i. of the assertion.
\end{proof}

It follows that $\Delta_{LT}$ restricts to a homomorphism
\begin{equation*}
  \Delta_{LT} : (o_L[[Z]]^\times)^{\mathcal{N} = 1} \longrightarrow o_L[[Z]]^{\psi_{Col} = \pi_L} \ .
\end{equation*}
Its kernel is the subgroup $\mu_{q-1}(L)$ of $(q-1)$th roots of unity in $o_L^\times$.
%\footnote{Frage f\"ur die Zukunft: What is the cokernel of this map?}

On the other hand $\Delta_{LT}$ obviously extends to the homomorphism
\begin{align*}
  \Delta_{LT} : o_L((Z))^\times & \longrightarrow o_L((Z)) \\
                  f & \longmapsto g_{LT}^{-1} \frac{f'}{f} \ ,
\end{align*}
with the same kernel $o_L^\times$.

\begin{lemma}\label{Delta-N-extended}
The identity $\psi_{Col} \circ \Delta_{LT} = \pi_L \Delta_{LT} \circ \mathcal{N}$ holds true on $o_L((Z))^\times$.
\end{lemma}
\begin{proof}
Let $f \in o_L((Z)^\times$ be any element. It can be written $f = Z^{- n} f_0$ with $f_0 \in o_L[[Z]]^\times$. Then
\begin{equation*}
  \psi_{Col} \circ \Delta_{LT} (f) = - n\psi_{Col}(\frac{1}{Z g_{LT}}) + \psi_{Col} \circ \Delta_{LT} (f_0)
\end{equation*}
and
\begin{equation*}
  \pi_L \Delta_{LT} \circ \mathcal{N} (f) = - n\pi_L \Delta_{LT}(\mathcal{N}(Z)) + \pi_L \Delta_{LT} \circ \mathcal{N} (f_0) \ .
\end{equation*}
The second summands being equal by Lemma \ref{Delta-N}.ii we see that we have to establish that
\begin{equation*}
  \psi_{Col}(\frac{1}{Z g_{LT}}) = \pi_L \Delta_{LT}(\mathcal{N}(Z)) \ .
\end{equation*}
By definition the left hand side is $Z^{-1} \psi_{Col}(\frac{[\pi_L]}{Z g_{LT}})$ and the right hand side is $\frac{\pi_L}{g_{LT}} \frac{\frac{d}{dZ} \mathcal{N}(Z)}{\mathcal{N}(Z)}$. Hence we are reduced to proving the identity
\begin{equation*}
  g_{LT} \mathcal{N}(Z) \psi_{Col}(\frac{[\pi_L]}{Z g_{LT}}) = \pi_L Z \frac{d}{dZ} \mathcal{N}(Z) \ ,
\end{equation*}
which is an identity in $o_L[[Z]]$ and therefore can be checked after applying $\varphi_L$. On the left hand side we obtain
\begin{align*}
  & \varphi_L(g_{LT}) \prod_{a \in LT_1} (a +_{LT} Z) \sum_{b \in LT_1} \frac{[\pi_L](b +_{LT} Z)}{(b +_{LT} Z)g_{LT}(b +_{LT} Z)} \\
  & \qquad
  = \varphi_L(g_{LT}) \varphi_L(Z) \prod_{a \in LT_1} (a +_{LT} Z) \sum_{b \in LT_1} \frac{1}{(b +_{LT} Z)g_{LT}(b +_{LT} Z)}  \\
  & \qquad
  = g_{LT} \frac{\pi_L}{[\pi_L]'} \varphi_L(Z) \prod_{a \in LT_1} (a +_{LT} Z) \sum_{b \in LT_1} \frac{1}{(b +_{LT} Z)g_{LT}(b +_{LT} Z)} \ ,
\end{align*}
where the second equality uses \eqref{f:pre1}. On the right hand side, using \eqref{f:pre2}, we have
\begin{equation*}
  \pi_L \varphi_L(Z) \varphi_L(\tfrac{d}{dZ} \mathcal{N}(Z)) = \frac{\pi_L}{[\pi_L]'} \varphi_L(Z)  \tfrac{d}{dZ} \varphi_L( \mathcal{N}(Z))
   = \frac{\pi_L}{[\pi_L]'} \varphi_L(Z)  \tfrac{d}{dZ} \prod_{a \in LT_1} (a +_{LT} Z) \ .
\end{equation*}
This further reduces us to proving that
\begin{equation*}
  g_{LT} \sum_{b \in LT_1} \frac{1}{(b +_{LT} Z)g_{LT}(b +_{LT} Z)} = \frac{ \frac{d}{dZ} \prod_{a \in LT_1} (a +_{LT} Z)}{\prod_{a \in LT_1} (a +_{LT} Z)} \ .
\end{equation*}
The invariance of $g_{LT}(Z)dZ$ implies
\begin{equation}\label{f:pre4}
  \tfrac{d}{dZ} (a +_{LT} Z) = \frac{g_{LT}(Z)}{g_{LT}(a +_{LT} Z)} \ .
\end{equation}
We see that the above right hand side, indeed, is equal to
\begin{equation*}
  \frac{ \frac{d}{dZ} \prod_{a \in LT_1} (a +_{LT} Z)}{\prod_{a \in LT_1} (a +_{LT} Z)} = \sum_{a \in LT_1} \frac{ \frac{d}{dZ} (a +_{LT} Z)}{a +_{LT} Z} = g_{LT} \sum_{a \in LT_1} \frac{1}{(a +_{LT} Z)g_{LT}(a +_{LT} Z)} \ .
\end{equation*}
\end{proof}

Hence we even have the homomorphism $\Delta_{LT} : (o_L((Z))^\times)^{\mathcal{N} = 1} \longrightarrow o_L((Z))^{\psi_{Col} = \pi_L}$ with kernel $\mu_{q-1}(L)$.

\section{Etale $(\varphi_L,\Gamma_L)$-modules}\label{sec:etale}

We define the ring $\mathscr{A}_L$ to be the $\pi_L$-adic completion of $o_L[[Z]][Z^{-1}]$ and we let $\mathscr{B}_L := \mathscr{A}_L[\pi_L^{-1}]$ denote the field of fractions of $\mathscr{A}_L$. The ring endomorphism $\varphi_L$ of $o_L[[Z]]$ maps $Z$ to $[\pi_L](Z)$. Since $[\pi_L](Z) \equiv Z^q \bmod \pi_L$ the power series $[\pi_L](Z)$ is a unit in $\mathscr{A}_L$. Hence $\varphi_L$ extends to a homomorphism $o_L[[Z]][Z^{-1}] \longrightarrow \mathscr{A}_L$ and then by continuity to a ring endomorphism $\varphi_L$ of $\mathscr{A}_L$ and finally to an embedding of fields $\varphi_L : \mathscr{B}_L \longrightarrow \mathscr{B}_L$. Similarly the invariant derivation $\partial_{\mathrm{inv}}$ first extends algebraically to $o_L[[Z]][Z^{-1}]$, then by continuity to $\mathscr{A}_L$, and finally by linearity to $\mathscr{B}_L$. Evidently we still have \eqref{f:inv} for any $f \in \mathscr{B}_L$.

\begin{remark}\label{basis}
$1, Z, \ldots, Z^{q-1}$ is a basis of $\mathscr{B}_L$ as a $\varphi_L(\mathscr{B}_L)$-vector space.
\end{remark}
\begin{proof}
See \cite{FX} Remark before Lemma 2.1 or \cite{GAL} Prop.\ \ref{GAL-phi-free}.
\end{proof}

This remark allows us to introduce the unique additive endomorphism $\psi_L$ of $\mathscr{B}_L$ which satisfies
\begin{equation*}
  \varphi_L \circ \psi_L = \pi_L^{-1} \cdot trace_{\mathscr{B}_L/\varphi_L(\mathscr{B}_L)} \ .
\end{equation*}
By the injectivity of $\varphi_L$ and the linearity of the field trace we have the projection formula
\begin{equation*}
  \psi_L(\varphi_L(f_1)f_2) = f_1 \psi_L(f_2) \qquad\text{for any $f_i \in \mathscr{B}_L$}.
\end{equation*}
as well as the formula
\begin{equation*}
  \psi_L \circ \varphi_L = \frac{q}{\pi_L} \cdot \id \ .
\end{equation*}

Correspondingly, we consider the unique multiplicative map $\mathcal{N}_L : \mathscr{B}_L \longrightarrow \mathscr{B}_L$ which satisfies
\begin{equation}\label{f:normoperator}
  \varphi_L \circ \mathcal{N}_L  =   \mathrm{Norm}_{\mathscr{B}_L/\varphi_L(\mathscr{B}_L)} \ .
\end{equation}

\begin{remark}\phantomsection\label{psi}
\begin{itemize}
  \item[i.] $\psi_L(\mathscr{A}_L) \subseteq \mathscr{A}_L$ and $\mathcal{N}_L(\mathscr{A}_L) \subseteq \mathscr{A}_L$.
  \item[ii.] On $o_L((Z))$ we have $\psi_L = \pi_L^{-1} \cdot \psi_{Col}$ and $\mathcal{N}_L =\mathcal{N}$.
  \item[iii.] { $\varphi_{L}\circ \psi_{L}\circ \partial_\mathrm{inv}=\partial_\mathrm{inv} \circ \varphi_{L}\circ \psi_{L}$ on $\mathscr{B}_L$.}
  \item[iv.] $\mathcal{N}_L(f)([c](Z)) =  \mathcal{N}_L(f([c](Z)))$ for any $c\in o_L^\times$ and $f\in \mathscr{B}_L$.
     % \footnote{For $c = \pi_L$ we have $\mathcal{N}_L(\varphi_L(f)) = f^q \neq \varphi_L(\mathcal{N}_L(f))$.}
  \item[v.]  $\mathcal{N}_L(f) \equiv f \bmod \pi_L \mathscr{A}_L$ for any $f \in \mathscr{A}_L$.
  \item[vi.]   If $f \in \mathscr{A}_L$ satisfies $f \equiv 1 \bmod \pi_L^m \mathscr{A}_L$ for some $m \geq 1$ then $\mathcal{N}_L(f) \equiv 1 \bmod \pi_L^{m+1} \mathscr{A}_L$.
      \item[vii.]   $(o_L((Z))^\times)^{\mathcal{N}=1} = (\mathscr{A}_L^\times)^{\mathcal{N}_L =1}$.
\end{itemize}
\end{remark}
\begin{proof}
i. The homomorphism $\varphi_L$ induces on $\mathscr{A}_L/\pi_L \mathscr{A}_L = k_L((Z))$ the injective $q$-Frobenius map. It follows that $\varphi_L^{-1}(\mathscr{A}_L) = \mathscr{A}_L$. Hence the assertion reduces to the claim that
\begin{equation}\label{f:tracemod}
  trace_{\mathscr{B}_L/\varphi_L(\mathscr{B}_L)}(\mathscr{A}_L) \subseteq \pi_L \mathscr{A}_L \ .
\end{equation}
But the trace map $trace_{\mathscr{B}_L/\varphi_L(\mathscr{B}_L)}$ induces the trace map for the purely inseparable extension $k_L((Z))/k_L((Z^q))$, which is the zero map.

ii. For any $a \in LT_1$ we have the ring homomorphism
\begin{align*}
  \sigma_a : o_L[[Z]] & \longrightarrow o_{L_1}[[Z]] \subseteq \mathscr{A}_{L_1} \\
                 f(Z) & \mapsto f(a +_{LT} Z) \ .
\end{align*}
Since $\sigma_a(Z) = a +_{LT} Z \equiv a + Z \bmod \deg 2$ we have $\sigma_a(Z) \in \mathscr{A}_{L_1}^\times$, so that $\sigma_a$ extends to $o_L[[Z]][Z^{-1}]$. By continuity $\sigma_a$ further extends to $\mathscr{A}_L$ and then by linearity to an embedding of fields
\begin{equation*}
  \sigma_a : \mathscr{B}_L \longrightarrow \mathscr{B}_{L_1} = \mathscr{B}_L L_1 \ .
\end{equation*}
Clearly these $\sigma_a$ are pairwise different. Moreover, for any $f \in o_L[[Z]]$, we have
\begin{equation*}
  \sigma_a \circ \varphi_L(f) (Z) = f([\pi_L](a +_{LT} Z)) = f(Z) \ .
\end{equation*}
We conclude, by continuity, that $\sigma_a | \varphi_L(\mathscr{B}_L) = \id$. It follows that $\prod_{a \in LT_1} (X - \sigma_a(f))$, for any $f \in \mathscr{B}_L$, is the characteristic polynomial of $f$ over $\varphi_L(\mathscr{B}_L)$. Hence
\begin{equation}\label{f:trace}
    trace_{\mathscr{B}_L/\varphi_L(\mathscr{B}_L)}(f) = \sum_{a \in LT_1} \sigma_a(f) \ ,
\end{equation}
which proves the assertion for $f \in o_L[[Z]]$. For general $f \in o_L((Z))$, using the notation and definition before Thm.\ \ref{Coleman} we compute
\begin{align*}
  \varphi_L \circ \psi_{Col}(f) & = \varphi_L\left( Z^{-n(f)} \psi_{Col}( [\pi_L]^{n(f)}f) \right)  \\
   & = \varphi_L(Z)^{-n(f)} \sum_{a \in LT_1} \sigma_a([\pi_L]^{n(f)}f)  \\
   & = \sum_{a \in LT_1} \sigma_a(f) = trace_{\mathscr{B}_L/\varphi_L(\mathscr{B}_L)}(f) \\
   & = \varphi_L \circ \pi_L \psi_L(f) \ .
\end{align*}
The proof for $\mathcal{N}_L$ is completely analogous.

iii. By the invariance of $\partial_\mathrm{inv}$ we have $ \partial_\mathrm{inv}\circ \sigma_a =\sigma_a \circ \partial_\mathrm{inv}$ on $o_L[[Z]][Z^{-1}]$, whence on $\mathscr{B}_L$ by continuity and linearity. Therefore, the claim follows from \eqref{f:trace}.

 iv. We compute
\begin{align*}
  \varphi_L(\mathcal{N}_L(f)([c](Z))) & = \mathcal{N}_L(f)([c]([\pi_L](Z))) = \mathcal{N}_L(f)([\pi_L]([c](Z))) \\
  & = \varphi_L(\mathcal{N}_L(f))([c](Z)) =  \prod_{a \in LT_1} \sigma_a(f)([c](Z)) \\
  & = \prod_{a \in LT_1} \sigma_{[c^{-1}](a)}(f([c](Z))) = \prod_{a \in LT_1} \sigma_a(f([c](Z))) \\
  & = \varphi_L(\mathcal{N}_L(f([c](Z)))) \ .
\end{align*}

v. We have
\begin{align*}
  \varphi_L \circ \mathcal{N}_L(f) \bmod \pi_L \mathscr{A}_L & = \mathrm{Norm}_{k_L((Z))/k_L((Z^q))} (f \bmod \pi_L \mathscr{A}_L) \equiv f^q \bmod \pi_L \mathscr{A}_L \\
   & = \varphi_L(f) \bmod \pi_L \mathscr{A}_L \ .
\end{align*}

vi. Let $f = 1 + \pi_L^m g$ with $g \in \mathscr{A}_L$. We compute
\begin{align*}
  \varphi_L(\mathcal{N}_L(1 + \pi_L^m g)) & = \prod_{a \in LT_1} 1 + \pi_L^m \sigma_a(g) \equiv 1 + \pi_L^m(\sum_{a \in LT_1} \sigma_a(g)) \bmod \pi_L^{m+1} \mathscr{A}_L  \\
   & \equiv 1 \bmod \pi_L^{m+1} \mathscr{A}_L  \equiv \varphi_L(1) \bmod \pi_L^{m+1} \mathscr{A}_L \
\end{align*} where the third identity uses \eqref{f:tracemod}.
The assertion follows since $\varphi_L$ remains injective modulo $\pi_L^j$ for any $j \geq 1$.

vii. We have the commutative diagram
\begin{equation*}
  \xymatrix{
  (o_L((Z))^\times)^{\mathcal{N}=1} \ar[rr]^{\subseteq} \ar[dr]_{\cong}
                &  &    (\mathscr{A}_L^\times)^{\mathcal{N}_L =1} \ar[dl]    \\
                & k_L((Z))^\times                 }
\end{equation*}
where the oblique arrows are given by reduction modulo $\pi_L$. The left one is an isomorphism by Remark \ref{change-generator}.i. The right one is injective as a consequence of the assertion vi. Hence all three maps must be bijective.
\end{proof}

  Due to Remark \ref{psi}.vii we may view the Coleman isomorphism in Thm.\ \ref{Coleman} as an isomorphism
\begin{equation}\label{f:Coleman}
  \varprojlim_n L_n^\times \xrightarrow{\; \cong \;} (\mathscr{A}_L^\times)^{\mathcal{N}_L =1} \ .
\end{equation}

We always equip $\mathscr{A}_L$ with the weak topology, for which the $o_L$-submodules $\pi_L^m\mathscr{A}_L + Z^m o_L[[Z]]$, for $m \geq 1$, form a fundamental system of open neighbourhoods of zero. The weak topology on any finitely generated $\mathscr{A}_L$-module $M$ is defined to be the quotient topology, with respect to any surjective homomorphism $\mathscr{A}_L^n \twoheadrightarrow M$, of the product topology on $\mathscr{A}_L^n$; this is independent of the choice of this homomorphism. We have the following properties (cf.\ \cite{SV} Lemmas 8.2 and 8.22 for a detailed discussion of weak topologies):
\begin{itemize}
  \item[--] $\mathscr{A}_L$ is a complete Hausdorff topological $o_L$-algebra (with jointly continuous multiplication).
  \item[--] $\mathscr{A}_L$ induces on $o_L[[Z]]$ its compact topology.
  \item[--] $M$ is a complete Hausdorff topological module (with jointly continuous scalar multiplication).
  \item[--] $M/\pi_L^m M$, for any $m \geq 1$, is locally compact.
\end{itemize}

\begin{remark}\label{phi-weakcont}
The endomorphisms $\varphi_L$ and $\psi_L$ of $\mathscr{A}_L$ are continuous for the weak topology.
\end{remark}
\begin{proof}
For $\varphi_L$ see \cite{GAL} Prop.\ \ref{GAL-weak-phi-Gamma}.i. For $\psi_L$ see \cite{FX} Prop.\ 2.4(b) (note that their $\psi$ is our $\frac{\pi_L}{q} \psi_L$).
%\footnote{They state the continuity of their $\psi$ on $\mathscr{B}_L$ without ever defining the weak topology on $\mathscr{B}_L$.}.
\end{proof}

Let $\Omega^1 = \Omega^1_{\mathscr{A}_L} = \mathscr{A}_L dZ$ denote the free rank one $\mathscr{A}_L$-module of differential forms. Obviously the residue map
\begin{align*}
  \mathrm{Res} : \qquad\qquad\  \Omega^1 & \longrightarrow o_L \\
       (\sum_i a_i Z^i)dZ & \longmapsto a_{-1}
\end{align*}
is continuous. Later on in section \ref{sec:Witt} it will be a very important fact that this map does not depend on the choice of the variable $Z$. For the convenience of the reader we explain the argument (cf.\ \cite{Fo1} A2.2.3). First of all we have to extend the maps $d$ and $\mathrm{Res}$ by linearity to maps
\begin{equation*}
  \mathscr{B}_L \xrightarrow{\; d \;} \Omega^1_{\mathscr{B}_L} := L \otimes_{o_L} \Omega^1_{\mathscr{A}_L} \xrightarrow{\mathrm{Res}} L \ .
\end{equation*}
Only for the purposes of the subsequent remark we topologize $\mathscr{B}_L$ by taking as a fundamental system of open neighbourhoods of zero the $o_L[[Z]]$-submodules
\begin{equation*}
  \pi_L^m \mathscr{A}_L + L \otimes_{o_L} Z^m o_L[[Z]])  \qquad\text{for $m \geq 1$}.
\end{equation*}
Using the isomorphism $\Omega^1_{\mathscr{B}_L} = \mathscr{B}_L dZ \cong \mathscr{B}_L$ we also make $\Omega^1_{\mathscr{B}_L}$ into a topological $o_L$-module. It is easy to see that the maps $d$ and $\mathrm{Res}$ are continuous.

\begin{remark}\phantomsection\label{Res-variable}
\begin{itemize}
  \item[i.] $d(\mathscr{B}_L)$ is dense in $\ker(\mathrm{Res})$.
  \item[ii.] $\mathrm{Res} = \mathrm{Res}_Z$ does not depend on the choice of the variable $Z$, i.e., if $Z'$ is any element in $\mathscr{A}_L$ whose reduction modulo $\pi_L$ is a uniformizing element in $k((Z))$, then $\mathrm{Res}_Z(\omega) = \mathrm{Res}_{Z'}(\omega)$ for all $\omega \in \Omega^1_{\mathscr{B}_L}$.
\end{itemize}
 \end{remark}
\begin{proof}
i. On the one hand $L[Z,Z^{-1}] \cap \ker(\mathrm{Res})$ is dense in $\ker(\mathrm{Res})$. On the other hand we have $L[Z,Z^{-1}] \cap \ker(\mathrm{Res}) \subseteq d(\mathscr{B}_L)$. ii. As a consequence of i. both maps $\mathrm{Res}_Z$ and $\mathrm{Res}_{Z'}$ have the same kernel. It therefore suffices to show that $\mathrm{Res}_Z(\frac{dZ'}{Z'}) = 1$. We have $Z' = cZ \eta (1+\pi_L\alpha)$ with $c \in o_L^\times$, $\eta \in 1 +Zo_L[[Z]]$, and $\alpha \in \mathscr{A}_L$. Hence
\begin{equation*}
  \tfrac{dZ'}{Z'} = \tfrac{dZ}{Z} + \tfrac{d\eta}{\eta} + \tfrac{d(1+\pi_L\alpha)}{1+\pi_L\alpha} \ .
\end{equation*}
Clearly $\mathrm{Res}_Z(\frac{d\eta}{\eta}) = 0$. Furthermore, if $m \geq 1$ is sufficiently big, then $\log(1 + \pi_L^m \beta)$, for any $\beta \in \mathscr{A}_L$, converges in $\mathscr{A}_L$. Since $(1 + \pi_L^j \mathscr{A}_L)/(1 + \pi_L^{j+1} \mathscr{A}_L) \cong \mathscr{A}_L / \pi_L \mathscr{A}_L$, for any $j \geq 1$, we have $(1+\pi_L\alpha)^{p^m} = 1+\pi_L^m \beta$ for some $\beta \in \mathscr{A}_L$. It follows that $p^m\frac{d(1+\pi_L\alpha)}{1+\pi_L\alpha} = \frac{d(1+\pi_L^m \beta)}{1+\pi_L^m \beta} =  d(\log (1 + \pi_L^m \beta))$ and therefore that $\mathrm{Res}_Z(\frac{d(1+\pi_L\alpha)}{1+\pi_L\alpha}) = 0$.
\end{proof}

Since $\Omega^1$ is a topological $\mathscr{A}_L$-module it follows that the residue pairing
\begin{align}\label{f:residue-pairing}
  \mathscr{A}_L \times \Omega^1 & \longrightarrow o_L \\
                     (f,\omega) & \longmapsto \mathrm{Res}(f\omega)    \nonumber
\end{align}
is jointly continuous. It induces, for any $m \geq 1$, the continuous pairing
\begin{align*}
  \mathscr{A}_L/\pi_L^m \mathscr{A}_L \times \Omega^1/\pi_L^m \Omega^1 & \longrightarrow L/o_L \\
                     (f,\omega) & \longmapsto \pi_L^{-m}\mathrm{Res}(f\omega) \bmod o_L
\end{align*}
and hence (cf.\ \cite{B-TG} X.28 Thm.\ 3) the continuous $o_L$-linear map
\begin{align}\label{f:dual}
   \Omega^1/\pi_L^m \Omega^1 & \longrightarrow \Hom_{o_L}^c(\mathscr{A}_L/\pi_L^m \mathscr{A}_L, L/o_L) \\
                     \omega & \longmapsto [f \mapsto \pi_L^{-m}\mathrm{Res}(f\omega) \bmod o_L] \ ,  \nonumber
\end{align}
where $\Hom_{o_L}^c$ denotes the module of continuous $o_L$-linear maps equipped with the compact-open topology. For the convenience of the reader we recall the following well known fact.

\begin{lemma}\label{dual}
The map \eqref{f:dual} is an isomorphism of topological $o_L$-modules.
\end{lemma}
\begin{proof}
Let $R := o_L/\pi_L^m o_L$. It is convenient to view the map in question as the map
\begin{align*}
   R((Z))dZ & \longrightarrow \Hom_R^c(R((Z)), R) \\
                     \omega & \longmapsto \ell_\omega(f) := \mathrm{Res}(f\omega) \ .
\end{align*}
One easily checks that $\omega = \sum_i \ell_\omega(Z^{-i-1})Z^i dZ$. Hence injectivity is clear. If $\ell$ is an arbitrary element in the right hand side we put $\omega := \sum_i \ell(Z^{-i-1})Z^i dZ$. The continuity of $\ell$ guarantees that $\ell(Z^i) = 0$ for any sufficiently big $i$. Hence $\omega$ is a well defined preimage of $\ell$ in the left hand side. Finally, the map is open since
\begin{equation*}
  \{f \in R((Z)) : \mathrm{Res} (fZ^n R[[Z]]dZ) = 0 \} = Z^{-n} R[[Z]]
\end{equation*}
is compact for any $n \geq 1$.
\end{proof}

For an arbitrary $\mathscr{A}_L$-module $N$ we have the adjunction isomorphism
\begin{align}\label{f:adjunction}
  \Hom_{\mathscr{A}_L}(N,\Hom_{o_L}(\mathscr{A}_L,L/o_L)) & \xrightarrow{\;\cong\;} \Hom_{o_L}(N,L/o_L) \\
                                                        F & \longmapsto F(.)(1) \ .    \nonumber
\end{align}

\begin{lemma}\label{top-adjunction}
For any finitely generated $\mathscr{A}_L/\pi_L^m \mathscr{A}_L$-module $M$ the adjunction \eqref{f:adjunction} together with \eqref{f:dual} induces the topological isomorphism
\begin{align*}
  \Hom_{\mathscr{A}_L}(M,\Omega^1/\pi_L^m \Omega^1) & \xrightarrow{\;\cong\;} \Hom_{o_L}^c(M,L/o_L) \\
                                                        F & \longmapsto \pi_L^{-m} \mathrm{Res}(F(.)) \bmod o_L \ .
\end{align*}
\end{lemma}
\begin{proof}
It is clear that \eqref{f:adjunction} restricts to an injective homomorphism
\begin{equation}\label{f:top-adjunction}
  \Hom_{\mathscr{A}_L}(M,\Hom_{o_L}^c(\mathscr{A}_L/\pi_L^m \mathscr{A}_L,L/o_L)) \longrightarrow \Hom_{o_L}^c(M,L/o_L) \ .
\end{equation}
The inverse of \eqref{f:adjunction} sends $\ell \in \Hom_{o_L}(M,L/o_L)$ to $F(m)(f) := \ell(fm)$ and visibly restricts to an inverse of \eqref{f:top-adjunction}. By inserting \eqref{f:dual} we obtain the asserted algebraic isomorphism. To check that it also is a homeomorphism we first clarify that on the left hand side we consider the weak topology of $\Hom_{\mathscr{A}_L}(M,\Omega^1/\pi_L^m \Omega^1)$ as a finitely generated $\mathscr{A}_L$-module. The elementary divisor theorem for the discrete valuation ring $\mathscr{A}_L$ implies that $M$ is isomorphic to a finite direct product of modules of the form $\mathscr{A}_L/\pi_L^n \mathscr{A}_L$ with $1 \leq n \leq m$. It therefore suffices to consider the case $M = \mathscr{A}_L/\pi_L^n \mathscr{A}_L$. We then have the commutative diagram of isomorphisms
\begin{equation*}
  \xymatrix@R=0.5cm{
  \Hom_{\mathscr{A}_L}(\mathscr{A}_L/\pi_L^n \mathscr{A}_L,\Omega^1/\pi_L^m \Omega^1) \ar[d]_{=} \ar[dr]^{}             \\
      \pi_L^{m-n} \Omega^1/\pi_L^m \Omega^1  &  \Hom_{o_L}^c(\mathscr{A}_L/\pi_L^n \mathscr{A}_L,L/o_L)        \\
  \Omega^1/\pi_L^n \Omega^1 \ar[u]^{\pi_L^{m-n} \cdot} \ar[ur]_-{\eqref{f:dual}}                 }
\end{equation*}
By Lemma \ref{dual} all maps in this diagram except possibly the upper oblique arrow, which is the map in the assertion, are homeomorphisms. Hence the oblique arrow must be a homeomorphism as well.
\end{proof}

The $\Gamma_L$-action \eqref{f:Gamma-action} on $o_L[[Z]]$ extends, by the same formula, to a $\Gamma_L$-action on $\mathscr{A}_L$ which, moreover, is continuous for the weak topology (see \cite{GAL} Prop.\ \ref{GAL-weak-phi-Gamma}.ii).

\begin{definition}\label{def:modR}
A $(\varphi_L,\Gamma_L)$-module $M$ (over $\mathscr{A}_L$) is a finitely generated $\mathscr{A}_L$-module $M$ together with
\begin{itemize}
  \item[--] a $\Gamma_L$-action on $M$ by semilinear automorphisms which is continuous for the weak topology
\footnote{In case $L = \mathbb{Q}_p$ we have automatic continuity. The simplest instance of this is the fact that any abstract group homomorphism $\Gamma_{\mathbb{Q}_p} \longrightarrow \mathbb{Z}_p^\times$ is continuous. By restricting to sufficiently small open subgroups this reduces to the claim that any abstract group homomorphism $\mathbb{Z}_p \longrightarrow \mathbb{Z}_p$ is continuous, i.e., is determined by its value in $1$. This follows from the triviality of any group homomorphism $\mathbb{Z}_p/\mathbb{Z} \longrightarrow \mathbb{Z}_p$. The latter holds because, by the surjectivity of the projection map $\mathbb{Z} \longrightarrow \mathbb{Z}_p/p \mathbb{Z}_p$, the group $\mathbb{Z}_p/\mathbb{Z}$ is $p$-divisible.}, and
  \item[--] a $\varphi_L$-linear endomorphism $\varphi_M$ of $M$ which commutes with the $\Gamma_L$-action.
\end{itemize}
It is called etale if the linearized map
\begin{align*}
     \varphi_M^{lin} : \mathscr{A}_L \otimes_{\mathscr{A}_L,\varphi_L} M & \xrightarrow{\; \cong \;} M \\
                                           f \otimes m & \longmapsto f \varphi_M (m)
\end{align*}
is bijective. We let $\mathfrak{M}^{et}(\mathscr{A}_L)$ denote the category of etale $(\varphi_L,\Gamma_L)$-modules $M$ over $\mathscr{A}_L$.
\end{definition}

\begin{remark}\label{semilinear-cont}
Let $\alpha : \mathscr{A}_L \longrightarrow \mathscr{A}_L$ be a continuous ring homomorphism, and let $\beta : M \longrightarrow M$ be any $\alpha$-linear endomorphism of a finitely generated $\mathscr{A}_L$-module $M$; then $\beta$ is continuous for the weak topology on $M$.
\end{remark}
\begin{proof}
The map
\begin{align*}
  \beta^{lin} : \mathscr{A}_L \otimes_{\mathscr{A}_L,\alpha} M & \longrightarrow M \\
  f \otimes m & \longmapsto f \beta(m)
\end{align*}
is $\mathscr{A}_L$-linear. We pick a free presentation $\lambda : \mathscr{A}_L^n \twoheadrightarrow M$. Then we find an $\mathscr{A}_L$-linear map $\widetilde{\beta}$ such that the diagram
\begin{equation*}
  \xymatrix{
    \mathscr{A}_L^n \ar@{>>}[d]_{\lambda} \ar[r]^-{\alpha^n} & \mathscr{A}_L^n = \mathscr{A}_L \otimes_{\mathscr{A}_L,\alpha} \mathscr{A}_L^n \ar@{>>}[d]_{\id \otimes \lambda} \ar[r]^-{\widetilde{\beta}} & \mathscr{A}_L^n \ar@{>>}[d]^{\lambda} \\
    M \ar@/_2pc/[rr]^-{\beta} \ar[r]^-{m \mapsto 1 \otimes m} & \mathscr{A}_L \otimes_{\mathscr{A}_L,\alpha} M \ar[r]^-{\beta^{lin}} & M   }
\end{equation*}
is commutative. All maps except possibly the lower left horizontal arrow are continuous. The universal property of the quotient topology then implies that $\beta$ must be continuous as well.
\end{proof}

Remarks \ref{phi-weakcont} and \ref{semilinear-cont} imply that the endomorphism $\varphi_M$ of a $(\varphi_L,\Gamma_L)$-module $M$ is continuous.

On any etale $(\varphi_L,\Gamma_L)$-module $M$ we have the $o_L$-linear endomorphism
\begin{align*}
  \psi_M : M \xrightarrow{(\varphi_M^{lin})^{-1}} \mathscr{A}_L \otimes_{\mathscr{A}_L,\varphi_L} M & \longrightarrow M \\
  f \otimes m & \longmapsto \psi_L (f) m \ ,
\end{align*}
which, by construction, satisfies the projection formulas
\begin{equation*}
  \psi_M(\varphi_L(f)m) = f \psi_M(m) \qquad \text{and}\qquad  \psi_M(f\varphi_M(m)) = \psi_L(f) m \ ,
\end{equation*}
for any $f \in \mathscr{A}_L$ and $m \in M$, as well as the formula
\begin{equation*}
  \psi_M \circ \varphi_M = \frac{q}{\pi} \cdot \id_M \ .
\end{equation*}
Remark \ref{phi-weakcont} is easily seen to imply that $\psi_M$ is continuous for the weak topology.

For technical purposes later on we need to adapt part of Colmez's theory of treillis to our situation. We will do  this in the following setting. Let $M$ be a finitely generated $\mathscr{A}_L$-module (always equipped with its weak topology) such that $\pi_L^n M = 0$ for some $n \geq 1$; we also assume that $M$ is equipped with a $\varphi_L$-linear endomorphism $\varphi_M$ which is etale, i.e., such that $\varphi_M^{lin}$ is bijective.

\begin{definition}
A treillis $N$ in $M$ is an $o_L[[Z]]$-submodule $N \subseteq M$ which is compact and such that its image in $M/\pi_L M$ generates this $k_L((Z))$-vector space.
\end{definition}

\begin{remark}\phantomsection\label{treillis}
\begin{itemize}
  \item[i.] If $e_1, \ldots, e_d$ are $\mathscr{A}_L$-generators of $M$ then $o_L[[Z]]e_1 + \ldots + o_L[[Z]]e_d$ is a treillis in $M$.
  \item[ii.] A compact $o_L[[Z]]$-submodule $N$ of $M$ is a treillis if and only if it is open.
  \item[iii.] For any two treillis $N_0 \subseteq N_1$ in $M$ the quotient $N_1/N_0$ is finite; in particular, any intermediate $o_L[[Z]]$-submodule $N_0 \subseteq N \subseteq N_1$ is a treillis as well.
\end{itemize}
\end{remark}
\begin{proof}
Part i. is obvious from the compactness of $o_L[[Z]]$. For ii. and iii. see \cite{Co3} Prop.\ I.1.2(i).
\end{proof}

Following Colmez we define
\begin{equation*}
  M^{++} := \{m \in M : \varphi_M^i(m) \xrightarrow{i \rightarrow \infty} 0 \}.
\end{equation*}
Since $o_L[[Z]]$ is compact it is easily seen that $M^{++}$ is an $o_L[[Z]]$-submodule of $M$. Obviously $M^{++}$ is $\varphi_M$-invariant.

\begin{lemma}\phantomsection\label{M++}
\begin{itemize}
  \item[i.] $M^{++}$ is a treillis.
  \item[ii.] $\varphi_M - 1$ is an automorphism of $M^{++}$.
\end{itemize}
\end{lemma}
\begin{proof}
i. Using Remark \ref{treillis}.i/iii this follows from \cite{Co3} Lemma II.2.3. This lemma is stated and proved there in the cyclotomic situation. But the only property of $\varphi_L$, besides being etale, which is used is that $\varphi_L(Z) \in Z^2 o_L[[Z]] + \pi_L Z o_L[[Z]]$.

ii. Obviously $m=0$ is the only element in $M^{++}$ which satisfies $\varphi_M (m) = m$. Now let $m \in M^{++}$ be an arbitrary element. Since $M$ is complete the series $m' := \sum_i \varphi^i_M(m)$ converges and satisfies $(\varphi_M - 1)(-m') = m$. But $M^{++}$ is open and hence closed in $M$ so that $-m' \in M^{++}$.
\end{proof}

The following lemma is a slight generalization of a result of Fontaine (cf.\ \cite{Her} Prop.\ 2.4.1).

\begin{lemma}\label{top-strict}
On any etale $(\varphi_L,\Gamma_L)$-module $M$ such that $\pi_L^n M = 0$ for some $n \geq 1$ the map $\varphi_M-1$ is open and, in particular, is topologically strict.
\end{lemma}
\begin{proof}
As $M^{++}$, by Lemma \ref{M++}.i and Remark \ref{treillis}.ii, is compact and open in $M$ we first see, using Lemma \ref{M++}.ii, that $\varphi_M - 1$ is a homeomorphism on $M^{++}$ and then that $\varphi_M - 1$ is an open map.
\end{proof}

The category $\mathfrak{M}^{et}(\mathscr{A}_L)$ has an internal Hom-functor. For any two modules $M$ and $N$ in $\mathfrak{M}^{et}(\mathscr{A}_L)$ the $\mathscr{A}_L$-module  $\Hom_{\mathscr{A}_L}(M,N)$ is finitely generated. It is a $(\varphi_L,\Gamma_L)$-module with respect to
\begin{equation*}
  \gamma(\alpha) := \gamma \circ \alpha \circ \gamma^{-1} \qquad\text{and}\qquad \varphi_{\Hom_{\mathscr{A}_L}(M,N)}(\alpha) := \varphi_N^{lin} \circ (\id_{\mathscr{A}_L} \otimes\, \alpha) \circ (\varphi_M^{lin})^{-1}
\end{equation*}
for any $\gamma \in \Gamma_L$ and any $\alpha \in \Hom_{\mathscr{A}_L}(M,N)$. We need to verify that the $\Gamma_L$-action, indeed, is continuous. This is a consequence of the following general facts.

\begin{remark}\label{pointwise-convergence}
For any two finitely generated $\mathscr{A}_L$-modules $M$ and $N$ we have:
\begin{itemize}
  \item[i.] The weak topology on $\Hom_{\mathscr{A}_L}(M,N)$ coincides with the topology of pointwise convergence.
  \item[ii.] The bilinear map
\begin{align*}
  \Hom_{\mathscr{A}_L}(M,N) \times M & \longrightarrow N \\
  (\alpha,m) & \longmapsto \alpha(m)
\end{align*}
is continuous for the weak topology on all three terms.
\end{itemize}
\end{remark}
\begin{proof}
Since any finitely generated module over the discrete valuation ring $\mathscr{A}_L$ is a direct sum of modules of the form $\mathscr{A}_L$ or $\mathscr{A}_L/\pi_L^j \mathscr{A}_L$ for some $j \geq 1$, it suffices to consider the case that $M$ and $N$ both are such cyclic modules. In fact, we may even assume that $M = \mathscr{A}_L$ and $N = \mathscr{A}_L =: \mathscr{A}_L/\pi_L^\infty \mathscr{A}_L$ or $N = \mathscr{A}_L/\pi_L^j \mathscr{A}_L$. We then have the isomorphism of $\mathscr{A}_L$-modules
\begin{align*}
  \mathrm{ev}_1 : \Hom_{\mathscr{A}_L}(\mathscr{A}_L,\mathscr{A}_L/\pi_L^j \mathscr{A}_L) & \xrightarrow{\;\cong\;} \mathscr{A}_L/\pi_L^j \mathscr{A}_L \\
  \alpha & \longmapsto \alpha(1) \ .
\end{align*}
For i. we have to show that this map is a homeomorphism for the topology of pointwise convergence and the weak topology on the left and right term, respectively. The topology of pointwise convergence is generated by the subsets $C(f,U) := \{\alpha \in \Hom_{\mathscr{A}_L}(\mathscr{A}_L,\mathscr{A}_L/\pi_L^j \mathscr{A}_L) : \alpha(f) \in U\}$, where $f \in \mathscr{A}_L$ and where $U$ runs over open subsets $U \subseteq \mathscr{A}_L/\pi_L^j \mathscr{A}_L$ (for the weak topology). For the open subset $U_f := \{n \in \mathscr{A}_L/\pi_L^j \mathscr{A}_L : fn \in U\}$ we have $C(f,U) = C(1,U_f)$. We see that $\mathrm{ev}_1(C(f,U)) = U_f$.

Using the topological isomorphism $\mathrm{ev}_1$ the bilinear map in ii. becomes the multiplication map $\mathscr{A}_L/\pi_L^j \mathscr{A}_L \times \mathscr{A}_L \longrightarrow \mathscr{A}_L/\pi_L^j \mathscr{A}_L$. It is continuous since, as noted earlier, $\mathscr{A}_L$ is a topological algebra for the weak topology.
\end{proof}

Let $(\gamma_i)_{i \in \mathbb{N}}$ in $\Gamma_L$, resp.\ $(\alpha_i)_{i \in \mathbb{N}}$ in $\Hom_{\mathscr{A}_L}(M,N)$, be a sequence which converges to $\gamma \in \Gamma_L$, resp.\ to $\alpha \in \Hom_{\mathscr{A}_L}(M,N)$ for the weak topology. We have to show that the sequence $(\gamma_i(\alpha_i))_i$ converges to $\gamma(\alpha)$ for the weak topology. By Remark \ref{pointwise-convergence}.i it, in fact, suffices to check pointwise convergence. Let therefore $m \in M$ be an arbitrary element. As $\Gamma_L$ acts continuously on $M$, we have $\lim_{i \mapsto \infty} \gamma_i^{-1}(m) = \gamma^{-1}(m)$. The Remark \ref{pointwise-convergence}.ii then implies that $\lim_{i \mapsto \infty} \alpha_i(\gamma_i^{-1}(m)) = \alpha(\gamma^{-1}(m)$. By the continuity of the $\Gamma_L$-action on $N$ we finally obtain that $\lim_{i \mapsto \infty} \gamma_i(\alpha_i(\gamma_i^{-1}(m))) = \gamma(\alpha(\gamma^{-1}(m)))$.

In order to check etaleness we use the linear isomorphisms $\varphi_M^{lin}$ and $\varphi_N^{lin}$ to identify $\Hom_{\mathscr{A}_L}(M,N)$ and $\Hom_{\mathscr{A}_L}(\mathscr{A}_L \otimes_{\mathscr{A}_L,\varphi_L} M, \mathscr{A}_L \otimes_{\mathscr{A}_L,\varphi_L} N) = \Hom_{\mathscr{A}_L}(M, \mathscr{A}_L \otimes_{\mathscr{A}_L,\varphi_L} N)$. Then the linearized map $\varphi_{\Hom_{\mathscr{A}_L}(M,N)}^{lin}$ becomes the map
\begin{align*}
  \mathscr{A}_L \otimes_{\mathscr{A}_L,\varphi_L} \Hom_{\mathscr{A}_L}(M,N) & \longrightarrow \Hom_{\mathscr{A}_L}(M, \mathscr{A}_L \otimes_{\mathscr{A}_L,\varphi_L} N) \\
  f \otimes \alpha & \longmapsto [m \mapsto f \otimes \alpha(m)] \ .
\end{align*}
To see that the latter map is bijective we may use, because of the flatness of $\varphi_L$ as an injective ring homomorphism between discrete valuation rings, a finite presentation of the module $M$ in order to reduce  to the case $M = \mathscr{A}_L$, in which the bijectivity is obvious. Hence $\Hom_{\mathscr{A}_L}(M,N)$ is an etale $(\varphi_L,\Gamma_L)$-module (cf.\ \cite{Fo1} A.1.1.7). One easily checks the validity of the formula
\begin{equation}\label{f:inner}
\varphi_{\Hom_{\mathscr{A}_L}(M,N)}(\alpha)(\varphi_M(m))=\varphi_N( \alpha(m)) \ .
\end{equation}

As a basic example we point out that $\Omega^1$ naturally is an etale $(\varphi_L,\Gamma_L)$-module via
\begin{equation*}
  \gamma(dZ) := [\chi_{LT}(\gamma)]'(Z)dZ = d[\chi_{LT}(\gamma)](Z) \ \text{and}\ \varphi_{\Omega^1}(dZ) := \pi_L^{-1} [\pi_L]'(Z)dZ = \pi_L^{-1} d[\pi_L](Z) \ .
\end{equation*}
Note that the congruence $[\pi_L](Z) \equiv \pi_L Z +Z^q \bmod \pi_L$ indeed implies that the derivative $[\pi_L]'(Z)$ is divisible by $\pi_L$. The simplest way to see that $\Omega^1$ is etale is to identify it with another obviously etale $(\varphi_L,\Gamma_L)$-module.

If $\chi : \Gamma_L \longrightarrow o_L^\times$ is any continuous character with representation module $W_\chi=o_Lt_\chi$ then, for any $M$ in $\mathfrak{M}^{et}(\mathscr{A}_L)$, we have the twisted module $M(\chi)$  in $\mathfrak{M}^{et}(\mathscr{A}_L)$ where $M(\chi) := M\otimes_{o_L}W_\chi$ as $\mathscr{A}_L$-module, $\varphi_{M(\chi)}(m\otimes w) := \varphi_M(m)\otimes w$, and $\gamma | M(\chi)(m\otimes w) :=  \gamma|M (m)\otimes  \gamma|W_\chi(w)=\chi(\gamma) \cdot\gamma|M(m)\otimes w$ for $\gamma \in \Gamma_L$.  It follows that $\psi_{M(\chi)}(m\otimes w) = \psi_M(m)\otimes w$. For the character $\chi_{LT}$ we take $W_{\chi_{LT}}=T=o_L t_0$ and $W_{\chi_{LT}^{-1}}=T^*=o_Lt_0^*$ as representation module, where $T^*$ denotes the $o_L$-dual with dual basis $t_0^*$ of $t_0$.

\begin{lemma}\label{Omega-as-twist}
The map
\begin{align*}
  \mathscr{A}_L(\chi_{LT}) & \xrightarrow{\;\cong\;} \Omega^1 \\
                        f\otimes t_0 & \longmapsto fd\log_{LT} = f g_{LT} dZ
\end{align*}
is an isomorphism of $(\varphi_L,\Gamma_L)$-modules.
\end{lemma}
\begin{proof}
Since $g_{LT}$ is a unit in $o_L[[Z]]$ it is immediately clear that the map under consideration is well defined and bijective. The equivariance follows from \eqref{f:dlog}.
\end{proof}

\begin{remark}\label{psi-invariance}
 For later applications  we want to point out that for $\hat{u} \in   (o_L((Z))^\times)^{\mathcal{N} = 1}$ the differential form $\frac{d\hat{u}}{\hat{u}}$ is $\psi_{\Omega^1}$-invariant: In fact using Lemma \ref{Omega-as-twist}, Remark \ref{psi}.ii, and Lemma \ref{Delta-N-extended} for the second, fourth, and fifth identity, respectively, we compute
\begin{align*}
  \psi_{\Omega^1}(\frac{d\hat{u}}{\hat{u}}) & =  \psi_{\Omega^1}(\Delta_{LT}(\hat{u})d\log_{LT})  =
     \psi_{\Omega^1}(\Delta _{LT}(\hat{u}) \varphi_{\Omega^1}(d\log_{LT})) \\
   & = \psi_L(\Delta _{LT}(\hat{u}))d\log_{LT}
    = \pi_L^{-1} \psi_{Col}(\Delta _{LT}(\hat{u}))d\log_{LT}   \\
   & = \Delta _{LT}(\hat{u})d\log_{LT}
    = \frac{d\hat{u}}{\hat{u}} \ .
\end{align*}
\end{remark}

\begin{lemma}\label{d-phi-psi}
The map $d : \mathscr{A}_L \longrightarrow \Omega^1$ satisfies:
\begin{itemize}
  \item[i.] $\pi_L \cdot \varphi_{\Omega^1} \circ d = d \circ \varphi_L$;
  \item[ii.] $\gamma \circ d = d \circ \gamma$ for any $\gamma \in \Gamma_L$;
  \item[iii.] $\pi_L^{-1} \cdot \psi_{\Omega^1} \circ d = d \circ \psi_L$.
\end{itemize}
\end{lemma}
\begin{proof}
i. For $f \in \mathscr{A}_L$ we compute
\begin{equation*}
  \varphi_{\Omega^1}(df) = \varphi_{\Omega^1}(f'dZ) = \pi_L^{-1} f'([\pi_L](Z))[\pi_L]'(Z) dZ
    = \pi_L^{-1} d(f([\pi_L](Z))) = \pi_L^{-1} d(\varphi_L(f)) \ .
\end{equation*}

ii. The computation is completely analogous to the one for i.

iii. Since $\varphi_{\Omega^1}$ is injective, the asserted identity is equivalent to $\varphi_{\Omega^1}\circ \psi_{\Omega^1} \circ d= d \circ \varphi_{L}\circ \psi_{L}$ by i. Lemma \ref{Omega-as-twist} implies that $(\varphi_{L}\circ \psi_{L}(f)) g_{LT} dZ = \varphi_{\Omega^1}\circ \psi_{\Omega^1} (f g_{LT} dZ)$. Using this, \eqref{f:inv}, and Remark \ref{psi}.iii in the second, first and fourth, and third identity, respectively, we compute
\begin{align*}
  \varphi_{\Omega^1}\circ \psi_{\Omega^1} (df) & = \varphi_{\Omega^1}\circ \psi_{\Omega^1} (\partial_{\mathrm{inv}}(f) g_{LT} dZ) \\
   & = (\varphi_{L}\circ \psi_{L}(\partial_\mathrm{inv}(f))g_{LT}dZ = \partial_\mathrm{inv}(\varphi_{L}\circ \psi_{L}(f))g_{LT}dZ \\
   & = d(\varphi_{L}\circ \psi_{L}(f)) \ .
\end{align*}
\end{proof}

\begin{proposition}\label{Res-phi-psi}
The residue map $\mathrm{Res} : \Omega^1 \longrightarrow L$ satisfies:
\begin{itemize}
  \item[i.] $\mathrm{Res} \circ \varphi_{\Omega^1} = \pi_L^{-1}q \cdot \mathrm{Res}$;
  \item[ii.] $\mathrm{Res} \circ \gamma = \mathrm{Res}$ for any $\gamma \in \Gamma_L$;
  \item[iii.] $\mathrm{Res} \circ \psi_{\Omega^1} = \mathrm{Res}$.
\end{itemize}
\end{proposition}
\begin{proof}
Of course, exact differential forms have zero residue. Let now $\alpha$ denote any of the endomorphisms $\varphi_{\Omega^1}$, $\gamma$, or $\psi_{\Omega^1}$ of $\Omega^1$. Using Lemma \ref{d-phi-psi} we have $(m+1)\mathrm{Res}(\alpha(Z^m dZ)) = \mathrm{Res}(\alpha(d(Z^{m+1}))) = \pi_L^\epsilon \mathrm{Res}(d\alpha(Z^{m+1})) = 0$ with $\epsilon \in \{-1, 0, 1\}$ and hence $\mathrm{Res}(\alpha(Z^m dZ)) = 0$ for any $m \neq -1$. Since $\mathrm{Res}$ is continuous it follows that $\alpha$ preserves the kernel of $\mathrm{Res}$. This reduces us to showing the asserted identities on the differential form $Z^{-1}dZ$. In other words we have to check that $\mathrm{Res}(\alpha(Z^{-1}dZ)) = \pi_L^{-1}q, 1, 1$, respectively, in the three cases.

i. We have $\varphi_{\Omega^1}(Z^{-1}dZ) = \pi_L^{-1}\frac{[\pi_L]'(Z)}{[\pi_L](Z)}dZ$. But $[\pi_L](Z) = Z^q(1 + \pi_L v(Z))$ with $v \in \mathscr{A}_L$. Hence $\varphi_{\Omega^1}(Z^{-1}dZ) = \pi_L^{-1}qZ^{-1}dZ + \pi_L^{-1}\frac{d(1+\pi_L v)}{1+\pi_L v}$. In the proof of Remark \ref{Res-variable}.ii we have seen that $\frac{d(1+\pi_L v)}{1+\pi_L v}$ has zero residue.

ii. Here we have $\gamma(Z^{-1}dZ) = \frac{[\chi_{LT}(\gamma)]'(Z)}{[\chi_{LT}(\gamma)](Z)} dZ$. But $[\chi_{LT}(\gamma)](Z) = Z u(Z)$ with a unit $u \in o_L[[Z]]^\times$. It follows that $\gamma(Z^{-1}dZ) = Z^{-1}dZ + \frac{u'}{u}dZ$. The second summand has zero residue, of course.

iii. The identity in i. implies that $\mathrm{Res} \circ \varphi_{\Omega^1} = \pi_L^{-1}q \cdot \mathrm{Res} = \mathrm{Res} \circ \psi_{\Omega^1} \circ \varphi_{\Omega^1}$. Hence the identity in iii. holds on the image of $\varphi_{\Omega^1}$ (as well as on the kernel of $\mathrm{Res}$). But in the course of the proof of  i.  we have seen that $Z^{-1}dZ \in \im(\varphi_{\Omega^1}) + \ker(\mathrm{Res})$.
\end{proof}

\begin{corollary}\label{phi-psi-dual}
The residue pairing satisfies
\begin{equation*}
  \mathrm{Res}(f\psi_{\Omega^1}(\omega)) = \mathrm{Res}(\varphi_L(f)\omega) \qquad\text{for any $f \in \mathscr{A}_L$ and $\omega \in \Omega^1$}.
\end{equation*}
\end{corollary}
\begin{proof}
By the projection formula the left hand side of the asserted equality is equal to $\mathrm{Res}(\psi_{\Omega^1}(\varphi_L(f)\omega))$. Hence the assertion follows from Prop.\ \ref{Res-phi-psi}.iii.
\end{proof}

For a finitely generated  $\mathscr{A}_L/\pi_L^n \mathscr{A}_L$-module $M$ the isomorphism in Lemma \ref{top-adjunction} induces the following pairing
\begin{align*}
[\;,\;]=[\;,\;]_M: M\times \Hom_{\mathscr{A}_L}(M,\Omega^1/\pi_L^n \Omega^1 ) & \longrightarrow  L/o_L \\
(m,F) & \longmapsto \pi_L^{-n} \mathrm{Res}(F(m)) \bmod o_L \ .
\end{align*}
Since $M$ is locally compact it is (jointly) continuous by \cite{B-TG} X.28 Thm.\ 3. Note that this pairing (and hence also the isomorphism in Lemma \ref{top-adjunction} ) is $\Gamma_L$-invariant by Prop.\ \ref{Res-phi-psi}.ii.

The map $\Hom_{\mathscr{A}_L}(\mathscr{A}_L/\pi_L^n \mathscr{A}_L,\Omega^1/\pi_L^n \Omega^1) \xrightarrow{\cong} \Omega^1/\pi_L^n \Omega^1$ which sends $F$ to $F(1)$ is an isomorphism of (etale) $(\varphi_L,\Gamma_L)$-modules. Cor.\  \ref{phi-psi-dual} then implies that
\begin{equation*}
  [ \varphi_{\mathscr{A}_L/\pi_L^n \mathscr{A}_L}(f),F]_{\mathscr{A}_L/\pi_L^n \mathscr{A}_L}= [f,\psi_{\Hom_{\mathscr{A}_L}(\mathscr{A}_L/\pi_L^n \mathscr{A}_L,\Omega^1/\pi_L^n \Omega^1 )}(F) ]_{\mathscr{A}_L/\pi_L^n \mathscr{A}_L}
\end{equation*}
for all $f\in \mathscr{A}_L/\pi_L^n \mathscr{A}_L$ and $F\in \Hom_{\mathscr{A}_L}(\mathscr{A}_L/\pi_L^n \mathscr{A}_L,\Omega^1/\pi_L^n \Omega^1 ) .$ More generally, we show:

\begin{proposition}\label{pairing]}
Let $M$ be an etale $(\varphi_L,\Gamma_L)$-module such that $\pi_L^n M = 0$ for some $n \geq 1$; we have:
\begin{itemize}
\item[i.] The operator $\psi_M$ is left adjoint to $\varphi_{\Hom_{\mathscr{A}_L}(M,\Omega^1/\pi_L^n \Omega^1 )}$ under the pairing $[\;,\;]$, i.e.,
\begin{equation*}
  [\psi_M(m),F]=[m,\varphi_{\Hom_{\mathscr{A}_L}(M,\Omega^1/\pi_L^n \Omega^1 )}(F)]
\end{equation*}
for all $m \in M$ and all $F \in {\Hom_{\mathscr{A}_L}(M,\Omega^1/\pi_L^n \Omega^1 )}$;
\item[ii.] the operator $\varphi_M$ is left adjoint to $\psi_{\Hom_{\mathscr{A}_L}(M,\Omega^1/\pi_L^n \Omega^1 )}$ under the pairing $[\;,\;]$, i.e.,
\begin{equation*}
  [\varphi_M(m),F]=[m,\psi_{\Hom_{\mathscr{A}_L}(M,\Omega^1/\pi_L^n \Omega^1 )}(F)]
\end{equation*}
for all $m \in M$ and all $F \in {\Hom_{\mathscr{A}_L}(M,\Omega^1/\pi_L^n \Omega^1 )}$.
\end{itemize}
\end{proposition}
\begin{proof}
For notational simplicity we abbreviate the subscript $\Hom_{\mathscr{A}_L}(M,\Omega^1/\pi_L^n \Omega^1)$ to $\Hom$.

i. Since $M$ is etale it suffices to check the asserted identity on elements of the form $f \varphi_M(m)$ with $f \in \mathscr{A}_L$ and $m \in M$. By the projection formula for $\psi_M$ the left hand side then becomes $[\psi_L(f)m,F]$. We compute the right hand side:
\begin{align*}
[f\varphi_M(m), \varphi_{\Hom}(F)] & \equiv \pi_L^{-n} \mathrm{Res}(\varphi_{\Hom}(F)(f\varphi_M(m))) \\
& \equiv \pi_L^{-n} \mathrm{Res} (f \varphi_{\Hom}(F)(\varphi_M(m))) \\
& \equiv \pi_L^{-n} \mathrm{Res} (f \varphi_{\Omega^1/\pi_L^n \Omega^1}(F(m))) \\
& \equiv \pi_L^{-n} \mathrm{Res} (\psi_{\Omega^1/\pi_L^n \Omega^1}(f\varphi_{\Omega^1/\pi_L^n \Omega^1}(F(m)))) \\
& \equiv \pi_L^{-n} \mathrm{Res} (\psi_L(f)F(m)) \\
& \equiv \pi_L^{-n} \mathrm{Res} (F(\psi_L(f)m)) \\
& \equiv [\psi_{L}(f)m,F] \mod o_L \ ;
\end{align*}
here the first and last identities are just the definition of $[\;,\;]$, the second and sixth use $\mathscr{A}_L$-linearity, the third the formula \eqref{f:inner}, the fourth Prop.\ \ref{Res-phi-psi} iii., and the fifth the projection formula for $\psi_{\Omega^1/\pi_L^n \Omega^1}$.

ii. Correspondingly we compute
\begin{align*}
[\varphi_M(m),f\varphi_{\Hom}(F)]
& \equiv \pi_L^{-n} \mathrm{Res} (f\varphi_{\Hom}(F)(\varphi_M(m))) \\
& \equiv \pi_L^{-n} \mathrm{Res} (f\varphi_{\Omega^1/\pi_L^n \Omega^1}(F(m))) \\
& \equiv \pi_L^{-n} \mathrm{Res} (\psi_{\Omega^1/\pi_L^n \Omega^1}(f\varphi_{\Omega^1/\pi_L^n \Omega^1}(F(m)))) \\
& \equiv \pi_L^{-n} \mathrm{Res} (\psi_L(f)F(m)) \\
& \equiv \pi_L^{-n} \mathrm{Res} ((\psi_{\Hom} (f \varphi_{\Hom} (F)))(m)) \\
& \equiv [m,\psi_{\Hom} (f \varphi_{\Hom} (F))] \mod o_L \ .
\end{align*}
\end{proof}

\begin{remark}\label{pairing-rangle}
Similarly, for any etale $(\varphi_L,\Gamma_L)$-module $M$ such that $\pi_L^n M = 0$ one can consider the $\Gamma_L$-invariant (jointly) continuous pairing
\begin{align*}
[\;,\;\rangle=[\;,\;\rangle_M: M \times \Hom_{\mathscr{A}_L}(M,\mathscr{A}_L(\chi_{LT})/\pi_L^n \mathscr{A}_L(\chi_{LT}) ) & \longrightarrow  L/o_L \\
(m,F) & \longmapsto \pi_L^{-n} \mathrm{Res} (F(m)d\log_{LT}) \bmod o_L
\end{align*}
which arises from $[\;,\;]_M$  by plugging in the isomorphism from Lemma \ref{Omega-as-twist}. Clearly, it has adjointness properties analogous to the ones in Prop.\ \ref{pairing]}.
\end{remark}

\section{The Kisin-Ren equivalence}\label{sec:KR}

Let $\widetilde{\mathbf{E}}^+ := \varprojlim o_{\mathbb{C}_p}/p o_{\mathbb{C}_p}$ with the transition maps being given by the Frobenius $\phi(a) = a^p$. We may also identify $\widetilde{\mathbf{E}}^+$ with
$\varprojlim o_{\mathbb{C}_p}/\pi_L o_{\mathbb{C}_p}$ with
the transition maps being given by the $q$-Frobenius
$\phi_q (a) = a^q$. Recall that $\widetilde{\mathbf{E}}^+$ is a complete valuation ring with residue field $\overline{\mathbb{F}_p}$ and its field of fractions $\widetilde{\mathbf{E}} = \varprojlim \mathbb{C}_p$ being algebraically closed of characteristic $p$. Let $\mathfrak{m}_{\widetilde{\mathbf{E}}}$ denote the maximal ideal in $\widetilde{\mathbf{E}}^+$.

The $q$-Frobenius $\phi_q$ first extends by functoriality to the rings of the Witt vectors $W(\widetilde{\mathbf{E}}^+) \subseteq W(\widetilde{\mathbf{E}})$ and then $o_L$-linearly to $W(\widetilde{\mathbf{E}}^+)_L := W(\widetilde{\mathbf{E}}^+) \otimes_{o_{L_0}} o_L \subseteq W(\widetilde{\mathbf{E}})_L := W(\widetilde{\mathbf{E}}) \otimes_{o_{L_0}} o_L$, where $L_0$ is the maximal unramified subextension of $L$. The Galois group $G_L$ obviously acts on $\widetilde{\mathbf{E}}$ and $W(\widetilde{\mathbf{E}})_L$ by automorphisms commuting with $\phi_q$.  This $G_L$-action is continuous for the weak topology on $W(\widetilde{\mathbf{E}})_L$ (cf.\ \cite{GAL} Lemma \ref{GAL-continuous-action}). Let $\mathbb{M}_L$ denote the ideal in $W(\widetilde{\mathbf{E}}^+)_L$ which is the preimage of $\mathfrak{m}_{\widetilde{\mathbf{E}}}$ under the residue class map.

Evaluation of the global coordinate $Z$ of $LT$ at $\pi_L$-power torsion points induces a map (not a homomorphism in the naive sense) $\iota: T \longrightarrow \widetilde{\mathbf{E}}^+$. Namely, if $t =  (z_n)_{n\geq 1} \in T$ with $[\pi_L](z_{n+1}) = z_n$ and $[\pi_L](z_1) = 0$, then $z_{n+1}^q \equiv z_n \bmod \pi_L$ and hence $\iota(t) := (z_n \bmod \pi_L)_n \in \widetilde{\mathbf{E}}^+$.

\begin{lemma}\label{iota}
The image of the map $\iota$ is contained in $\mathfrak{m}_{\widetilde{\mathbf{E}}}$. The map
\begin{align*}
   \iota_{LT} : T & \longrightarrow \mathbb{M}_L \\
    t & \longmapsto \lim_{n \rightarrow \infty} ([\pi_L] \circ \phi_q^{-1})^n([\iota(t)])) \ ,
\end{align*}
where $[\iota(t)]$ denotes the Teichm\"uller representative of $\iota(t)$, is well defined and satisfies:
\begin{itemize}
  \item[a.] $[a](\iota_{LT}(t)) = \iota_{LT}(at)$ for any $a \in o_L$;
  \item[b.] $\phi_q(\iota_{LT}(t)) = \iota_{LT}(\pi_L t) = [\pi_L](\iota_{LT}(t))$;
  \item[c.] $\sigma(\iota_{LT}(t))=[\chi_{LT}(\sigma)](\iota_{LT}(t))$ for any $\sigma \in G_L$.
\end{itemize}
\end{lemma}
\begin{proof}
This is \cite{KR} Lemma 1.2, which refers to \cite{Co1} Lemma 9.3. For full details see \cite{GAL} \S\ref{GAL-sec:coeff}.
\end{proof}

As before we fix an $o_L$-generator $t_0$ of $T$ and put $\omega_{LT} := \iota_{LT}(t_0)$. By sending $Z$ to $\omega_{LT}$ we obtain an embedding of rings
\begin{equation*}
     o_L[[Z]] \longrightarrow W(\widetilde{\mathbf{E}}^+)_L \ .
\end{equation*}
As explained in \cite{KR} (1.3) it extends to embeddings of rings
\begin{equation*}
    \mathscr{A}_L \longrightarrow W(\widetilde{\mathbf{E}})_L  \qquad\text{and}\qquad   \mathscr{B}_L \longrightarrow L \otimes_{o_L} W(\widetilde{\mathbf{E}})_L \ .
\end{equation*}
The left map, in fact, is a topological embedding for the weak topologies on both sides (\cite{GAL} Prop.\ \ref{GAL-properties-j}.i).
The Galois group $G_L$ acts through its quotient $\Gamma_L$ on $\mathscr{B}_L$ by $(\sigma,f) \longmapsto f([\chi_{LT}(\sigma)](Z))$. Then, by Lemma \ref{iota}.c, the above embeddings are $G_L$-equivariant. Moreover, the $q$-Frobenius $\phi_q$ on $L \otimes_{o_L} W(\widetilde{\mathbf{E}})_L$, by Lemma \ref{iota}.b, restricts to the endomorphism $f \longmapsto f \circ [\pi_L]$ of $\mathscr{B}_L$ which we earlier denoted by $\varphi_L$.

We define $\mathbf{A}_L$ to be the image of $\mathscr{A}_L$ in $W(\widetilde{\mathbf{E}})_L$. It is a complete discrete valuation ring with prime element $\pi_L$ and residue field the image $\mathbf{E}_L$ of $k_L ((Z)) \hookrightarrow \widetilde{\mathbf{E}}$. As a consequence of Lemma \ref{iota}.a this subring $\mathbf{A}_L$ is independent of the choice of $t_0$.  As explained above each choice of $t_0$ gives rise to an isomorphism
\begin{equation}\label{f:isoVariousA}
  (\mathscr{A}_L, \varphi_L,\Gamma_L, \text{weak topology}) \xrightarrow{\; \cong \;} (\mathbf{A}_L, \phi_q,\Gamma_L, \text{weak topology})
\end{equation}
between the $o$-algebras $\mathscr{A}_L$ and $\mathbf{A}_L$ together with their additional structures. By literally repeating the Def.\ \ref{def:modR} we have the notion of (etale) $(\phi_q,\Gamma_L)$-modules over $\mathbf{A}_L$ as well as the category $\mathfrak{M}^{et}(\mathbf{A}_L)$. In the same way as for Remark \ref{psi}.i we may define the operator $\psi_L$ on $\mathbf{A}_L$ and then on any etale $(\phi_q,\Gamma_L)$-module over $\mathbf{A}_L$. The above algebra isomorphism gives rise to an equivalence of categories
\begin{equation}\label{f:equCatvariable}
  \mathfrak{M}^{et}(\mathscr{A}_L) \xrightarrow{\; \simeq \;} \mathfrak{M}^{et}(\mathbf{A}_L) \ ,
\end{equation}
which also respects the $\psi_L$-operators.  Using the norm map for the extension $\mathbf{A}_L/\phi_q(\mathbf{A}_L)$ we define, completely analogously as in \eqref{f:normoperator}, a multiplicative norm operator $\mathcal{N} : \mathbf{A}_L \longrightarrow \mathbf{A}_L$. Then,  using also Lemma \ref{iota}.b, $\mathcal{N}_L$ and $\mathcal{N}$ correspond to each other under the isomorphism \eqref{f:isoVariousA}. In particular, \eqref{f:isoVariousA} (for any choice of $t_0$) induces an isomorphism
\begin{equation}\label{f:N-iso}
  (\mathscr{A}_L^\times)^{\mathcal{N}_L =1} \xrightarrow{\;\cong\;} (\mathbf{A}_L^\times)^{\mathcal{N}=1} \ .
\end{equation}

We form the maximal integral unramified extension ($=$ strict Henselization) of $\mathbf{A}_L$  inside $W(\widetilde{\mathbf{E}})_L$. Its $p$-adic completion $\mathbf{A}$ still is contained in $W(\widetilde{\mathbf{E}})_L$. Note that $\mathbf{A}$ is a complete discrete valuation ring with prime element $\pi_L$ and residue field the separable algebraic closure $\mathbf{E}_L^{sep}$ of $\mathbf{E}_L$ in $\widetilde{\mathbf{E}}$.  By the functoriality properties of strict Henselizations the $q$-Frobenius $\phi_q$ preserves $\mathbf{A}$. According to \cite{KR} Lemma 1.4 the $G_L$-action on $W(\widetilde{\mathbf{E}})_L$ respects $\mathbf{A}$ and induces an isomorphism $H_L = \ker(\chi_{LT}) \xrightarrow{\cong} \Aut^{cont}(\mathbf{A}/\mathbf{A}_L)$.

Let $\Rep_{o_L}(G_L)$ denote the abelian category of finitely generated $o_L$-modules equipped with a continuous linear $G_L$-action. The following result is established in \cite{KR} Thm.\ 1.6.

\begin{theorem}\label{KR-equiv}
The functors
\begin{equation*}
    V \longmapsto D_{LT}(V) := (\mathbf{A} \otimes_{o_L} V)^{\ker(\chi_{LT})} \qquad\text{and}\qquad M \longmapsto (\mathbf{A} \otimes_{\mathbf{A}_L} M)^{\phi_q \otimes \varphi_M =1}
\end{equation*}
are exact quasi-inverse equivalences of categories between $\Rep_{o_L}(G_L)$ and $\mathfrak{M}^{et}(\mathbf{A}_L)$.
\end{theorem}

 For the convenience of the reader we discuss a few properties, which will be used later on, of the functors in the above theorem.

First of all we recall that the tensor product $M \otimes_{o_L} N$ of two linear-topological $o$-modules $M$ and $N$ is equipped with the linear topology for which the $o_L$-submodules
\begin{equation*}
  \im(U_M \otimes_{o_L} N \rightarrow M \otimes_{o_L} N) + \im(M \otimes_{o_L} U_N \rightarrow M \otimes_{o_L} N) \subseteq M \otimes_{o_L} N \ ,
\end{equation*}
where $U_M$ and $U_N$ run over the open submodules of $M$ and $N$, respectively, form a fundamental system of open neighbourhoods of zero. One checks that, if a profinite group $H$ acts linearly and continuously on $M$ and $N$, then its diagonal action on $M \otimes_{o_L} N$ is continuous as well.

In our situation we consider $M = \mathbf{A}$ with its weak topology induced by the weak topology of $W(\widetilde{\mathbf{E}})_L$ and $N = V$ in $\Rep_{o_L}(G_L)$ equipped with its $\pi_L$-adic topology. The diagonal action of $G_L$ on $\mathbf{A} \otimes_{o_L} V$ then, indeed, is continuous. In addition, since the $\pi_L$-adic topology on $\mathbf{A}$ is finer than the weak topology any open $o_L$-submodule of $\mathbf{A}$ contains $\pi_L^j \mathbf{A}$ for a sufficiently big $j$. Hence $\{\im(U \otimes_{o_L} V \rightarrow \mathbf{A} \otimes_{o_L} V) : U \subseteq \mathbf{A}\ \text{any open $o_L$-submodule}\}$ is a fundamental system of open neighbourhoods of zero in $\mathbf{A} \otimes_{o_L} V$. This implies that the tensor product topology on $\mathbf{A} \otimes_{o_L} V$ is nothing else than its weak topology as a finitely generated $\mathbf{A}$-module.

\begin{remark}\label{induces-weak}
For any $V$ in $\Rep_{o_L}(G_L)$ the tensor product topology on $\mathbf{A} \otimes_{o_L} V$ induces the weak topology on $D_{LT}(V)$. In particular, the residual $\Gamma_L$-action on $D_{LT}(V)$ is continuous.
\end{remark}
\begin{proof}
The finitely generated $\mathbf{A}_L$-module $D_{LT}(V)$ is of the form $D_{LT}(V) \cong \oplus_{i=1}^r \mathbf{A}_L/\pi_L^{n_i} \mathbf{A}_L$ with $1 \leq n_i \leq \infty$. Using the isomorphism in the subsequent Prop.\ \ref{exact}.ii we obtain that $\mathbf{A} \otimes_{o_L} V \cong \oplus_{i=1}^r \mathbf{A}/\pi_L^{n_i} \mathbf{A}$. We see that the inclusion $D_{LT}(V) \subseteq \mathbf{A} \otimes_{o_L} V$ is isomorphic to the direct product of the inclusions $\mathbf{A}_L/\pi_L^{n_i} \mathbf{A}_L \subseteq \mathbf{A}/\pi_L^{n_i} \mathbf{A}$, which clearly are compatible with the weak topologies.
%\footnote{{\color{red} I guess the last claim needs that the $\pi_L$-adic topology on $W(\widetilde{\mathbf{E}})_L$ induces the $\pi_L$-adic topology of $\mathbf{A}$. Again this will eventually be in my course.}}
\end{proof}

\begin{proposition}\phantomsection\label{exact}
\begin{itemize}
  \item[i.] The functor $D_{LT}$ is exact.
  \item[ii.] For any $V$ in $\Rep_{o_L}(G_L)$ the natural map $\mathbf{A} \otimes_{\mathbf{A}_L} D_{LT}(V) \xrightarrow{\; \cong \;} \mathbf{A} \otimes_{o_L} V$ is  an isomorphism (compatible with the $G_L$-action and the Frobenius on both sides).
\end{itemize}
\end{proposition}
\begin{proof}
We begin with three preliminary observations.

1) As $\bA$ is $o_L$-torsion free, the functor $\bA\otimes_{o_L}-$ is exact.

2) The functor $D_{LT}$ restricted to the full subcategory of finite length objects $V$ in $\Rep_{o_L}(G_L)$ is exact. This follows immediately from 1) and the vanishing of $H^1(H_L,\bA\otimes_{o_L} V)$ in Lemma \ref{phi-version} below.

3) For any $V$ in $\Rep_{o_L}(G_L)$ we have $D_{LT}(V) = \varprojlim_n D_{LT}(V/\pi_L^n V)$. To see this we compute
\begin{align*}
  \varprojlim_n D_{LT}(V/\pi_L^n V) & = \varprojlim_n (\mathbf{A} \otimes_{o_L} V/\pi_L^n V)^{H_L} = (\varprojlim_n (\mathbf{A} \otimes_{o_L} V/\pi_L^n V))^{H_L} \\
      & = (\mathbf{A} \otimes_{o_L} \varprojlim_n V/\pi_L^n V)^{H_L} = (\mathbf{A} \otimes_{o_L} V)^{H_L} \\
      & = D_{LT}(V) \ .
\end{align*}
here the third identity becomes obvious if one notes that $V$ as an $o_L$-module is a finite direct sum of modules of the form $o_L /\pi_L^j o_L$ with $1 \leq j \leq \infty$. In this case $\varprojlim_n (\mathbf{A} \otimes_{o_L} o_L/\pi_L^{j+n} o_L) = \varprojlim_n \mathbf{A}/\pi_L^{j+n} \mathbf{A} = \mathbf{A}/\pi_L^j \mathbf{A} = \mathbf{A} \otimes_{o_L} o_L/\pi_L^j o_L$.

i. Let
\begin{equation*}
    0 \longrightarrow V_1 \longrightarrow V_2 \longrightarrow V_3 \longrightarrow 0
\end{equation*}
be an exact sequence in $\Rep_{o_L}(G_L)$. By 2) we obtain the exact sequences of projective systems of finitely generated $\mathbf{A}_L$-modules
\begin{equation*}
   D_{LT}(\ker(\pi_L^n|V_3)) \longrightarrow D_{LT}(V_1/\pi_L^n V_1) \longrightarrow D_{LT}(V_2/\pi_L^nV_2)\longrightarrow D_{LT}(V_3/\pi_L^nV_3) \longrightarrow 0 \ .
\end{equation*}
Since $\bA_L$ is a noetherian pseudocompact ring taking projective limits is exact. By 3) the resulting exact sequence is
\begin{equation*}
  \varprojlim_n D_{LT}(\ker(\pi_L^n|V_3)) \longrightarrow D_{LT}(V_1) \longrightarrow D_{LT}(V_2) \longrightarrow D_{LT}(V_3) \longrightarrow 0 \ .
\end{equation*}
But since the torsion subgroup of $V_3$ is finite and the transition maps in the projective system $(\ker(\pi_L^n|V_3))_n$ are multiplication by $\pi_L$, any composite of sufficiently many transition maps in this projective system and hence also in the projective system $(D_{LT}(\ker(\pi_L^n|V_3)))_n$ is zero. It follows that $\varprojlim_n D_{LT}(\ker(\pi_L^n|V_3)) = 0$.

ii.  The compatibility properties are obvious from the definition of the map. To show its bijectivity we may assume, by devissage and 3), that $\pi_L V = 0$. In this case our assertion reduces to the bijectivity of the natural map $\mathbf{E}_L^{sep} \otimes_{\mathbf{E}_L} (\mathbf{E}_L^{sep} \otimes_k V)^{\mathrm{Gal}(\mathbf{E}_L^{sep}/\mathbf{E}_L)} \longrightarrow \mathbf{E}_L^{sep} \otimes_k V$. But this is a well known consequence of the vanishing of the Galois cohomology group $H^1(\mathrm{Gal}(\mathbf{E}_L^{sep}/\mathbf{E}_L),\GL_d(\mathbf{E}_L^{sep}))$ where $d := \mathrm{dim}_k V$.
\end{proof}

\begin{lemma}\label{innerhom}
The above equivalence of categories $D_{LT}$ is compatible with the formation of inner Hom-objects, i.e., there are canonical isomorphisms
\begin{equation*}
  \Hom_{\bA_L}(D_{LT}(V_1),D_{LT}(V_2)) = D_{LT}(\Hom_{o_L}(V_1,V_2))
\end{equation*}
for every $V_1,V_2$ in $\Rep_{o_L}(G_L)$. We also have
\begin{equation*}
   \psi_{\Hom_{\bA_L}(D_{LT}(V_1),D_{LT}(V_2))} = \psi_{D_{LT}(\Hom_{o_L}(V_1,V_2))} \ .
\end{equation*}
\end{lemma}
\begin{proof}
We have
\begin{align*}
   D_{LT}(\Hom_{o_L}(V_1,V_2)) & = (\bA\otimes_{o_L}\Hom_{o_L}(V_1,V_2))^{H_L}   \\
   & = \Hom_{\bA}(\bA \otimes_{o_L} V_1, \bA \otimes_{o_L} V_2)^{H_L} \\
   & = \Hom_{\bA}(\bA\otimes_{\bA_L}D_{LT}(V_1), \bA \otimes_{o_L} V_2)^{H_L}   \\
   & = \Hom_{\bA_L}( D_{LT}(V_1),\bA \otimes_{o_L} V_2)^{H_L} \\
   & = \Hom_{\bA_L}( D_{LT}(V_1),(\bA \otimes_{\bA_L} V_2)^{H_L}) \\
   & = \Hom_{\bA_L}( D_{LT}(V_1),D_{LT}(V_2)) \ .
\end{align*}
Here the second identity is clear for $V_1$ being free, the general case follows by choosing a finite presentation of $V_1$ (as $o_L$-module neglecting the group action). The third identity uses Prop.\ \ref{exact}.ii, while the fourth one comes from the adjointness of base extension and restriction. The fifth one uses the fact that $H_L$ acts trivially on $D_{LT}(V_1)$.

One easily checks that the above sequence of identities is compatible with the $\Gamma_L$-actions (which are induced by the diagonal $G_L$-action on $\mathbf{A} \otimes_{o_L} -$).  The compatibility with Frobenius can be seen as follows.
First of all we abbreviate $\varphi_{D_{LT}(\Hom)} := \varphi_{D_{LT}(\Hom_{o_L}(V_1,V_2))}$ and $\varphi_{\Hom} := \varphi_{\Hom_{\bA_L}( D_{LT}(V_1),D_{LT}(V_2))}$. An element $\beta = \sum_i a_i \otimes \alpha_i \in (\bA \otimes_{o_L}\Hom_{o_L}(V_1,V_2))^{H_L}$ becomes, under the above identifications, the map $\iota_\beta : D_{LT}(V_1) \to (\bA \otimes_{o_L} V_2)^{H_L} $ which sends $\sum_j c_j\otimes v_j$ to $\sum_{i,j} a_ic_j\otimes \alpha_i(v_j )$. Assuming that $c_j = \phi_q(c_j')$ we compute
\begin{align*}
    \iota_{\varphi_{D_{LT}(\Hom)}(\beta)}(\varphi_{D_{ LT}(V_1)}(\sum_j c_j' \otimes v_j) & = \iota_{\sum_i \phi_q(a_i) \otimes \alpha_i}(\sum_j \phi_q(c_j') \otimes v_j) \\
    & = \sum_{i,j} \phi_q(a_i)\phi_q(c_j') \otimes \alpha_i(v_j) \\
    & = \varphi_{D_{ LT}(V_2)}(\sum_{i,j} a_ic_j' \otimes \alpha_i(v_j)) \\
    & = \varphi_{D_{ LT}(V_2)}(\iota_\beta(\sum_j c_j' \otimes v_j))  \\
    & = \varphi_{\Hom}(\iota_\beta)(\varphi_{D_{ LT}(V_1)}(\sum_j c_j' \otimes v_j)) \ ,
\end{align*}
where the last identity comes from \ref{f:inner}. Using the etaleness of $D_{LT}(V_1)$ we deduce that $\iota_{\varphi_{D_{LT}(\Hom)}(\beta)} = \varphi_{\Hom}(\iota_\beta)$ for any $\beta \in D_{LT}(\Hom_{o_L}(V_1,V_2))$. The additional formula for the $\psi$-operators is a formal consequence of the compatibility of the $\varphi$-operators.
\end{proof}

\begin{remark}\label{twist}
For any $V$ in $\Rep_{o_L}(G_L)$ and any continuous character $\chi : \Gamma_L \longrightarrow o_L^\times$ with representation module $W_\chi = o_Lt_\chi$ the twisted $G_L$-representation $V(\chi)$ is defined to be $V(\chi) =  V \otimes W_\chi$ as $o_L$-module and $\sigma_{| V(\chi)}(v \otimes w) = \sigma_{|V}(v) \otimes \sigma_{|W_\chi}(w) = \chi(\sigma) \cdot \sigma_{| V}(v)\otimes w$. One easily checks that $D_{LT}(V(\chi)) = D_{LT}(V)(\chi)$. If $V=o_L/\pi_L^no_L$, $1\leq n\leq \infty$ is the trivial representation, we usually identify $V(\chi)$ and $W_\chi$. Recall that for the character $\chi_{LT}$ we take $W_{\chi_{LT}}=T=o_L t_0$ and $W_{\chi_{LT}^{-1}}=T^*=o_Lt_0^*$ as representation module, where $T^*$ denotes the $o_L$-dual with dual basis $t_0^*$ of $t_0$.
\end{remark}

Defining $\Omega^1 := \Omega^1_{\bA_L} = \bA_L d\omega_{LT}$ any choice of $t_0$ defines an isomorphism $\Omega^1_{\mathscr{A}_L} \cong \Omega^1_{\bA_L}$ by sending $f(Z)dZ$ to $f(\omega_{LT})d \omega_{LT}$; moreover $\Omega^1_{\mathscr{A}_L}$ and $\Omega^1_{\bA_L}$ correspond to each other under the equivalence of categories \eqref{f:equCatvariable}. Due to the isomorphism \eqref{f:isoVariousA} we obtain a residue pairing
\begin{align}\label{f:residue-pairingbA}
  \bA_L \times \Omega^1_{\bA_L} & \longrightarrow o_L \\
                     (f(\omega_{LT}),g(\omega_{LT})d\omega_{LT}) & \longmapsto \mathrm{Res}(f(Z)g(Z)dZ)    \nonumber
\end{align}
which satisfies
\begin{equation}\label{f:adjunctionbA}
  \mathrm{Res}(f\psi_{\Omega^1}(\omega)) = \mathrm{Res}(\phi_q(f)\omega) \qquad\text{for any $f \in \bA_L$ and $\omega \in \Omega^1_{\bA_L}$}
\end{equation}
(by Lemma  \ref{phi-psi-dual}) and which is independent of the choice of $t_0$, i.e., $\omega_{LT}$, by Remark \ref{Res-variable}.ii  and Lemma \ref{iota}.a. In particular, we have a well defined map $\mathrm{Res} : \Omega^1_{\bA_L} \to o_L$.
In this context Remark \ref{psi-invariance}  together with Remark \ref{psi}.vii tell us that
\begin{equation}\label{f:psi-invariancebA}
   \psi_{\Omega^1}(\frac{d\hat{u}}{\hat{u}}) = \frac{d\hat{u}}{\hat{u}} \
\end{equation}
holds true for every $\hat{u} \in (\mathbf{A}_L^\times)^{\mathcal{N}=1}$.

Moreover, Lemma \ref{top-adjunction} translates into the (existence of the) topological isomorphism
\begin{align}\label{f:Lemma3.6bA-Version}
  \Hom_{\bA_L}(M,\Omega_{\bA_L}^1/\pi_L^n \Omega_{\bA_L}^1) & \xrightarrow{\;\cong\;} \Hom_{o_L}^c(M,L/o_L) \\
                                                   \nonumber     F & \longmapsto \pi_L^{-n} \mathrm{Res}(F(.)) \bmod o_L \ ,
\end{align}
for any $M$ annihilated by $\pi_L^n$.   Lemma \ref{Omega-as-twist} implies the  isomorphism
\begin{align*}
    \bA_L(\chi_{LT}) = \bA_L \otimes_{o_L} T & \xrightarrow{\;\cong\;} \Omega_{\bA_L}^1   \\
    f(\iota_{LT}(t_0)) \otimes t_0 & \longmapsto f(\iota_{LT}(t_0)) g_{LT}(\iota_{LT}(t_0)) d\iota_{LT}(t_0) \ .
\end{align*}
Using Lemma \ref{iota}.a as well as the second identity in \eqref{f:dlog} one verifies that this isomorphism (unlike its origin in Lemma \ref{Omega-as-twist}) does not depend on the choice of $t_0$.  We use it in order to transform \eqref{f:Lemma3.6bA-Version} into the
topological isomorphism
\begin{align}\label{f:dualitybA-Version}
  \Hom_{\bA_L}(M,\bA_L/\pi_L^n\bA_L(\chi_{LT})  ) & \xrightarrow{\;\cong\;} \Hom_{o_L}^c(M,L/o_L) .
\end{align}
Finally we obtain the following analogues of Prop.\ \ref{pairing]} and Remark  \ref{pairing-rangle}.

\begin{remark}\label{pairing-ranglebAVersion}
For any etale $(\varphi_L,\Gamma_L)$-module $M$ such that $\pi_L^n M = 0$ one has the $\Gamma_L$-invariant (jointly) continuous pairing
\begin{align*}
[\;,\;\rangle=[\;,\;\rangle_M: M \times \Hom_{\bA_L}(M,\bA_L/\pi_L^n \bA_L(\chi_{LT}) ) & \longrightarrow  L/o_L \\
(m,F) & \longmapsto \pi_L^{-n} \mathrm{Res} (F(m)d\log_{LT}(\omega_{LT})) \bmod o_L
\end{align*}
with adjointness properties analogous to the ones in Prop.\ \ref{pairing]}. Again, it is independent of the choice of $t_0$.
\end{remark}

\section{Iwasawa cohomology}

For any $V$ in $\Rep_{o_L}(G_L)$ we also write $H^i(K,V )= H^i(G_K,V)$, for any algebraic extension $K$ of $L$, and we often abbreviate $\varphi := \varphi_{D_{LT}(V)}$.

\begin{remark}\label{phi-1}
For any $V$ in $\Rep_{o_L}(G_L)$ the sequence
\begin{equation}\label{f:phi-1}
  0 \longrightarrow V \xrightarrow{\; \subseteq \;} \mathbf{A} \otimes_{o_L} V \xrightarrow{\phi_q \otimes \id - 1} \mathbf{A} \otimes_{o_L} V \longrightarrow 0
\end{equation}
is exact.
\end{remark}
\begin{proof}
We clearly have the exact sequence $0 \rightarrow k_L \rightarrow \mathbf{E}_L^{sep} \xrightarrow{x \mapsto x^q - x} \mathbf{E}_L^{sep} \rightarrow 0$. By devissage we deduce the exact sequence $0 \rightarrow o_L/\pi_L^n o_L \rightarrow \mathbf{A}/\pi_L^n \mathbf{A} \xrightarrow{\phi_q - 1} \mathbf{A}/\pi_L^n \mathbf{A} \rightarrow 0$ for any $n \geq 1$. Since the projective system $\{o_L/\pi_L^n o_L\}_n$ has surjective transition maps, passing to the projective limit is exact and gives the exact sequence
\begin{equation}
\label{f:Artin-Schreier-sequence}
0 \rightarrow o_L \rightarrow \mathbf{A} \xrightarrow{\phi_q - 1} \mathbf{A} \rightarrow 0.\end{equation} Finally, $\mathbf{A}$ is $o_L$-torsion free and hence flat over $o_L$. It follows that tensoring by $V$ is exact.
\end{proof}

Since $\mathbf{A}$ is the $\pi_L$-adic completion of an unramified extension of $\mathbf{A}_L$ with Galois group $H_L$ the $H_L$-action on $\mathbf{A}/\pi_L^n \mathbf{A}$, for any $n \geq 1$, and hence on $\mathbf{A} \otimes_{o_L} V$, whenever $\pi_L^n V = 0$ for some $n \geq 1$, is continuous for the discrete topology. We therefore may, in the latter case, pass from \eqref{f:phi-1} to the associated long exact Galois cohomology sequence with respect to $H_L$.

\begin{lemma}\label{phi-version}
Suppose that $\pi_L^n V = 0$ for some $n \geq 1$; we then have:
\begin{itemize}
  \item[i.] $H^i(H_L, \mathbf{A} \otimes_{o_L} V) = 0$ for any $i \geq 1$;
  \item[ii.] the long exact cohomology sequence for \eqref{f:phi-1} gives rise to an exact sequence
\begin{equation}\label{f:partial-phi}
  0 \longrightarrow H^0(L_\infty,V) \longrightarrow D_{LT}(V) \xrightarrow{\;\varphi - 1\;} D_{LT}(V) \xrightarrow{\;\partial_\varphi\;} H^1(L_\infty,V) \longrightarrow 0 \ .
\end{equation}
\end{itemize}
\end{lemma}
\begin{proof}
i. Since $\mathbf{A} \otimes_{o_L} V \cong \mathbf{A} \otimes_{\mathbf{A}_L} D_{LT}(V)$ by Prop.\ \ref{exact}.ii, it suffices to show the vanishing of $H^i(H_L, \mathbf{A}/\pi_L^n \mathbf{A}) = 0$ for $i \geq 1$. This reduces, by devissage, to the case $n = 1$, i.e., to $H^i(H_L, \mathbf{E}_L^{sep}) = 0$, which is a standard fact of Galois cohomology. ii. follows immediately from i.
\end{proof}

In order to derive from this a computation of Iwasawa cohomology in terms of $(\varphi_L,\Gamma_L)$-modules we first have to recall Pontrjagin duality and local Tate duality in our setting. The trace pairing
\begin{align*}
L \times L & \longrightarrow  \mathbb{Q}_p  \\
(x,y) & \longmapsto \mathrm{Tr}_{L/\mathbb{Q}_p}(xy)
\end{align*}
gives rise to the inverse different
\begin{align*}
  \mathfrak{D}_{L/\mathbb{Q}_p}^{-1} = \{ x \in L : \mathrm{Tr}_{L/\mathbb{Q}_p}(x o_L) \subseteq \mathbb{Z}_p \} & \xrightarrow{\;\cong\;} \Hom_{\mathbb{Z}_p}(o_L,\mathbb{Z}_p) \\
  y & \longmapsto [x \mapsto \mathrm{Tr}_{L/\mathbb{Q}_p}(yx)] \ .
\end{align*}
Let $\mathfrak{D}_{L/\mathbb{Q}_p} = \pi_L^s o_L$. we fix once and for all the $o_L$-linear  isomorphism
\begin{align}\label{f:different}
o_L & \xrightarrow{\;\cong\;} \Hom_{\mathbb{Z}_p}(o_L,\mathbb{Z}_p) \\
y & \longmapsto [x \mapsto \mathrm{Tr}_{L/\mathbb{Q}_p}(\pi_L^{-s}xy)] \ . \nonumber
\end{align}
By tensoring with $\mathbb{Q}_p/\mathbb{Z}_p$ it induces the isomorphism of torsion $o_L$-modules
\begin{equation*}
  \Xi: L/o_L\cong \Hom_{\mathbb{Z}_p}(o_L,\mathbb{Z}_p)\otimes_{\mathbb{Z}_p}\mathbb{Q}_p/\mathbb{Z}_p\cong \Hom_{\mathbb{Z}_p}(o_L,\mathbb{Q}_p/\mathbb{Z}_p) \ .
\end{equation*}

Now let $M$ be any topological $o_L$-module. Since $\Hom_{\mathbb{Z}_p}(o_L,-)$ is right adjoint to scalar restriction from $o_L$ to $\mathbb{Z}_p$ and by using $\Xi^{-1}$ in the second step, we have a natural isomorphism
\begin{equation}\label{f:LPontrjagin}
  \Hom_{\mathbb{Z}_p}(M,\mathbb{Q}_p/\mathbb{Z}_p) \cong \Hom_{o_L}(M,\Hom_{\mathbb{Z}_p}(o_L,\mathbb{Q}_p/\mathbb{Z}_p)) \cong \Hom_{o_L}(M,L/o_L) \ .
\end{equation}

\begin{lemma}\label{LPontrjagin}
The isomorphism \eqref{f:LPontrjagin} restricts to an isomorphism
\begin{equation*}
\Hom_{\mathbb{Z}_p}^c(M,\mathbb{Q}_p/\mathbb{Z}_p) \cong \Hom_{o_L}^c(M,L/o_L)
\end{equation*}
of topological groups between the subgroups of continuous homomorphisms endowed with the compact-open topology.
\end{lemma}
\begin{proof}
Coming from an isomorphism between the targets the second isomorphism in \eqref{f:LPontrjagin} obviously restricts to a topological isomorphism
\begin{equation*}
  \Hom_{o_L}^c(M,\Hom_{\mathbb{Z}_p}(o_L,\mathbb{Q}_p/\mathbb{Z}_p)) \cong \Hom_{o_L}^c(M,L/o_L) \ .
\end{equation*}
The first isomorphism is induced by the homomorphism $\lambda \mapsto \lambda(1)$ between the targets and therefore, at least restricts to a continuous injective map
\begin{equation*}
  \Hom_{o_L}^c(M,\Hom_{\mathbb{Z}_p}(o_L,\mathbb{Q}_p/\mathbb{Z}_p)) \longrightarrow \Hom_{\mathbb{Z}_p}^c(M,\mathbb{Q}_p/\mathbb{Z}_p) \ .
\end{equation*}
Let $\ell : M \rightarrow \mathbb{Q}_p/\mathbb{Z}_p$ be a continuous homomorphism. Then the composite map
\begin{equation*}
  M \times o_L \xrightarrow{(m,a) \mapsto am} M \xrightarrow{\;\ell\;} \mathbb{Q}_p/\mathbb{Z}_p
\end{equation*}
is continuous. Therefore the preimage $F_\ell \in \Hom_{o_L}(M,\Hom_{\mathbb{Z}_p}(o_L,\mathbb{Q}_p/\mathbb{Z}_p))$ of $\ell$, which is given by $F_\ell(m)(a) := \ell(am)$, is continuous by \cite{B-TG} X.28 Thm.\ 3. Finally let $A \subseteq M$ be any compact subset and $V \subseteq \mathbb{Q}_p/\mathbb{Z}_p$ be any subset. Define $B := \{a \in o_L : aA \subseteq A\}$. Then $1 \in B$, and, since $A$ is closed, also $B$ is closed and hence compact. Put $\widetilde{V} := \{\lambda \in \Hom_{\mathbb{Z}_p}(o_L,\mathbb{Q}_p/\mathbb{Z}_p) : \lambda(B) \subseteq V\}$. One easily checks that $\ell(A) \subseteq V$ if and only if $F_\ell(A) \subseteq \widetilde{V}$. This means that the inverse bijection $\ell \mapsto F_\ell$ is continuous as well.
\end{proof}

In the following we shall use the notation
\begin{equation*}
  M^\vee:=\Hom_{o_L}^c(M,L/o_L) \ ,
\end{equation*}
always equipped with the compact-open topology. The following version of Pontrjagin duality should be well known. Since we could not find a reference we will sketch a proof for the convenience of the reader.

\begin{proposition}[Pontrjagin duality]\label{LPontrjagin-duality}
The functor $-^\vee$ defines an involutory contravariant autoequivalence of the category of (Hausdorff) locally compact linear-topological $o_L$-modules. In particular, for such a module $M$, the canonical map
\begin{equation*}
  M \xrightarrow{\ \cong\ } (M^\vee)^\vee
\end{equation*}
is an isomorphism of topological $o_L$-modules.
\end{proposition}
\begin{proof}
We recall that a topological $o_L$-module $M$ is called linear-topological if it has a fundamental system of open zero neighbourhoods consisting of $o_L$-submodules. If $M$ is linear-topological and locally compact one easily checks that it has a fundamental system of open zero neighbourhoods consisting of compact open $o_L$-submodules.

Classical Pontrjagin duality $M \longmapsto \Hom^c(M, \mathbb{R}/\mathbb{Z})$, the right hand side being the group of all continuous group homomorphisms equipped with the compact-open topology, is an autoequivalence of the category of locally compact abelian groups $M$. We first compare this, for any locally compact linear-topological $\mathbb{Z}_p$-module $M$, with the group $\Hom^c_{\mathbb{Z}_p}(M, \mathbb{Q}_p/\mathbb{Z}_p)$, as always equipped with the compact-open topology. There is an obvious injective and continuous map
\begin{equation}\label{f:Pontrjagin}
  \Hom^c_{\mathbb{Z}_p}(M, \mathbb{Q}_p/\mathbb{Z}_p) \longrightarrow \Hom^c(M, \mathbb{R}/\mathbb{Z}) \ .
\end{equation}

\textit{Step 1: The map \eqref{f:Pontrjagin} is bijective.} Let $\ell : M \rightarrow \mathbb{R}/\mathbb{Z}$ be any continuous group homomorphism. We have to show that $\im (\ell) \subseteq \mathbb{Q}_p/\mathbb{Z}_p$ and that $\ell$ is continuous for the discrete topology on $\mathbb{Q}_p/\mathbb{Z}_p$. We fix a compact-open $\mathbb{Z}_p$-submodule $U \subseteq M$. Then $\ell(U)$ is a compact subgroup of $\mathbb{R}/\mathbb{Z}$, and hence is either equal to $\mathbb{R}/\mathbb{Z}$ or is finite. Since $U$ is profinite the former cannot occur. We conclude that $\ell (U)$ is a finite subgroup of $\mathbb{Q}_p/\mathbb{Z}_p$. In particular, there is an $r \in \mathbb{N}$ such that $p^r \cdot \ell | U = 0$. The quotient module $M/U$ is discrete and $\mathbb{Z}_p$-torsion. It follows that $p^r \cdot \ell (M) \subseteq \mathbb{Q}_p/\mathbb{Z}_p$ and hence that $\ell (M) \subseteq \mathbb{Q}_p/\mathbb{Z}_p$. Since $\mathbb{R}/\mathbb{Z}$ induces the discrete topology on the finite subgroup $\ell (U)$ the restricted homomorphism $\ell : U \rightarrow \mathbb{Q}_p/\mathbb{Z}_p$ is continuous. Since $U$ is open in $M$ this suffices for the continuity of $\ell : M \rightarrow \mathbb{Q}_p/\mathbb{Z}_p$. That $\ell$ then is a homomorphism of $\mathbb{Z}_p$-modules is automatic.

\textit{Step 2: The map \eqref{f:Pontrjagin} is open.} As usual, $C(A,V)$, for a compact subset $A \subseteq M$ and an arbitrary subset $V \subseteq \mathbb{Q}_p/\mathbb{Z}_p$, denotes the open subset of all $\ell \in \Hom^c_{\mathbb{Z}_p}(M, \mathbb{Q}_p/\mathbb{Z}_p)$ such that $\ell (A) \subseteq V$. First one checks that the $C(m + U,V)$, for $m \in M$, $U \subseteq M$ a compact-open $\mathbb{Z}_p$-submodule, and $V$ arbitrary, form a subbase of the compact-open topology. For any $\ell_0 \in C(m + U,V)$ we have $\ell_0 \in C(m + U, \ell_0(m) + \ell_0(U)) \subseteq C(m + U,V)$, where $\ell_0(U)$ is a finite subgroup of $\mathbb{Q}_p/\mathbb{Z}_p$. These observations reduce us to showing that the sets $C(m + U, v + \frac{1}{p^n}\mathbb{Z}/\mathbb{Z})$ are open in $\Hom^c(M, \mathbb{R}/\mathbb{Z})$. We fix a point $\ell_0 \in C(m + U, v + \frac{1}{p^n}\mathbb{Z}/\mathbb{Z})$, and we let $p^t$ be the order of $m$ modulo $U$. Note that we have $\ell_0(m) \in v + \frac{1}{p^n} \mathbb{Z}/\mathbb{Z}$ and hence
\begin{equation}\label{f:ell0}
  \ell_0(U) \subseteq \tfrac{1}{p^n} \mathbb{Z}/\mathbb{Z} \qquad\text{and}\qquad \ell_0(m) \in \tfrac{1}{p^{n+t}} \mathbb{Z}/\mathbb{Z} \ .
\end{equation}
We use the open subsets
\begin{equation*}
  V_1(\ell_0) := \big( (-\tfrac{1}{2p^{n+t+1}},\tfrac{1}{2p^{n+t+1}}) + \tfrac{1}{p^n}\mathbb{Z} \big)/\mathbb{Z} \subseteq V_2(\ell_0) := V_1(\ell_0) - V_1(\ell_0)  \subseteq \mathbb{R}/\mathbb{Z} \ .
\end{equation*}
They satisfy:
\begin{itemize}
  \item[a)] $V_2(\ell_0) \cap \frac{1}{p^{n+t+1}} \mathbb{Z}/\mathbb{Z} = V_1(\ell_0) \cap \frac{1}{p^{n+t+1}} \mathbb{Z}/\mathbb{Z} = \frac{1}{p^n} \mathbb{Z}/\mathbb{Z}$.
  \item[b)] $V_1(\ell_0) +  \frac{1}{p^n} \mathbb{Z}/\mathbb{Z} \subseteq V_1(\ell_0)$.
\end{itemize}
We claim that $\ell_0 \in C(m + U, v + V_1(\ell_0)) \subseteq C(m + U, v + \frac{1}{p^n} \mathbb{Z}/\mathbb{Z})$. This means that in $\Hom^c(M, \mathbb{R}/\mathbb{Z})$ we have found an open neighbourhood of $\ell_0$ which is contained in $C(m + U, v + \frac{1}{p^n} \mathbb{Z}/\mathbb{Z})$. Hence $C(m + U, v + \frac{1}{p^n}\mathbb{Z}/\mathbb{Z})$ is open in $\Hom^c(M, \mathbb{R}/\mathbb{Z})$. Since $\frac{1}{p^n} \mathbb{Z}/\mathbb{Z} \subseteq V_1(\ell_0)$ we certainly have $\ell_0 \in C(m + U, v + V_1(\ell_0))$. Let now $\ell \in C(m + U, v + V_1(\ell_0))$ be an arbitrary element. We have $\ell (U) = \frac{1}{p^j} \mathbb{Z}/\mathbb{Z}$ for some $j \geq 0$ and consequently $\ell(m) \in \frac{1}{p^{j+t}} \mathbb{Z}/\mathbb{Z}$.

\textit{Case 1: $j \leq n$.} Using \eqref{f:ell0} we then have $\ell_0(m) - v \in \frac{1}{p^n} \mathbb{Z}/\mathbb{Z}$ and $\ell(m) - \ell_0(m) \in \frac{1}{p^{n+t}} \mathbb{Z}/\mathbb{Z}$ and hence
\begin{equation*}
  \ell(m) - v = \ell(m) - \ell_0(m) + \ell_0(m) - v \in \tfrac{1}{p^{n+t}} \mathbb{Z}/\mathbb{Z} \ .
\end{equation*}
Since also $\ell(m) - v \in V_1(\ell_0)$ we deduce from a) that $\ell(m) - v \in \frac{1}{p^n} \mathbb{Z}/\mathbb{Z}$ and therefore
\begin{equation*}
  \ell (m + U) = \ell(m) + \ell(U) \subseteq v + \tfrac{1}{p^n} \mathbb{Z}/\mathbb{Z} + \tfrac{1}{p^j} \mathbb{Z}/\mathbb{Z} = v + \tfrac{1}{p^n} \mathbb{Z}/\mathbb{Z} \ .
\end{equation*}

\textit{Case 2: $j > n$.} We obtain
\begin{align*}
  \tfrac{1}{p^{n+1}} \mathbb{Z}/\mathbb{Z} & \subseteq \tfrac{1}{p^j} \mathbb{Z}/\mathbb{Z} = \ell(U) \subseteq - \ell(m) + v + V_1(\ell_0) \\
   & = - (\ell(m) - \ell_0(m)) -\ell_0(m) + v + V_1(\ell_0) \\
   & \subseteq - (\ell(m) - \ell_0(m)) + \tfrac{1}{p^n} \mathbb{Z}/\mathbb{Z} + V_1(\ell_0) \\
   & \subseteq - (\ell(m) - \ell_0(m)) + V_1(\ell_0) \\
   & \subseteq - (V_1(\ell_0) + \tfrac{1}{p^n} \mathbb{Z}/\mathbb{Z}) + V_1(\ell_0) \\
   & \subseteq V_1(\ell_0) - V_1(\ell_0) \\
   & \subseteq V_2(\ell_0) \ ,
\end{align*}
where the fourth and the sixth inclusion use b). This is in contradiction to b). We deduce that this case, in fact, cannot occur.

At this point we have shown that \eqref{f:Pontrjagin} is a topological isomorphism of locally compact abelian groups. From now on we assume that $M$ is a locally compact linear-topological $o_L$-module. By combining this latter isomorphism with the isomorphism in Lemma \ref{LPontrjagin} we obtain a topological isomorphism of locally compact abelian groups
\begin{equation}\label{f:topiso}
  \Hom_{o_L}^c(M,L/o_L) \cong \Hom^c(M, \mathbb{R}/\mathbb{Z}) \ ,
\end{equation}
which is natural in $M$. Of course, $\Hom_{o_L}^c(M,L/o_L)$ naturally is an $o_L$-module again.

\textit{Step 3: The $o_L$-module $\Hom_{o_L}^c(M,L/o_L)$ is linear-topological.} It is straightforward to check that the $C(U,\{0\})$ with $U$ running over all compact open $o_L$-submodules of $M$ form a fundamental system of open zero neighbourhoods in $M^\vee$. Each such $C(U,\{0\})$ evidently is an $o_L$-submodule.

Hence the topological isomorphism \eqref{f:topiso} also applies to $M^\vee$ instead of $M$. One checks that under this isomorphism the natural map $M \rightarrow \Hom^c(\Hom^c(M, \mathbb{R}/\mathbb{Z}), \mathbb{R}/\mathbb{Z})$ corresponds to the natural map $M \rightarrow (M^\vee)^\vee$. We finally see that the classical Pontrjagin duality implies that $M \xrightarrow{\cong} (M^\vee)^\vee$ is a topological isomorphism; it, of course, is $o_L$-linear.
\end{proof}

\begin{remark}\label{LPontrjagin-exactness}
Let $M_0 \xrightarrow{\alpha} M \xrightarrow{\beta} M_1$ be a sequence of locally compact linear-topological $o_L$-modules such that $\im (\alpha) = \ker (\beta)$ and $\beta$ is topologically strict with closed image; then the dual sequence $M_1^\vee \xrightarrow{\beta^\vee} M^\vee \xrightarrow{\alpha^\vee} M_0^\vee$ is exact as well, i.e., we have $\im (\beta^\vee) = \ker (\alpha^\vee)$.
\end{remark}
\begin{proof}
We have $\ker (\alpha^\vee) = (M/\im (\alpha))^\vee \cong \im (\beta)^\vee$, where the second isomorphism uses the assumption that $\beta$ is topologically strict. The assertion therefore reduces to the claim that the closed immersion $\im (\beta) \subseteq M_1$ induces a surjection between the corresponding Pontrjagin duals. For this see, for example, \cite{HR} Thm.\ 24.11.
\end{proof}

We recall that the weak topology on a finitely generated $\bA_L$-module $M$ is $o_L$-linear; moreover, it is locally compact if $M$ is annihilated by some power of $\pi_L$. Suppose that $M$ is a finitely generated $\bA_L/\pi_L^n \bA_L$-module. From \eqref{f:Lemma3.6bA-Version} and \eqref{f:dualitybA-Version} we have   topological isomorphisms $M^\vee \cong \Hom_{\bA_L}(M, \Omega^1/\pi_L^n \Omega^1)\cong\Hom_{\bA_L}(M, \mathbf{A}_L/ \pi_L^n \mathbf{A}_L(\chi_{LT}))$. By Prop.\ \ref{LPontrjagin-duality} they dualize into  topological isomorphisms $M \cong \Hom_{\bA_L}(M, \Omega^1/\pi_L^n\Omega^1)^\vee\cong \Hom_{\bA_L}(M, \mathbf{A}_L/ \pi_L^n \mathbf{A}_L(\chi_{LT}))^\vee$. If $M$ actually is an etale $(\varphi_q,\Gamma_L)$-module then we see that in the adjoint pairs of maps $(\psi_M, \varphi_{\Hom_{\bA_L}(M, \Omega^1/\pi_L^n \Omega^1)})$ and $(\varphi_M, \psi_{\Hom_{\bA_L}(M, \Omega^1/\pi_L^n \Omega^1)})$ from Remark \ref{pairing-ranglebAVersion} each map determines the other uniquely.

\begin{remark}\label{Pontrjagin-dual-inner}
Let $V$ be an object in $\Rep_{o_L}(G_L)$  of finite length. Then, the pairing $[\;,\;\rangle_{D_{LT}(V)}$ from Remark \ref{pairing-ranglebAVersion}, the Remark \ref{twist},  and the compatibility of the functor $D_{LT}(-)$ with internal Hom's by Lemma \ref{innerhom} induce, for $n$ sufficiently large, a natural isomorphism of topological groups
\begin{align*}
 D_{LT}(V)^\vee & \cong \Hom_{\mathbf{A}_L}(D_{LT}(V), \mathbf{A}_L/ \pi_L^n \mathbf{A}_L(\chi_{LT})) \\
 & \cong \Hom_{\mathbf{A}_L}(D_{LT}(V), D_{LT}((o_L/ \pi_L^n o_L)(\chi_{LT}))) \\
  & = D_{LT}(\Hom_{o_L}(V, (o_L/\pi_L^n o_L) (\chi_{LT}))) = D_{LT}(V^\vee(\chi_{LT})) \ ,
\end{align*}
which is independent of the choice of $n$ and under which $\psi_{D_{LT}(V^\vee(\chi_{LT}))}$ identifies with $\varphi_{D_{LT}(V)}^\vee$ by Remark \ref{pairing-ranglebAVersion}.
\end{remark}

\begin{proposition}[Local Tate duality]\label{Tate-local}
Let $V$ be an object in $\Rep_{o_L}(G_L)$  of finite length, and $K$ any finite extension of $L$. Then the cup product and the local invariant map induce perfect pairings of finite $o_L$-modules
\begin{equation*}
  H^i(K,V) \times  H^{2-i}(K,\Hom_{\mathbb{Z}_p}(V,\mathbb{Q}_p/\mathbb{Z}_p(1)) ) \longrightarrow H^2(K, \mathbb{Q}_p/\mathbb{Z}_p(1)) = \mathbb{Q}_p/\mathbb{Z}_p
  \end{equation*}
and
\begin{equation*}
  H^i(K,V) \times H^{2-i}(K,\Hom_{o_L}(V,L/o_L(1)) ) \longrightarrow H^2(K, L/o_L(1)) = L/o_L
\end{equation*}
where $-(1)$ denotes the Galois twist by the cyclotomic character. In other words, there are canonical isomorphisms
\begin{equation*}
  H^i(K,V)\cong H^{2-i}(K,V^\vee(1))^\vee \ .
\end{equation*}
\end{proposition}
\begin{proof}
Note that the isomorphism $H^2(K, L/o_L(1)) = L/o_L$ arises from $H^2(K, \mathbb{Q}_p/ \mathbb{Z}_p(1)) = \mathbb{Q}_p/ \mathbb{Z}_p$ by tensoring with $o_L$ over $\mathbb{Z}_p$. The first pairing is the usual version of local Tate duality (cf.\ \cite{Ser} II.5.2 Thm.\ 2). It induces the first isomorphism in
\begin{align*}
H^i(K,V) & \cong \Hom_{\mathbb{Z}_p}(H^{2-i}(K,\Hom_{\mathbb{Z}_p}(V,\mathbb{Q}_p/\mathbb{Z}_p(1)) ),\mathbb{Q}_p/\mathbb{Z}_p) \\
& \cong \Hom_{o_L}(H^{2-i}(K,\Hom_{\mathbb{Z}_p}(V,\mathbb{Q}_p/\mathbb{Z}_p(1)) ),L/o_L) \\
& \cong \Hom_{o_L}(H^{2-i}(K,\Hom_{o_L}(V,L/o_L(1)) ) ),L/o_L) \ ,
\end{align*}
while the second and third are induced by Lemma \ref{LPontrjagin}. To obtain the second pairing it remains to check that the above composite isomorphism is given by the cup product again. By the functoriality properties of the cup product this reduces to the following formal fact. Let $\xi : L/o_L \longrightarrow \mathbb{Q}_p/\mathbb{Z}_p$ be any group homomorphism. Then the diagram
\begin{equation*}
  \xymatrix{
    H^2(K,L/o_L(1))  \ar[d]_{H^2(K,\xi(1))} \ar[r]^-{=} & L/o_L \ar[d]^{\xi} \\
    H^2(K,\mathbb{Q}_p/\mathbb{Z}_p(1)) \ar[r]^-{=} & \mathbb{Q}_p/\mathbb{Z}_p   }
\end{equation*}
commutes, where the horizontal maps are the local invariant maps. This in turn is an easy consequence of the $\mathbb{Z}_p$-linearity of the local invariant map if one uses the following description of $\xi$ viewed as a map $L/o_L = \mathbb{Q}_p/\mathbb{Z}_p \otimes_{\mathbb{Z}_p} o_L \xrightarrow{\xi} \mathbb{Q}_p/\mathbb{Z}_p$. Let $\zeta : o_L \longrightarrow \Hom_{\mathbb{Z}_p}(\mathbb{Q}_p/\mathbb{Z}_p, \mathbb{Q}_p/\mathbb{Z}_p) = \mathbb{Z}_p$ be the homomorphism which sends $a$ to $c \mapsto \xi(c \otimes a)$. Then $\xi(c \otimes a) = \zeta(a)c$.
\end{proof}

For any $V$ in $\Rep_{o_L}(G_L)$ we define the generalized Iwasawa cohomology of $V$ by
\begin{equation*}
    H^*_{Iw}(L_\infty/L,V) := \varprojlim_K H^*(K,V)
\end{equation*}
where $K$ runs through the finite Galois extensions of $L$ contained in $L_\infty$ and the transition maps in the projective system are the cohomological corestriction maps. \footnote{Note that, for any finite extension $K/L$ contained in $L_\infty$ the definition $H^*_{Iw}(L_\infty/K,V) := \varprojlim_{K \subseteq K' \subseteq L_\infty} H^*(K',V)$ produces the same $o_L$-modules. Our notation indicates that we always consider these groups as $\Gamma_L$-modules.}

Shapiro's lemma for cohomology gives natural isomorphisms
\begin{equation*}
    H^*(K,V) = H^*(G_{L}, o_L[G_{L}/G_K] \otimes_{o_L} V)
\end{equation*}
where, on the right hand side, $G_{L}$ acts diagonally on the coefficients. In this picture the corestriction map, for $K \subseteq K'$,  becomes the map induced on cohomology by the map $\pr \otimes \id_V : o_L[G_{L}/G_{K'}] \otimes_{o_L} V \longrightarrow o_L[G_{L}/G_K] \otimes_{o_L} V$.

\begin{lemma}\label{limit}
$H^*_{Iw}(L_\infty/L,V) = H^*(G_{L}, o_L[[\Gamma_L]] \otimes_{o_L} V)$ (where the right hand side refers to cohomology with continuous cochains).
\end{lemma}
\begin{proof}
On the level of continuous cochain complexes we compute
\begin{align*}
    \varprojlim_K C^\bullet(G_{L}, o_L[G_{L}/G_K] \otimes_{o_L} V) & = C^\bullet_{}(G_{L}, \varprojlim_K (o_L[G_{L}/G_K] \otimes_{o_L} V)) \\
    & = C^\bullet_{}(G_{L}, o_L[[\Gamma_L]] \otimes_{o_L} V) \ .
\end{align*}
The second identity comes from the isomorphism $\varprojlim_K (o_L[G_{L}/G_K] \otimes_{o_L} V) \cong o_L[[\Gamma_L]] \otimes_{o_L} V$ which is easily seen by using a presentation of the form $0 \rightarrow o_L^s \rightarrow o_L^r \rightarrow V \rightarrow 0$. Since the transition maps in this projective system of complexes are surjective the first hypercohomology spectral sequence for the composite functor $\varprojlim \circ H^0(G_{L},.)$ degenerates so that the second hypercohomology spectral sequence becomes
\begin{equation*}
    R^i \varprojlim_K H^j_{}(G_{L}, o_L[G_{L}/G_K] \otimes_{o_L} V) \Longrightarrow H^{i+j}_{}(G_{L}, o_L[[\Gamma_L]] \otimes_{o_L} V) \ .
\end{equation*}
Due to the countability of our projective system we have $R^i \varprojlim_K = 0$ for $i \geq 2$. Hence this spectral sequence degenerates into short exact sequences
\begin{multline*}
    0 \longrightarrow R^1 \varprojlim_K H^{i-1}_{}(G_{L}, o_L[G_{L}/G_K] \otimes_{o_L} V) \longrightarrow \\
    H^*_{}(G_{L}, o_L[[\Gamma_L]] \otimes_{o_L} V) \longrightarrow \varprojlim_K H^i_{}(G_{L}, o_L[G_{L}/G_K] \otimes_{o_L} V) \longrightarrow 0 \ .
\end{multline*}
It is well known that the Galois cohomology groups $H^*_{}(G_{L}, o_L[G_{L}/G_K] \otimes_{o_L} V)$ are finitely generated $o_L$-modules (cf.\ \cite{Ser} II.5.2 Prop.\ 14). It therefore follows from \cite{Jen} Thm.\ 8.1 that the above $R^1 \varprojlim_K$-terms vanish.
\end{proof}

\begin{lemma}\label{delta-functor}
$H^*_{Iw}(L_\infty/L,V)$ is a $\delta$-functor on $\Rep_{o_L}(G_L)$.
\end{lemma}
\begin{proof}
Let $0 \rightarrow V_1 \rightarrow V_2 \rightarrow V_3 \rightarrow 0$ be a short exact sequence in $\Rep_{o_L}(G_L)$. Then the sequence of topological $G_L$-modules $0 \rightarrow o_L[[\Gamma_L]] \otimes_{o_L} V_1 \rightarrow o_L[[\Gamma_L]] \otimes_{o_L} V_2 \rightarrow o_L[[\Gamma_L]] \otimes_{o_L} V_3 \rightarrow 0$ is short exact as well. In view of Lemma \ref{limit} our assertion therefore follows from \cite{NSW} Lemma 2.7.2 once we show that
\begin{itemize}
  \item[1.] the topology of $o_L[[\Gamma_L]] \otimes_{o_L} V_1$ is induced by the topology of $o_L[[\Gamma_L]] \otimes_{o_L} V_2$ and
  \item[2.] the (surjective continuous) map $o_L[[\Gamma_L]] \otimes_{o_L} V_2 \rightarrow o_L[[\Gamma_L]] \otimes_{o_L} V_3$ has a continuous section as a map of topological spaces.
\end{itemize}
Each $o_L[[\Gamma_L]] \otimes_{o_L} V_i$ is a profinite (hence compact) abelian group with a countable base of the topology, which therefore is metrizable by \cite{B-TG} IX.21 Prop.\ 16. One easily deduces 1. and that the map in 2. is open. The property 2. then follows from \cite{Mic} Cor.\ 1.4.
\end{proof}

\begin{remark}\label{cohomology-twist}
For any $V_0$ in $\Rep_{o_L}(G_L)$ which is $o_L$-free and on which $G_L$ acts through its factor $\Gamma_L$ there is a natural isomorphism $H^*_{Iw}(L_\infty/L,V \otimes_{o_L} V_0) \cong H^*_{Iw}(L_\infty/L,V) \otimes_{o_L} V_0$.
\end{remark}
\begin{proof}
In view of Lemma \ref{limit} the asserted isomorphism is induced by the $G_L$-equivariant isomorphism on coefficients
\begin{align*}
  o_L[[\Gamma_L]] \otimes_{o_L} V \otimes_{o_L} V_0 & \xrightarrow{\;\cong\;} o_L[[\Gamma_L]] \otimes_{o_L} V \otimes_{o_L} V_0 \\
  \gamma \otimes v \otimes v_0 & \longmapsto \gamma \otimes v \otimes \gamma^{-1} v_0 \ ;
\end{align*}
on the left $G_L$ acts diagonally on all three factors, whereas on the right it acts trivially on the third factor.
\end{proof}

\begin{remark}\label{TatelocalIw}
Let $V$ be in $\Rep_{o_L}(G_L)$ of finite length; in particular $V$ is discrete. Then local Tate duality (Prop.\ \ref{Tate-local}) induces an isomorphism
\begin{equation*}
  H^i_{Iw}(L_\infty/L,V)\cong H^{2-i}(L_\infty,V^\vee(1))^\vee \ .
\end{equation*}
\end{remark}
\begin{proof}
Use Prop.\ \ref{Tate-local} over the layers $L_n$ and take limits.
\end{proof}

We point out that $H^*_{Iw}(L_\infty/L,V) = H^*_{}(G_{L}, o_L[[\Gamma_L]] \otimes_{o_L} V)$ is a left $o_L[[\Gamma_L]]$-module through the action of $\gamma \in \Gamma_L$ by right multiplication with $\gamma^{-1}$ on the factor $o_L[[\Gamma_L]]$.
%\footnote{Fukaya and Kato use here the second alternative for the action which arises by applying the involution to $o_L[[\Gamma_L]]$. The latter has the strong advantage that the Iwasawa-module structure is the natural one while the $G_{L}$-structure is apparently more complicated, which does not matter as this is hidden behind taking cohomology. Thus Otmar strongly recommends to follow Fukaya and Kato's convention here in the end.}

\begin{lemma}\phantomsection\label{vanishing}
\begin{itemize}
  \item[i.] $H^*_{Iw}(L_\infty/L,V) = 0$ for $\ast \neq 1,2$.
  \item[ii.] $H^2_{Iw}(L_\infty/L,V)$ is finitely generated as $o_L$-module.
  \item[iii.] $H^1_{Iw}(L_\infty/L,V)$ is finitely generated as $o_L[[\Gamma_L]]$-module.
\end{itemize}
\end{lemma}
\begin{proof}
i. In case $\ast > 2$ the assertion follows from the fact that the groups $G_K$ have cohomological $p$-dimension $2$ (\cite{Ser} II.4.3 Prop.\ 12). The vanishing of $H^0_{Iw}(L_\infty/L,V) = \varprojlim_K V^{G_K}$ is clear if $V$ is finite. Hence we may assume that $V$ is finitely generated free over $o_L$. Note that the identity $H^0_{Iw}(L_\infty/L,V) = \varprojlim_K V^{G_K}$ shows that $H^0_{Iw}(L_\infty/L,V)$ is a profinite $o_L$-module. On the other hand we then have the exact sequence
\begin{equation*}
  0 \longrightarrow H^0_{Iw}(L_\infty/L,V) \xrightarrow{\;\pi_L\cdot\;} H^0_{Iw}(L_\infty/L,V) \longrightarrow H^0_{Iw}(L_\infty/L,V/\pi_L V)
\end{equation*}
Since we observed already that the last term vanishes it follows that $H^0_{Iw}(L_\infty/L,V)$ is an $L$-vector space. Both properties together enforce the vanishing of $H^0_{Iw}(L_\infty/L,V)$.

ii. We have
\begin{equation*}
  H^2_{Iw}(L_\infty/L,V) = \varprojlim_K H^2(K,V) = \varprojlim_K H^0(K,V^\vee(1))^\vee
  = (\bigcup_K V^\vee(1)^{G_K})^\vee \ ,
\end{equation*}
which visible is a finitely generated $o_L$-module.

iii. \textit{Case 1: $V$ is finite.} By Remark \ref{TatelocalIw} $H_{Iw}^1(L_\infty/L,V) = H^1(L_\infty,V^\vee(1))^\vee$ is the Pontrjagin dual of a discrete torsion module and hence  is a compact $o_L[[\Gamma_L]]$-module. The compact Nakayama lemma (cf.\ \cite{NSW} Lemma 5.2.18) therefore reduces us to showing that the Pontrjagin dual $(H_{Iw}^1(L_\infty/L,V)_\Gamma)^\vee = H^1(L_\infty,V^\vee(1))^\Gamma$ of the $\Gamma$-coinvariants of $H_{Iw}^1(L_\infty/L,V)$ is cofinitely generated over $o_L$; here $\Gamma$ is a conveniently chosen open subgroup of $\Gamma_L$. The Hochschild-Serre spectral sequence for the extension $L_\infty/K $, where $K:= L_\infty^\Gamma$, gives us an exact sequence
\begin{equation*}
  H^1(K,V^\vee(1)) \longrightarrow H^1(L_\infty,V^\vee(1))^\Gamma \longrightarrow H^2(\Gamma, V^\vee(1)^{H_L}) \ .
\end{equation*}
The first group is finite by local Galois cohomology. At this point we choose $\Gamma$ to be isomorphic to $\mathbb{Z}_p^d$. Then $H^2(\Gamma,\mathbb{F}_p)$ is finite. Since $\mathbb{F}_p$ is the only simple $\mathbb{Z}_p[[\Gamma]]$-module it follows by devissage that $H^2(\Gamma, V^\vee(1)^{H_L})$ is finite.

\textit{Case 2: $V$ is $o_L$-free.} As pointed out above $o_L[[\Gamma_L]] \otimes_{o_L} V$ is a $o_L[[\Gamma_L]]$-module, which is finitely generated free and on which $G_L$ acts continuously and  $o_L[[\Gamma_L]]$-linearly. In view of Lemma \ref{limit} we therefore may apply \cite{FK} Prop.\ 1.6.5 and obtain that $H^*_{Iw}(L_\infty/L,V)$, as a $o_L[[\Gamma_L]]$-module, is isomorphic to the cohomology of a bounded complex of finitely generated projective $o_L[[\Gamma_L]]$-modules. The ring $o_L[[\Gamma_L]]$ is noetherian. Hence the cohomology of such a complex is finitely generated.

The general case follows by using Lemma \ref{delta-functor} and applying the above two special cases to the outer terms of the short exact sequence $0 \rightarrow V_{tor} \rightarrow V \rightarrow V/V_{tor} \rightarrow 0$, where $V_{tor}$ denotes the torsion submodule of $V$.
\end{proof}

\begin{theorem}\label{psi-version}
Let $V$  in $\Rep_{o_L}(G_L)$.   Then, with $\psi = \psi_{D_{LT}(V(\tau^{-1}))}$, we have a short exact sequence
\begin{multline}\label{f:psi-version}
  0 \longrightarrow H^1_{Iw}(L_\infty/L,V) \longrightarrow D_{LT}(V(\chi_{LT}\chi_{cyc}^{-1})) \xrightarrow{\;\psi-1\;}  D_{LT}(V(\chi_{LT}\chi_{cyc}^{-1})) \\ \longrightarrow H^2_{Iw}(L_\infty/L,V) \longrightarrow 0 \ ,
\end{multline}
which is functorial in $V$.
\end{theorem}
\begin{proof}
In the sense of Remark \ref{twist} we take $T^*\otimes_{\mathbb{Z}_p}\mathbb{Z}_p(1)$ and $T\otimes_{\mathbb{Z}_p}\Hom_{\mathbb{Z}_p}(\mathbb{Z}_p(1),\mathbb{Z}_p)$ as representation module for $\tau$ and $\tau^{-1},$ respectively.

We first assume that $V$ has finite length. Then the exact sequence \eqref{f:partial-phi} is a sequence of locally compact linear-topological $o_L$-modules. In fact, the first term is finite and the last term is cofinitely generated over $o_L$. In particular, the first and the last term carry the discrete topology and the first map is a closed immersion. Moreover, the map $\varphi - 1$ in the middle, by Lemma \ref{top-strict}, is topologically strict with open image. The latter implies that $\partial_\varphi$ induces an isomorphism of discretely topologized $o_L$-modules $D_{LT}(V)/(\varphi - 1)D_{LT}(V) \xrightarrow{\cong} H^1(L_\infty,V)$. In particular, the map $\partial_\varphi$ is topologically strict as well.

Therefore, using Remark \ref{LPontrjagin-exactness}, we obtain that the dual sequence
\begin{equation*}
  0 \longrightarrow H^1(L_\infty,V)^\vee \longrightarrow D_{LT}(V)^\vee \xrightarrow{\;\varphi^\vee - 1\;} D_{LT}(V)^\vee \longrightarrow H^0(L_\infty,V)^\vee \longrightarrow 0
\end{equation*}
is exact. If we identify the terms in this latter sequence according to Remark \ref{TatelocalIw} and Remark \ref{Pontrjagin-dual-inner} then the result is the exact sequence in the assertion.

Now let $V$ be arbitrary and put $V_n := V/\pi_L^n V$. We have the exact sequence of projective systems
\begin{equation*}
  0 \rightarrow H^1_{Iw}(L_\infty/L,V_n) \rightarrow D_{LT}(V_n(\tau^{-1})) \xrightarrow{\;\psi-1\;}  D_{LT}(V_n(\tau^{-1})) \rightarrow H^2_{Iw}(L_\infty/L,V_n) \rightarrow 0 \ .
\end{equation*}
Since the functor $D_{LT}$ is exact (Prop.\ \ref{exact}.i) we have
\begin{equation*}
  \varprojlim_n D_{LT}(V_n(\tau^{-1})) = \varprojlim_n D_{LT}(V(\tau^{-1}))/\pi_L^n D_{LT}(V(\tau^{-1})) = D_{LT}(V(\tau^{-1})) \ .
\end{equation*}
Moreover,
\begin{align*}
  \varprojlim_n H^i_{Iw}(L_\infty/L,V_n) & = \varprojlim_n \varprojlim_K H^i(K,V_n) = \varprojlim_K \varprojlim_n H^i(K,V_n) = \varprojlim_K H^i(K,V) \\
   & = H^i_{Iw}(L_\infty/L,V) \ .
\end{align*}
Therefore it remains to show that passing to the projective limit in the above exact sequence of projective systems is exact. For this it suffices to show that $R^1 \varprojlim$ of the two projective systems $\{H^1_{Iw}(L_\infty/L,V_n)\}_n$ and $\{(\psi - 1)D_{LT}(V_n(\tau^{-1}))\}_n$ vanishes. Because of $D_{LT}(V_n(\tau^{-1})) = D_{LT}(V(\tau^{-1}))/\pi_L^n D_{LT}(V(\tau^{-1}))$ the transition maps in the second projective system are surjective, which guarantees the required vanishing. For the first projective system we choose an open pro-$p$ subgroup $\Gamma$ in $\Gamma_L$, so that $o_L[[\Gamma]]$ is a complete local noetherian commutative ring. From Lemma \ref{vanishing}.iii we know that $\{H^1_{Iw}(L_\infty/L,V_n)\}_n$ is a projective system of finitely generated $o_L[[\Gamma]]$-modules. Hence \cite{Jen} Thm.\ 8.1 applies and gives the required vanishing.
\end{proof}

\begin{remark}\label{Gamma-equivariance}
Each map in the exact sequence \eqref{f:psi-version} is continuous and $o_L[[\Gamma_L]]$-equivariant.
\end{remark}
\begin{proof}
Continuity and $\Gamma_L$-equivariance follow from the construction. Since the weak topology on $D_{LT}(V)$ is $o_L$-linear and complete we may apply \cite{Laz} Thm.\ II.2.2.6 (which is valid for any profinite group) and obtain that the continuous $\Gamma_L$-action extends, by continuity, uniquely to an $o_L[[\Gamma_L]]$-action on $D_{LT}(V)$.
\end{proof}

\section{The Kummer map}\label{sec:Kummer}

We consider the Kummer isomorphism
\begin{equation*}
  \kappa : A(L_\infty) := \varprojlim_{n,m} L_n^\times/{L_n^\times}^{p^m} \xrightarrow{\;\cong\;} H^1_{Iw}(L_\infty,\mathbb{Z}_p(1)) \ .
\end{equation*}
Recall that we have fixed a generator $t_0$ of the Tate module $T = o_L t_0$. Correspondingly we have the dual generator $t_0^*$ of the $o_L$-dual $T^* = o_L t_0^*$ of $T$.
This leads to the twisted Kummer isomorphism
\begin{equation*}
  A(L_\infty) \otimes_{\mathbb{Z}_p }T^* \xrightarrow[\cong]{\kappa \otimes_{\mathbb{Z}_p} T^*}  H^1_{Iw}(L_\infty,\mathbb{Z}_p(1)) \otimes_{\mathbb{Z}_p }T^* \cong H^1_{Iw}(L_\infty,o_L(\tau))
\end{equation*}
where the second isomorphism comes from Remark \ref{cohomology-twist}. On the other hand, by Thm.\ \ref{psi-version}, we have a natural  isomorphism
\begin{equation*}
  Exp^* : H^1_{Iw}(L_\infty,o_L(\tau)) \xrightarrow{\;\cong\;} D_{LT}(o_L)^{\psi=1} = \mathbf{A}_L^{\psi=1} \ .
\end{equation*}
%\footnote{Streng genommen m\"usste man f\"ur den rechten Iso in der vorletzten Formel auch noch eine Basis des zyklotomischen Tate-Moduls $\mathbb{Z}_p(1)$ w\"ahlen. Aber mir scheint unsere Behandlung der Twists (nach Remark 4.6, welche \"ubrigens den zyklotomischen Twist \"uberhaupt nicht abdeckt) sehr ungl\"ucklich. Wir betrachten nur Twists, die einen kanonischen Darstellungsmodul haben, n\"amlich $T$, $T^*$ und $\mathbb{Z}_p(1)$. Soweit m\"oglich h\"atten wir mit denen arbeiten sollen. Insbesondere scheint mir Thm. 5.13 mit $V \otimes T \otimes \mathbb{Z}_p(1)^{-1}$ statt $V(\chi_{LT}\chi_{cyc}^{-1})$ zu gelten -- siehe Remark 4.7. Folglich existiert die Abbildung $Exp^*$ kanonisch auf $H^1_{Iw}(L_\infty,\mathbb{Z}_p(1)) \otimes_{\mathbb{Z}_p }T^*$. Wenn ich nichts \"ubersehen habe, sind also alle drei Abbildungen in dem nachfolgenden Thm. 6.2 unabh\"angig von Wahlen.}

Finally we have the homomorphism
\begin{align*}
  \nabla : (\varprojlim_n L_n^\times) \otimes_\mathbb{Z} T^* & \longrightarrow  \mathbf{A}_L^{\psi=1} \\
       u \otimes at_0^* & \longmapsto a \frac{\partial_\mathrm{inv}(g_{u,t_0})}{g_{u,t_0}}(\iota_{LT}(t_0)) \ .
\end{align*}
It is well defined by Thm.\ \ref{Coleman}, the last sentence in section \ref{sec:Coleman}, and Remark \ref{psi}.ii.
%\footnote{$A(L_\infty)$ ist wohl strikt gr\"o{\ss}er als die $p$-adische Komplettierung von $\varprojlim_n L_n^\times$.}
%Viewed as a homomorphism into the $p$-adically complete group $\bA_L$ it extends uniquely to the $p$-adic completion $A(L_\infty) \otimes_{\mathbb{Z}_p }T^*$ of $\varprojlim_n L_n^\times \otimes_\mathbb{Z} T^*$. Since $\bA_L^{\psi_L = 1}$ is $p$-adically closed in $\bA_L$ this extension is still a homomorphism
%\footnote{Since $\bA_L$ is $p$-torsion free the induced topology on $\bA_L^{\psi_L = 1}$ coincides with the $p$-adic topology of $\bA_L^{\psi_L = 1}$. Hence $\bA_L^{\psi_L = 1}$ is $p$-adically complete.}
%\begin{equation*}
%  \nabla : A(L_\infty) \otimes_{\mathbb{Z}_p} T^*  \longrightarrow   \mathbf{A}_L^{\psi=1} \ .
%\end{equation*}

\begin{remark}\label{independencet0}
The map $\nabla$ is independent of the choice of $t_0$.
\end{remark}
\begin{proof}
Let $u \otimes at_0^* \in (\varprojlim_n L_n^\times) \otimes_\mathbb{Z} T^*$. We temporarily write $\nabla_{t_0}$ instead of $\nabla$, and we let $t_1$ be a second generator of $T$, so that $t_1 = c t_0$ for some $c \in o_L^\times$. Then $u \otimes at_0^* = u \otimes act_1^*$, and by inserting the definitions in the first line we compute
\begin{align*}
  \nabla_{t_1}(u \otimes at_0^*) & = \nabla_{t_1}(u \otimes act_1^*) = \frac{ac}{g_{LT}(\iota_{LT}(t_1))} \frac{g_{u,t_1}'(\iota_{LT}(t_1))}{g_{u,t_1}(\iota_{LT}(t_1))} \\
   & = \frac{ac}{g_{LT}(\iota_{LT}(ct_0))} \frac{g_{u,ct_0}'(\iota_{LT}(ct_0))}{g_{u,ct_0}(\iota_{LT}(ct_0))}  \\
   & = \frac{ac}{g_{LT}([c](\iota_{LT}(t_0)))} \frac{g_{u,ct_0}'([c](\iota_{LT}(t_0)))}{g_{u,ct_0}([c](\iota_{LT}(t_0)))}  \\
   & = \frac{a}{g_{LT}(\iota_{LT}(t_0))} \frac{(g_{u,ct_0} \circ [c])'(\iota_{LT}(t_0))}{(g_{u,ct_0}\circ [c])(\iota_{LT}(t_0))} \\
   & = \frac{a}{g_{LT}(\iota_{LT}(t_0))} \frac{g_{u,t_0}' (\iota_{LT}(t_0))}{g_{u,t_0} (\iota_{LT}(t_0))} \\
   & = \nabla_{t_0}(u \otimes at_0^*) \ ,
\end{align*}
where we use Lemma \ref{iota}.a for the fourth identity, \eqref{f:dlog} for the fifth one, and Remark \ref{change-generator}.ii for the sixth one.
\end{proof}

Generalizing \cite{CC} Prop.\ V.3.2.iii) (see also \cite{Co2} Thm.\ 7.4.1\footnote{Colmez has no minus but which seems to be a mistake.}) we will establish the following kind of reciprocity law.

\begin{theorem}\label{Kummer-commutative}
The diagram
  \begin{equation}\label{f:Kummer}
     \xymatrix{
         (\varprojlim_n L_n^\times) \otimes_{\mathbb{Z}} T^* \ar[dr]_-{\nabla} \ar[rr]^-{-\kappa \otimes T^*}_-{} &   &  H^1_{Iw}(L_\infty,o_L(\tau)) \ar[dl]^-{Exp^*}_-{\cong} \\
         & {\mathbf{A}_L^{\psi=1}}  &  }
\end{equation}
is commutative.
\end{theorem}

In a first step we consider, for any $n \geq 1$, the diagram
\begin{equation*}
\xymatrix{
   \big((\mathbf{A}_L/\pi_L^n \mathbf{A}_L)(\chi_{LT}) \big)/\im(\varphi - 1) \ar[d]_{\partial_\varphi}^{\cong}   \ar@{}[r]|{\times} & (\mathbf{A}_L/\pi_L^n \mathbf{A}_L)^{\psi=1}  \ar[r]^-{[\ , \ \rangle} & L/o_L \ar@{=}[d] \\
   H^1(L_\infty, o_L/\pi_L^n o_L(\chi_{LT})) \ar@{=}[dd] \ar@{}[r]|{\times} & H^1_{Iw}(L_\infty/L, o_L/\pi_L^n o_L(\chi_{cyc}\chi_{LT}^{-1}))  \ar[u]_{Exp^*}^{\cong} \ar[r] & L/o_L \ar@{=}[dd] \\
     & (\varprojlim_n L_n^\times) \otimes_{\bZ} o_L/\pi_L^n o_L(\chi_{LT}^{-1}) \ar[d]^{rec \otimes_{\mathbb{Z}_p} \id} \ar[u]_{-\kappa \otimes_{\mathbb{Z}_p} \id}   &   \\
   {\Hom(H_L,o_L/\pi_L^n o_L)(\chi_{LT})} \ar@{}[r]|{\times} & H_L^{ab}(p)\otimes_{\mathbb{Z}_p} o_L/\pi_L^n o_L(\chi_{LT}^{-1}) \ar[r] & L/o_L,   }
\end{equation*}
where the second pairing is induced by local Tate duality and the third pairing is the obvious one. By $rec : (\varprojlim_n L_n^\times) \longrightarrow H_L^{ab}(p)$ we denote the map into the maximal abelian pro-$p$ quotient $H_L^{ab}(p)$ of $H_L$ induced by the reciprocity homomorphisms of local class field theory for the intermediate extensions $L_m$. Note that $\Gal(L_\infty^{ab}/L_\infty) = \varprojlim_m \Gal(L_m^{ab}/L_\infty) = \varprojlim_m \Gal(L_m^{ab}/L_m)$, where $L_?^{ab}$ denotes the maximal abelian extension of $L_?$. The upper half of the diagram is commutative by the construction of the map $Exp^*$. The commutativity of the lower half follows from \cite{NSW} Cor.\ 7.2.13. All three pairings are perfect in the sense of Pontrjagin duality.

In order to prove Thm.\ \ref{Kummer-commutative} we have to show that, for any $u \in \varprojlim_n L_n^\times$ and any $a \in o_L$, we have
\begin{equation*}
  [z \otimes t_0,- Exp^*(\kappa(u) \otimes a t_0^*) \rangle \equiv \mathrm{Res} (za \frac{\partial_\mathrm{inv}(g_{u,t_0})}{g_{u,t_0}} d\log_{LT})  \mod \pi_L^n
\end{equation*}
for any $z \in \mathbf{A}_L$ and any $n \geq 1$. Due to the commutativity of the above diagram the left hand side is equal to $a\partial_\varphi(z \otimes t_0)(rec(u) \otimes t_0^*)= a\partial_\varphi(z )(rec(u))$. On the other hand the right hand side, by \eqref{f:inv}, is equal to $\mathrm{Res}(za (\frac{g_{u,t_0}'}{g_{u,t_0}}dZ)_{|Z = \iota_{LT}(t_0)}) = \mathrm{Res}(za \frac{d(g_{u,t_0}(\iota_{LT}(t_0)))}{g_{u,t_0}(\iota_{LT}(t_0))})$. By the $o_L$-bilinearity of all pairings involved we may assume that $a=1$. Hence we are reduced to proving that
\begin{equation*}
   \mathrm{Res}(z \frac{d(g_{u,t_0}(\iota_{LT}(t_0)))}{g_{u,t_0}(\iota_{LT}(t_0))}) = \partial_\varphi(z )(rec(u))
\end{equation*}
holds true for any $z \in \bA_L$  and $u \in \varprojlim_n L_n^\times$. According to the theory of fields of norms we have the natural identification $\varprojlim_n L_n^\times = \mathbf{E}_L^\times$ (cf.\ \cite{KR} Lemma 1.4). Under this identification, by \cite{Lau} Thm.\ 3.2.2, $rec(u)$ coincides with the image $rec_{\mathbf{E}_L}(u)$ of $u$ under the reciprocity homomorphism $rec_{\mathbf{E}_L} : \mathbf{E}_L^\times \longrightarrow H_L^{ab}(p)$ in characteristic $p$. Furthermore, $g_{u,t_0}(\iota_{LT}(t_0)) \in  (\bA_L^\times)^{\mathcal{N} = 1}$,  is, by Remark \ref{change-generator}.i,  Remark \ref{psi}.vii, and \eqref{f:N-iso}, a lift of $u \in \mathbf{E}_L^\times$. This reduces the proof of Thm.\ \ref{Kummer-commutative} further to the following proposition which generalizes the explicit reciprocity law in \cite{Fo1} Prop.\ 2.4.3.

\begin{proposition}\label{reduction-to-char-p}
For any $z\in \mathbf{A}_L$  and any $u \in \mathbf{E}_L^\times$ with (unique) lift $\hat{u} \in (\mathbf{A}_L^\times)^{\mathcal{N} = 1}$ we have
\begin{equation*}
\mathrm{Res}(z\frac{d\hat{u}}{\hat{u}}) = \partial_\varphi(z)(rec_{\mathbf{E}_L}(u)) \ ,
\end{equation*}
where $\partial_\varphi $ is the connecting homomorphism in \eqref{f:partial-phi}.
\end{proposition}

Obviously the connecting homomorphism $\partial_\varphi$ for $V = o_L$ induces, by reduction modulo $\pi^n o_L$, the corresponding connecting homomorphism for $V = o_L/\pi_L^n o_L$. Hence we may prove the identity in Prop.\ \ref{reduction-to-char-p} as a congruence modulo $\pi^n o_L$ for any $n \geq 1$. Recall that for $\hat{u} \in  (\mathbf{A}_L^\times)^{\mathcal{N} = 1}$ the differential form $\frac{d\hat{u}}{\hat{u}}$ is $\psi_{\Omega^1}$-invariant by \eqref{f:psi-invariancebA}. Hence, by the adjointness of $\psi_{\Omega^1}$ and $\varphi_L$ (cf.\   \ref{f:adjunctionbA}), we obtain the equality
\begin{equation*}
\mathrm{Res}(\varphi_L^m(z)\frac{d\hat{u}}{\hat{u}}) = \mathrm{Res}(z\frac{d\hat{u}}{\hat{u}})
\end{equation*}
for any   $m \geq 1$. This reduces Prop.\ \ref{reduction-to-char-p} and consequently Thm.\ \ref{Kummer-commutative} to proving the congruence
\begin{equation}\label{f:key-identity}
\mathrm{Res}(\varphi_L^{n-1}(z)\frac{d\hat{u}}{\hat{u}}) \equiv \partial_\varphi(z)(rec_{\mathbf{E}_L}(u)) \mod \pi_L^n o_L
\end{equation}
for all $n \geq 1$. This will be the content of the next section (cf.\ Lemma \ref{lemma-Witt}).

\section{The generalized Schmid-Witt formula}\label{sec:Witt}

The aim of this section is to generalize parts of Witt's seminal paper \cite{Wit} (see also the detailed accounts \cite{Tho} and \cite{Koe} of Witt's original article) to the case of {\it ramified} Witt vectors.

First of all we need to recall a few facts about ramified Witt vectors $W(B)_L$ for $o_L$-algebras $B$. Details of this construction can be found in \cite{Haz}. But we will use \cite{GAL} where a much more straightforward approach is fully worked out. We denote by $\Phi_B = (\Phi_0,\Phi_1,\ldots) : W(B)_L \longrightarrow B^{\mathbb{N}_0}$ the homomorphism of $o_L$-algebra, called the ghost map, given by the polynomials $\Phi_n(X_0,\ldots,X_n) = X_0^{q^n} + \pi_L X_1^{q^{n-1}} + \ldots \pi_L^n X_n$. On the other hand, the multiplicative Teichm\"uller map $B \longrightarrow W(B)_L$ is given by $b \mapsto [b] := (b,0,\ldots)$ (cf.\ \cite{GAL} Lemma \ref{GAL-teichmuller}). If $B$ is a $k_L$-algebra then the Frobenius endomorphism $F = \phi_q$ of $W(B)_L$ has the form $\phi_q (b_0,\ldots,b_n,\ldots) = (b_0^q,\ldots,b_n^q,\ldots)$ (cf.\ \cite{GAL} Prop.\ \ref{GAL-charp}.i).

For a perfect $k_L$-algebra $B$ we have $W_n(B)_L = W(B)_L/ \pi_L^n W(B)_L$ for any $n \geq 1$ and, in particular, $\ker(\Phi_0) = \pi_L W(B)_L$; moreover, any $b = (b_0,b_1,\ldots) \in W(B)_L$ has the unique convergent expansion $b = \sum_{m=1}^\infty \pi_L^m [b_m^{q^{-m}}]$ (cf.\ \cite{GAL} Prop.\ \ref{GAL-perfekt}).

\begin{proposition}\label{s-map}
Suppose that $\pi_L$ is not a zero divisor in $B$ and that $B$ has an endomorphism of $o_L$-algebras $\sigma$ such that $\sigma(b) \equiv b^q \bmod \pi_L B$ for any $b \in B$. Then there is a unique homomorphism of $o_L$-algebras
\begin{equation*}
  s_B : B \longrightarrow W(B)_L \quad\text{such that $\Phi_i \circ s_B = \sigma^i$ for any $i \geq 0$}.
\end{equation*}
Moreover, we have:
\begin{itemize}
  \item[i.] $s_B$ is injective;
  \item[ii.] for any $n \geq 1$ there is a unique homomorphism of $o_L$-algebras $s_{B,n} : B/\pi_L^n B \longrightarrow W_n(B/\pi_L B)_L$ such that the diagram
\begin{equation*}
  \xymatrix{
    B \ar[d]_{\pr} \ar[r]^-{s_B} & W(B)_L \ar[r]^-{W(\pr)_L} & W(B/ \pi_L B)_L \ar[d]^{\pr} \\
    B/\pi_L^n B \ar[rr]^{s_{B,n}} &  & W_n(B/\pi_L B)_L   }
\end{equation*}
      is commutative;
  \item[iii.] if $B/\pi_L B$ is perfect then $s_{B,n}$, for any $n \geq 1$, is an isomorphism.
\end{itemize}
\end{proposition}
\begin{proof}
See \cite{GAL} Prop.\ \ref{GAL-s-map}.
\end{proof}

\begin{lemma}\label{phi-maps}
For any perfect $k_L$-algebra $B$ we have:
\begin{itemize}
  \item[i.] The diagram
\begin{equation*}
  \xymatrix{
    W_n(W(B)_L)_L \ar[d]_{W_n(\mathrm{pr})_L} \ar[r]^-{\Phi_{n-1}} & W(B)_L \ar[d]^{\mathrm{pr}} \\
    W_n(B)_L \ar[r]^{\phi_q^{n-1}} & W_n(B)_L   }
\end{equation*}
is commutative for any $n \geq 1$.
  \item[ii.] The composite map
\begin{equation*}
  W(B)_L \xrightarrow{\; s_{W(B)_L} \;} W(W(B)_L)_L \xrightarrow{W(\pr)_L} W(B)_L
\end{equation*}
is the identity.
\end{itemize}
\end{lemma}
\begin{proof}
i. Let $(\mathbf{b}_0,\ldots,\mathbf{b}_{n-1}) \in W_n(W(B)_L)_L$ with $\mathbf{b}_j = (b_{j,0}, b_{j,1},\ldots)$. As $\ker(\Phi_0) = \pi_L W(B)_L$ we have $\mathbf{b}_j \equiv [b_{j,0}] \bmod \pi_L W(B)_L$. Hence \cite{GAL} Lemma  \ref{GAL-q-congruence} implies that $\mathbf{b}_j^{q^m} \equiv [b_{j,0}^{q^m}] \bmod \pi_L^{m+1} W(B)_L$ for any $m \geq 0$. Using this as well as \cite{GAL} Lemma \ref{GAL-sum}.i we now compute
\begin{align*}
  \Phi_{n-1}(\mathbf{b}_0,\ldots,\mathbf{b}_{n-1}) & = \sum_{m=0}^{n-1} \pi_L^m \mathbf{b}_m^{q^{n-1-m}} \\
   & \equiv \sum_{m=0}^{n-1} \pi_L^m [b_{m,0}^{q^{n-1-m}}]  \mod \pi_L^n W(B)_L  \\
   & = (b_{0,0}^{q^{n-1}}, \ldots , b_{n-1,0}^{q^{n-1}}, 0,\ldots)  \\
   & = \phi_q^{n-1}(b_{0,0} , \ldots ,b_{n-1,0}, 0,\ldots) = \phi_q^{n-1} \circ W_n(\mathrm{pr})_L (\mathbf{b}_0,\ldots,\mathbf{b}_{n-1}) \ .
\end{align*}

ii. First of all we note that the Frobenius on $W(B)_L$ is the $q$th power map modulo $\pi_L$. Hence the homomorphism $s_{W(B)_L}$ exists.

Let $\mathbf{b} = (\mathbf{b}_0,\ldots,\mathbf{b}_j,\ldots) \in W(W(B)_L)_L$ with $\mathbf{b}_j = (b_{j,0}, b_{j,1},\ldots) \in W(B)_L$ be the image under $s_{W(B)_L}$ of some $b = (b_0,b_1,\ldots) \in W(B)_L$. We have to show that $b_i = b_{i,0}$ for any $i \geq 0$. By the characterizing property of $s_{W(B)_L}$ we have $\Phi_i(\mathbf{b}) = (b_0^{q^i},\ldots,b_j^{q^i},\ldots)$. On the other hand the computation in the proof of i. shows that $\Phi_i(\mathbf{b}) = (b_{0,0}^{q^i}, \ldots , b_{i,0}^{q^i},\ldots)$. Hence $b_i^{q^i} = b_{i,0}^{q^i}$ and therefore $b_i = b_{i,0}$.
\end{proof}

By construction (and \cite{GAL} Prop.\ \ref{GAL-comparison}) we have $\mathbf{A}_L \subseteq W(\widetilde{\mathbf{E}})_L$. But there is the following observation.

\begin{remark}\label{contained}
Let $A \subseteq W(\widetilde{\mathbf{E}})_L$ be a $\phi_q$-invariant $o_L$-subalgebra such that $A/\pi_L A \subseteq \widetilde{\mathbf{E}}$; we then have $A \subseteq W(A/\pi_L A)_L$.
\end{remark}
\begin{proof}
We consider the diagram
\begin{equation*}
  \xymatrix{
     A \ar[d]_{\subseteq} \ar[r]^-{s_A} & W(A)_L \ar[d]_{\subseteq} \ar[r]^-{W(\pr)_L} & W(A/\pi_L A)_L \ar[d]_{\subseteq} \\
    W(\widetilde{\mathbf{E}})_L \ar[r]^-{s_{W(\widetilde{\mathbf{E}})_L}} & W(W(\widetilde{\mathbf{E}})_L)_L \ar[r]^-{W(\pr)_L} & W(\widetilde{\mathbf{E}})_L .   }
\end{equation*}
For the commutative left square we apply Prop.\ \ref{s-map} (with $\sigma := \phi_q$). The right hand square is commutative by naturality. By Lemma \ref{phi-maps}.ii the composite map in the bottom row is the identity. Hence the composite map in the top row must be an inclusion.
\end{proof}

This applies to $\mathbf{A}_L$ and shows that
\begin{equation*}
  \mathbf{A}_L \subseteq W(\mathbf{E}_L)_L \subseteq W(\widetilde{\mathbf{E}})_L
\end{equation*}
holds true. In particular, we have the commutative diagram (cf.\ Prop.\ \ref{s-map}.ii)
\begin{equation}\label{f:alpha-n}
  \xymatrix{
    \mathbf{A}_L \ar[d]_{\pr} \ar[rr]^{\subseteq} \ar@{-->}[drr]^{\alpha_n} & & W(\mathbf{E}_L)_L \ar[d]^{\pr} \\
    \mathbf{A}_L/\pi_L^n \mathbf{A}_L \ar[rr]^{\overline{\alpha}_n := s_{\mathbf{A}_L,n}} & & W_n(\mathbf{E}_L)_L   }
\end{equation}
for any $n \geq 1$, where $\alpha_n$ by definition is the composite of the outer maps. For later use before Lemma \ref{lemma-Witt} we note that Remark \ref{contained} also applies to $\mathbf{A}$ showing that $\mathbf{A} \subseteq W(\mathbf{E}_L^{sep})_L$.

\begin{lemma}\label{phi-maps2}
For any $n \geq 1$ the diagram
\begin{equation*}
  \xymatrix{
    W_n(\mathbf{A}_L)_L \ar[d]_{W_n(\pr)_L} \ar[r]^-{\Phi_{n-1}} & \mathbf{A}_L \ar[d]^{\alpha_n} \\
    W_n(\mathbf{E}_L)_L \ar[r]^{\phi_q^{n-1}} & W_n(\mathbf{E}_L)_L   }
\end{equation*}
is commutative.
\end{lemma}
\begin{proof}
We consider the diagram
\begin{equation*}
  \xymatrix{
  W_n(\mathbf{A}_L)_L \ar[rr]^{\Phi_{n-1}} \ar[dd]_{W_n(\pr)_L} \ar[dr]_{\subseteq}
      & & \mathbf{A}_L \ar'[d][dd]^{\alpha_n} \ar[dr]^{\subseteq}       \\
  & W_n(W(\widetilde{\mathbf{E}})_L)_L \ar[rr]^(0.4){\Phi_{n-1}} \ar[dd]_(0.3){W_n(\mathrm{pr})_L}
      &  & W(\widetilde{\mathbf{E}})_L \ar[dd]^{\mathrm{pr}} \\
  W_n(\mathbf{E}_L)_L \ar'[r][rr]^{\phi_q^{n-1}} \ar[dr]_{\subseteq} & & W_n(\mathbf{E}_L)_L \ar[dr]^{\subseteq}
                 \\
  & W_n(\widetilde{\mathbf{E}})_L \ar[rr]^{\phi_q^{n-1}}
      &  &  W_n(\widetilde{\mathbf{E}})_L .      }
\end{equation*}
The front square is commutative by Lemma \ref{phi-maps}.i. The top and bottom squares are commutative by the naturality of the involved maps and the side squares for trivial reasons. Hence the back square is commutative as claimed.
\end{proof}

\begin{lemma}\label{overlinealpha}
For any $n \geq 1$ the map $\overline{\alpha}_n : \mathbf{A}_L /\pi_L^n \mathbf{A}_L \longrightarrow W_n(\mathbf{E}_L)_L$ is injective.
\end{lemma}
\begin{proof}
We have to show that $V_n(\mathbf{E}_L)_L \cap \mathbf{A}_L = \pi_L^n \mathbf{A}_L$. We know already that $\pi_L^n \mathbf{A}_L \subseteq V_n(\mathbf{E}_L)_L \cap \mathbf{A}_L = V_n(\widetilde{\mathbf{E}})_L \cap \mathbf{A}_L = \pi_L ^n W(\widetilde{\mathbf{E}})_L \cap \mathbf{A}_L$. But $\mathbf{A}_L \subseteq W(\widetilde{\mathbf{E}})_L$ both are discrete valuation rings with the prime element $\pi_L$. Therefore we must have equality.
\end{proof}

The above two lemmas together with the surjectivity of $W_n(\alpha_1)_L$  imply that, for any $n \geq 1$, there is a unique homomorphism of $o_L$-algebras $w_{n-1} : W_n(\mathbf{E}_L)_L \longrightarrow \mathbf{A}_L / \pi_L^n \mathbf{A}_L$ such that the diagram
\begin{equation}\label{f:A-omega}
  \xymatrix{
    W_n(\mathbf{A}_L)_L \ar[d]_{W_n(\pr)_L} \ar[r]^-{\Phi_{n-1}} & \mathbf{A}_L \ar[d]^{\mathrm{pr}} \\
    W_n(\mathbf{E}_L)_L \ar[r]^{w_{n-1}} &  \mathbf{A}_L / \pi_L^n \mathbf{A}_L  }
\end{equation}
is commutative. Furthermore, we have
\begin{equation}\label{f:alphaphiq}
\overline{\alpha}_n \circ w_{n-1} = {\phi_q^{n-1}}_{ | W_n(\mathbf{E}_L)_L}  \qquad\text{and}\qquad  w_{n-1} \circ \overline{\alpha}_n = {\phi_q^{n-1}}_{ | \mathbf{A}_L / \pi_L^n \mathbf{A}_L} \ .
\end{equation}

We also may apply Prop.\ \ref{s-map} to $o_L$ itself (with $\sigma := \id$) and obtain analogous commutative diagrams as well as the corresponding maps
\begin{equation*}
  \xymatrix{
    o_L \ar[d]_{\pr} \ar[dr]^-{\alpha_n}  &  & \\
    o_L/ \pi_L^n o_L \ar[r]^-{\overline{\alpha}_n} & W_n(k_L)_L \ar[r]^-{w_{n-1}} & o_L/ \pi_L^n o_L \ .   }
\end{equation*}
But here $\overline{\alpha}_n$ and $w_{n-1}$ are isomorphisms which are inverse to each other (cf.\ Prop.\ \ref{s-map}.iii). Of course these maps for $o_L$ and $\mathbf{A}_L$ are compatible with respect to the inclusions $o_L \subseteq \mathbf{A}_L$ and $k_L \subseteq \mathbf{E}_L$.\\

For the rest of this section let $K$ denote any local field isomorphic to $k((Z))$ with $k=k_L$ (such an isomorphism depending on the choice of an uniformizing element $Z$ of $K$), $K^{sep}$ any separable closure of it and $H = \Gal(K^{sep}/K)$ its Galois group. Furthermore we write $\overline{K}$ for an algebraic closure of $K^{sep}$, $\phi_q$ for the $q$th power Frobenius, and $\wp:= \phi_q -1$ for the corresponding Artin-Schreier operator. By induction with respect to $n$ one easily proves the following fact.
%\footnote{Vgl.\ Lorenz, Einf\"uhrung in die Algebra II, \S26.5 F2}

\begin{lemma}\label{Artin-Schreier}
We have the short exact sequences
\begin{equation*}
  0 \longrightarrow W_n(k)_L \longrightarrow W_n(K^{sep})_L \xrightarrow{\;\wp\;} W_n(K^{sep})_L \longrightarrow 0
\end{equation*}
and
\begin{equation*}
  0 \longrightarrow W_n(k)_L \longrightarrow W_n(\bar{ K})_L \xrightarrow{\;\wp\;}   W_n(\overline{K})_L \longrightarrow 0 \ .
\end{equation*}
\end{lemma}
From the $H$-group cohomology long exact sequence  associated with the first sequence above we obtain a homomorphism
\begin{equation*}
    W_n(K)_L=(W_n(K^{sep})_L)^H \xrightarrow{\ \partial\ } H^1(H,W_n(k)_L)  \xrightarrow{rec_K^*} \Hom^{cont}(K^\times,W_n(k)_L) \ ,
\end{equation*}
which induces the generalized, bilinear Artin-Schreier-Witt pairing
\begin{align*}
   [\;,\;):= [\;,\;)_K: W_n(K)_L \times K^\times & \longrightarrow  W_n(k)_L  \\
          (x,a) & \longmapsto [x,a):=\partial(x)(rec_K(a)) \ ,
\end{align*}
i.e., $[x,a)=rec_K(a)(\alpha)-\alpha$ for any $\alpha \in W_n(K^{sep})_L$ with $\wp(\alpha)=x$. It is bilinear in the sense that it si $o_L$-linear in the first and additive in the second variable.

\begin{remark}\label{Krad}
Let $K^{rad}$ be the perfect closure of $K$ in $\overline{K}$. Then one can use the second exact sequence to extend the above pairing to $ W_n(K^{rad})_L \times K^\times$.
\end{remark}

For a separable extension  $F$ of $K$ we obtain similarly by taking  $\mathrm{Gal}(K^{sep}/F)$-invariants (instead of $H$-invariants) an Artin-Schreier-Witt-pairing for $F$
\begin{align*}
   [\;,\;)_F : W_n(F)_L \times F^\times & \longrightarrow  W_n(k)_L  \\
          (x,a) & \longmapsto [x,a):=\partial(x)(rec_F(a)) \ ,
\end{align*} (with respect to the same $q$!) satisfying
\begin{equation*}
  [x,a)_F = [x,\mathrm{Norm}_{F/K}(a))_K  \qquad\text{for $x \in W_n(K)_L$ and $a \in F^\times$}
\end{equation*}
- and similarly for any pair of separable extensions $F$ and $F'$ - by the functoriality of class field theory.

Although $W_n(k)_L$ is not a cyclic group in general, many aspects of Kummer/Artin-Schreier theory still work. In particular, for any $\alpha = (\alpha_0,\ldots,\alpha_{n-1}) \in W_n(K^{sep})_L$ with $\wp(\alpha) = x \in W_n(K)_L$ the extension $K(\alpha) := K(\alpha_0,\ldots,\alpha_{n-1}) = (K^{sep})^{H_x} = K(\wp^{-1}(x))$ of $K$ is Galois with abelian Galois group $\Gal(K(\alpha)/K)$ contained in $W_n(k)_L$ via sending $\sigma$ to $\chi_x(\sigma) := \sigma(\alpha)-\alpha$; here $H_x \subseteq H$ denotes the stabilizer of $\alpha$, which also is the stabilizer of $\wp^{-1}(x)$.

We also need the injective additive map
\begin{align*}
  \tau : \qquad\  W_n(B)_L & \longrightarrow W_{n+1}(B)_L \\
  (x_0,\ldots,x_{n-1}) & \longmapsto (0,x_0,\ldots,x_{n-1})
\end{align*}
induced in an obvious way by the additive Verschiebung $V$ (cf.\ \cite{GAL} Prop.\ \ref{GAL-F-V}). If $B$ is a $k_L$-algebra then
\begin{equation}\label{f:tau-additivity}
\tau \circ \phi_q = \phi_q \circ \tau \quad\text{and, in particular,}\quad  \wp \circ \tau = \tau \circ \wp
\end{equation}
(cf.\ \cite{GAL} Prop.\ \ref{GAL-charp}.i).

\begin{lemma}\label{properties[)}
Let $K \subseteq F\subseteq K^{sep}$ be a finite extension. Then, for any $a \in F^\times$, $x\in W_n(F)_L$, and $\alpha \in W_n(K^{sep})_L$ with $\wp(\alpha) = x$ we have:
 \begin{enumerate}
 \item[i.] $[\tau x, a)_F=\tau[x,a)_F$ (where we use the same notation for the pairing at level $n+1$ and $n$, respectively!);
 \item[ii.]  if  $a$ belongs to $(F^\times)^{p^n}$, then $[x,a)_F=0$;
 \item[iii.] $[x,a)_F=0$ if and only if $a \in \mathrm{Norm}_{F(\alpha)/F}(F(\alpha)^\times)$.
 \end{enumerate}
\end{lemma}
\begin{proof}
i. By \eqref{f:tau-additivity} we have $\wp(\tau\alpha)=\tau x$. Therefore $[\tau x,a) = rec_K(a)(\tau\alpha)-\tau\alpha = \tau(rec_K(a)(\alpha)-\alpha) = \tau[x,a)$.
ii. Since $p^nW_n(k)_L=0$ (this is not sharp with regard to $p^n$!) this is immediate from the bilinearity of the pairing. iii. Because of $[x,a)_F = \chi_x(rec_F(a))$ this is clear from local class field theory.
\end{proof}

For any subset $S$ of an $o_L$-algebra $R$ we define the \textit{subset}
\begin{equation*}
  W_n(S)_L := \{(s_0,\ldots,s_{n-1})\in W_n(R)_L : s_i \in S\ \text{for all $i$}\}
\end{equation*}
of $W_n(R)_L$ as well as $V W_n(S)_L := V(W_n(S)_L)$. If $I \subseteq R$ is an ideal then $W_n(I)_L$ is an ideal in $W_n(R)_L$, and we have the exact sequence
\begin{equation}\label{f:exact-ideal}
  0 \longrightarrow W_n(I)_L \longrightarrow W_n(R)_L \longrightarrow W_n(R/I)_L \longrightarrow 0 \ .
\end{equation}
If $R' \subseteq R$ is an $o_L$-subalgebra (not necessarily with a unit), then $W_n(R')_L\subseteq W_n(R)_L$ forms a subgroup and there is an exact sequence of abelian groups
\begin{equation}\label{f:exactR'}
    0 \longrightarrow VW_n(R')_L \longrightarrow W_n(R')_L \xrightarrow{\;\Phi_0\;} R' \longrightarrow 0 \ .
\end{equation}
We apply this to $R'=ak[a]$ for $a\in K^\times$.

\begin{proposition}\label{aka}
For any $x\in W_n(ak[a])_L$ we have $[x,a)=0$.
\end{proposition}
\begin{proof}
We prove by induction on $n$ that for any finite separable extension $F$ of $K$ and for any $a \in F^\times$ the corresponding statement holds true with $R' = ak[a]$.

For both, $n=1$ (trivially) and $n>1$ (by induction hypothesis and  Lemma \ref{properties[)}.i), we know the implication
\begin{equation*}
  x \in VW_n(ak[a])_L \Longrightarrow [x,a)=0 \ .
\end{equation*}
Therefore, for arbitrary $x \in W_n(ak[a])_L$ we have $[x,a)=[[x_0],a)$ by the bilinearity of the pairing and Lemma    \ref{GAL-sum}.i in \cite{GAL} - we constantly will make use of the fact that the first component of Witt vectors behaves additively. Moreover, again by the additivity of the pairing and using \eqref{f:exactR'} it suffices to prove (for all $n>0$) that
\begin{equation*}
  [[ra^l],a)=0 \mbox{ for all }r\in k^\times, l\geq 1 \ .
\end{equation*}
Writing $l=l'p^m$ with $l'$ and $p$ coprime and denoting by $r'$ the $p^m$th root of $r$ we see that
\begin{equation*}
  l'[[ra^l],a) = [[(r'a^{l'})^{p^m}],a^{l'}) = [[(r'a^{l'})^{p^m}],a^{l'})+[[(r'a^{l'})^{p^m}],r') = [[(r'a^{l'})^{p^m}],r'a^{l'})
\end{equation*}
by Lemma \ref{properties[)}.ii because $r'\in (k^\times)^{p^n}$. Noting that $l'$ is a unit in $W_n(k)_L$ we are reduced to the case $x=[a^{p^m}]$ for $m\geq 0$.

To this end let $\alpha_0\in F^{sep}$ be in $\wp^{-1}(a)$ and define $\tilde{\alpha_0}:=\prod_{\xi \in k/\im(\chi_a)}(\alpha_0 +\xi).$ Then
\begin{align}\label{f:norm}
   \mathrm{Norm}_{F(\alpha_0)/F}(\tilde{\alpha_0}) & = \prod_{\sigma\in \mathrm{Gal}(F(\alpha_0)/F)}\sigma\tilde{\alpha_0} \nonumber   \\
   & = \prod_{\xi'\in \im(\chi_a)}\prod_{\xi\in k/im(\chi_a)}(\alpha_0+\xi'+\xi) \\
   & = \prod_{\xi\in k}(\alpha_0+\xi)=a \ ,   \nonumber
\end{align}
since $\alpha_0+\xi,$ $\xi\in k,$ are precisely the zeros of $X^q-X-a$.

Now let $\beta$ be in $\wp^{-1}([a^{p^m}])$ with $\beta_0=\alpha_0^{p^m}$. Then $F(\alpha_0)(\beta) = F(\alpha_0,\beta_0,\ldots,\beta_{n-1}) = F(\alpha_0)(\beta-[\beta_0])$ and
\begin{equation*}
  \wp(\beta-[\beta_0])=[a^{p^m}]-\wp([\alpha_0^{p^m}])
\end{equation*}
belongs to $VW_n(\alpha_0k[\alpha_0])_L$ because $a^{p^m}=(\alpha_0^q-\alpha_0)^{p^m}$ belongs to $ \alpha_0k[\alpha_0]$ as well as $\wp([\alpha_0^{p^m}])=[(\alpha_0^{p^m})^q]-[\alpha_0^{p^m}]\in W_n(\alpha_0k[\alpha_0])_L$ and $[a^{p^m}]_0=a^{p^m}=\wp([\alpha_0^{p^m}])_0$.

Note that for $n=1$ we have $F(\alpha_0)(\beta) = F(\alpha_0)$ (i.e., the last consideration is not needed) and since $a$ is a norm with respect to the extension $F(\alpha_0)/F$ the claim follows from Lemma \ref{properties[)}.iii.

Now let $n>1$. Then, the induction hypothesis for $F' := F(\alpha_0)$ and $a' := \alpha_0$ implies that $[[a^{p^m}]-\wp([\alpha_0^{p^m}]),\alpha_0)=0$, i.e.,  that $\alpha_0 \in \mathrm{Norm}_{F(\alpha_0)(\beta)/F(\alpha_0)}(F(\alpha_0)(\beta)^\times)$ by Lemma \ref{properties[)}.iii.

Replacing $\alpha_0$ by $\alpha_0+\xi,$ $\xi\in k$, we see that also $\alpha_0+\xi \in \mathrm{Norm}_{F(\alpha_0)(\beta)/F(\alpha_0)}(F(\alpha_0)(\beta)^\times)$ (note that $F(\alpha_0) = F(\alpha_0+\xi)$ and the composite $F(\alpha_0)(\beta) = F(\alpha_0)F(\beta)$ does not depend on the choices involved above). By the multiplicativity of the norm we obtain that $\tilde{\alpha_0}$ lies in $\mathrm{Norm}_{F(\alpha_0)(\beta)/F(\alpha_0)}(F(\alpha_0)(\beta)^\times)$, whence by transitivity of the norm and \eqref{f:norm} $a$ belongs to $\mathrm{Norm}_{F(\alpha_0)(\beta)/F}(F(\alpha_0)(\beta)^\times)$ and thus also to $\mathrm{Norm}_{F(\beta)/F}(F(\beta)^\times)$ because $F(\beta) \subseteq F(\alpha_0)(\beta)$. Thus $[[a^{p^m}],a)=0$, again by Lemma \ref{properties[)}.iii, as had to be shown.
\end{proof}

Now we will define a second bilinear pairing
\begin{equation*}
   (\;,\;) : W_n(K)_L \times K^\times \longrightarrow  W_n(k)_L
\end{equation*}
by using the residue pairing (cf.\ \eqref{f:residue-pairing})
\begin{equation*}
   \mathrm{Res} : \rA_L \times \Omega^1_{\rA_L} \longrightarrow   o_L \ .
\end{equation*}
To this end we choose an isomorphism $\rA_L/\pi_L\rA_L=k((Z)) \cong K$ and remark that our construction will not depend on this choice of a prime element of $K$ by Remark \ref{Res-variable}. Consider the map $d\log : o_L((Z))^\times \longrightarrow \Omega^1_{\rA_L},$ sending $f$ to $\frac{df}{f}.$  We define the upper pairing in
\begin{equation*}
\xymatrix{
   W_n(\rA_L)_L \ar[d]_{\Phi_{n-1} }^{ }   \ar@{}[r]|{\times} & {o_L((Z))^\times} \ar[d]_{  d\log }\ar[r]^-{\{\ , \ \}} & o_L \ar@{=}[d]_{}^{} \\
   {\rA_L  } \ar@{}[r]|{\times} & {\Omega^1_{\rA_L}}\ar[r]^-{\mathrm{Res}} & o_L }
\end{equation*}
via the commutativity of the diagram.
%{\color{blue} Kann man $o_L((Z))^\times$ hier nicht auch durch $\rA_L^\times$ ersetzen? Insbesondere: gilt dann unten noch \eqref{f:claim} bzw. ist $h$ immer noch eine Potenzreihe?}

\begin{lemma}\label{pairing2}
There is a unique well defined bilinear pairing $(\ ,\ )$   such that the diagram
\begin{equation*}
\xymatrix{
   W_n(\rA_L)_L \ar@{->>}[d]_{W_n(\alpha_1)_L }   \ar@{}[r]|{\times} & {o_L((Z))^\times} \ar@{->>}[d]_{ \mod \pi_L }\ar[r]^-{\{\ , \ \}} & o_L \phantom{ } \ar@{->>}[d]^{\alpha_n } \\
   {W_n(K)_L} \ar@{}[r]|{\times} & {K^\times}\ar[r]^-{(\ ,\ )} & W_n(k)_L, }
\end{equation*}
is commutative.
\end{lemma}
\begin{proof}
(Note that the reduction map $o_L((Z))^\times \rightarrow K^\times$ indeed is surjective.) We need to show that
\begin{equation*}
  \{\ker(W_n(\pr)_L),o_L((Z))^\times\} \subseteq \pi_L^n o_L  \qquad\text{and}\qquad   \{W_n(\rA_L)_L, \ker(o_L((Z))^\times \rightarrow K^\times)\} \subseteq \pi_L^n o_L \ .
\end{equation*}
For $a = (a_0,\ldots,a_{n-1}) \in W_n(\pr)_L$ such that $a_i \in \pi_L \rA_L$ we obviously have $\Phi_{n-1}(a) \in \pi_L^n \rA_L$. Hence $\{a, o_L((Z))^\times\} \subseteq \pi_L^n o_L$.

For the second inclusion we first observe that $\ker(o_L((Z))^\times \rightarrow K^\times) = 1 + \pi_L o_L[[Z]]$. Hence we have to prove that
\begin{equation}\label{f:claim}
  \mathrm{Res}(\Phi_{n-1}(f)d\log(1+\pi_L h)) \in \pi_L^no_L
\end{equation}
holds true for all $f = (f_0,\ldots,f_{n-1}) \in \rA_L^n$ and $h \in o_L[[Z]]$. We observe that sending $Z$ to $Z' := Z(1 + \pi_L h)$ defines a ring automorphism first of $o_L[[Z]]$, then by localization of $o_L((Z))$, and finally by $\pi_L$-adic completion of $\mathscr{A}_L$. We write $f_i(Z) = g_i(Z')$ and $g := (g_0,\ldots,g_{n-1})$, and we compute
\begin{align*}
  \mathrm{Res}_Z(\Phi_{n-1}(f)d\log(1+\pi_L h)) & = \mathrm{Res}_Z(\Phi_{n-1}(g(Z'))d\log(Z'))-\mathrm{Res}_Z(\Phi_{n-1}(f(Z))d\log(Z)) \\
   & = \mathrm{Res}_{Z'}(\Phi_{n-1}(g(Z'))d\log(Z'))-\mathrm{Res}_Z(\Phi_{n-1}(f(Z))d\log(Z)) \\
   & = \mathrm{Res}_{Z}([\Phi_{n-1}(g(Z))-\Phi_{n-1}(f(Z))]d\log(Z)) \ .
\end{align*}
Here the second equality uses the fact that the residue does not depend on the choice of the variable (cf.\ Remark \ref{Res-variable}) while in the third equality we just rename the variable $Z'$ into $Z$ both in the argument and the index of $\mathrm{Res}$, which of course does not change the value. Now note that, since $Z' \equiv Z \bmod \pi_L o_L[[Z]]$, we have the congruences
\begin{equation*}
  f_i(Z)=g_i(Z')\equiv g_i(Z) \mod \pi_L \mathscr{A}_L  \qquad\text{for any $0 \leq i \leq n-1$}.
\end{equation*}
This implies that
\begin{equation*}
  \Phi_{n-1}(g(Z)) - \Phi_{n-1}(f(Z)) \equiv 0 \mod \pi_L^n o_L
\end{equation*}
(cf.\ \cite{GAL} Lemma \ref{GAL-Phi-q-congruence}.i) whence the claim \eqref{f:claim} by the $o_L$-linearity of the residue.
\end{proof}

\begin{remark}
Alternatively, one can define similarly a pairing by using the full ghost map $\Phi=(\Phi_{0},\ldots,\Phi_{n-1})$ via commutativity of the diagram
\begin{equation*}
\xymatrix{
   W_n(\rA_L)_L \ar@{^(->}[d]_{\Phi_{} }^{ }   \ar@{}[r]|{\times} & {o_L((Z))^\times} \ar[d]_{  d\log }\ar[r] & W_n(o_L)_L \ar@{^(->}[d]_{\Phi_{}}^{} \\
   {\rA_L^n} \ar@{}[r]|{\times} & {\Omega^1_{\rA_L}}\ar[r]^-{\mathrm{Res}} & o_L^n }
\end{equation*}
and by showing that for all $f \in W_n(\rA_L)_L $ and $h \in o_L((Z))^\times $ the residue vector $ (\mathrm{Res}(\Phi_{i}(f)\frac{dh}{h}))_i$ belongs to the image of (the right hand) $\Phi$. We leave it to the interested reader to check that this induces the same pairing as $(\;,\; )$ above by applying $W_n(\pr)_L$ to the target. For unramified Witt vectors this is done  in \cite{Tho} Prop.\ 3.5.
\footnote{Another alternative formulation for the definition of $(\; ,\; )$ goes as follows: The residue pairing
\begin{align*}
\mathrm{Res} : \rA_L/\pi_L^n\rA \times \Omega^1_{\rA_L/\pi_L^n\rA} & \longrightarrow o_L/\pi_L^no_L
\end{align*}
induces the pairing
\begin{equation*}
\xymatrix{
   {\im(\Phi_{n-1})} + \pi_L^n\rA_L/\pi_L^n\rA   \ar@{}[r]|{\times} & {\phantom{mmm}\Omega^1_{\rA_L/\pi_L^n\rA}/d\log(1+\pi_L o_L[[Z]])} \ar[r] & o_L/\pi_L^no_L \phantom{ }   \\
   {W_n(K)_L}\ar[u]_{w_{n-1}} \ar@{}[r]|{\times} & {\phantom{mmmmm}K^\times\phantom{mmmmm}}\ar[u]\ar[r]^-{(\ ,\ )} & W_n(k)_L\ar[u]_{w_{n-1}}, }
\end{equation*}
where the middle vertical map is induced by $d\log$ and the inverse of the isomorphism $o_L((Z))^\times/(1+\pi_Lo_L[[Z]])\cong K^\times$.  }
\end{remark}

Our aim is to show that the two pairings $[\ ,\ )$ and $(\ ,\ )$, in fact, coincide. This generalizes a result of Witt (\cite{Wit} Satz 18), which we learned from \cite{Fo1,Fo2}. The strategy is to reduce this to the comparison of the restrictions of the two pairings to $W_n(k)_L\times K^\times$.

For an element $x = \sum_j x_{j}Z^j\in K$ with $x_{j} \in k$ and $x_{j}=0$ for $j<v_Z(x)$ (the valuation of $K$) we set $x^+ := \sum_{j\geq 1} x_{j}Z^j$ and $x^- := \sum_{j<0} x_{j}Z^j$. Then, for $x = (x_0,\ldots,x_{n-1}) \in W_n(K)_L$, with arbitrary $n\geq 1$, we define iteratively elements (the 'constant term' and the plus and negative parts of $x$ with respect to the variable $Z$)
\begin{equation*}
  \Omega_Z^n(x) \in W_n(k)_L,\ x^+ \in W_n(Zk[[Z]])_L,\ \text{and}\ x^- \in W_n(Z^{-1}k[Z^{-1}])_L
\end{equation*}
such that
\begin{equation*}
  x = \Omega_Z^n(x)+x^+ +x^- \ .
\end{equation*}
For $n = 1$ put
\begin{equation*}
  \Omega_Z^1(x) := [x_0-x_0^+-x_0^-],\ x^+ := [x_0^+],\ \text{and}\ x^- := [x_0^-]
\end{equation*}
and for $n > 1$ define
\begin{equation*}
  \Omega^n_Z(x) := [x_0-x_0^+-x_0^-] + \tau\Omega^{n-1}_Z(y),\ x^+ := [x_0^+] + \tau y^+,\  \text{and}\ x^- := [x_0^-] + \tau y^- \ ,
\end{equation*}
where $y\in W_{n-1}(K)_L$ satisfies $\tau y = x-[x_0-x_0^+-x_0^-] - [x_0^+] - [x_0^-]$.

\begin{remark}\label{unique}
$x = \Omega_Z^n(x)+x^+ +x^-$ is the unique decomposition of $x \in W_n(K)_L$ such that the three summands lie in $W_n(k)_L$, $W_n(Zk[[Z]])_L$, and $W_n(Z^{-1}k[Z^{-1}])_L$, respectively.
\end{remark}
\begin{proof}
Let $x = a + a_+ + a_-$ be any decomposition such that $a \in W_n(k)_L$, $a_+ \in W_n(Zk[[Z]])_L$, and $a_- \in W_n(Z^{-1}k[Z^{-1}])_L$. Since the projection onto the zeroth component is additive we immediately obtain that $a = [x_0 - x_0^+ - x_0^-] + \tau b$, $a_+ = [x_0^+] + \tau b_+$, and $x_- = [x_0^-] + \tau b_-$ for (uniquely determined) elements $b \in W_{n-1}(k)_L$, $b_+ \in W_{n-1}(Zk[[Z]])_L$, and $b_- \in W_{n-1}(Z^{-1}k[Z^{-1}])_L$. We put $y := b + b_+ + b_-$ and obtain $x = [x_0 - x_0^+ - x_0^-] + [x_0^+] + [x_0^-] + \tau y$. Hence $y$ is the element in the above inductive construction for $x$. By induction with respect to $n$ we have $b = \Omega_Z^{n-1}(y)$, $b_+ = y^+$ and $b_- = y^-$. It follows that $a = \Omega^n_Z(x)$, $a_+ = x^+$, and $a_- = x^-$.
\end{proof}

\begin{lemma}\label{omega}
For any prime element $Z$ in $K$ and any $x\in W_n(K)_L$ we have
\begin{equation*}
  (x,Z) = \Omega_Z^n(x) \ .
\end{equation*}
\end{lemma}
\begin{proof}
For $x \in W_n(Zk[[Z]])_L \cup W_n(Z^{-1}k[Z^{-1}])_L$ we may choose the lift $f$ of $x$ to lie in $W_n(Zo_L[[Z]])_L$ and $W_n(Z^{-1}o_L[Z^{-1}])_L$, respectively. It is straightforward to see that then $\{f,Z\} = \mathrm{Res}(\Phi_{n-1}(f)d\log Z) = \mathrm{Res}(\Phi_{n-1}(f) \frac{dZ}{Z}) = 0$. By Remark \ref{unique} we have $\Omega_Z^n(x) = 0$ as well.

By the additivity of $(\ ,\ )$ in the first component it therefore remains to treat the case that $x \in W_n(k)_L$. Let $\tilde{x} \in W_n(W(k)_L)_L \subseteq W_n(\rA_L)_L$ be any lift of $x$. Then we have that $\{\tilde{x},Z\} = Res(\Phi_{n-1}(\tilde{x})\frac{dZ}{Z}) = \Phi_{n-1}(\tilde{x})$. But $\alpha_n(\Phi_{n-1}(\tilde{x}))= \overline{\alpha}_n \circ w_{n-1}(x) = \phi_q^{n-1}(x) = x$  by \eqref{f:A-omega} and \eqref{f:alphaphiq} for $o_L$. Hence $(x,Z) = \alpha_n(\{\tilde{x},Z\}) = \alpha_n(\Phi_{n-1}(\tilde{x})) = x = \Omega_Z^n(x)$, the last identity again by Remark \ref{unique}.
\end{proof}

\begin{lemma}\label{constant}
For any prime element $Z$ in $K$ and any $x\in W_n(k)_L$ we have $[x,Z)=(x,Z)$.
\end{lemma}
\begin{proof}
We choose $\alpha \in W_n(k^{sep})_L$ such that $\wp(\alpha) = x$. Then $K(\alpha) \subseteq k(\alpha)((Z))$ is an unramified extension of $K=k((Z))$. From local class field theory we therefore obtain that $rec_K(Z) = \phi_q$. It follows that $[x,Z) = rec_K(Z)(\alpha)-\alpha = \wp(\alpha) = x$. On the other hand Lemma \ref{omega} implies that $(x,Z) = \Omega_Z^n(x) = x$ as well.
\end{proof}

\begin{proposition}\label{constants}
For any prime element $Z$ in $K$, any $a\in K^\times$, and any $x\in W_n(K)_L$ we have
\begin{equation*}
  [x,a)=[(x,a),Z) \qquad\text{and}\qquad (x,a)=((x,a),Z) \ .
\end{equation*}
\end{proposition}
\begin{proof}
As $\Omega_Z^n((x,a)) = (x,a)$ the second identity is a consequence of Lemma \ref{omega}.

For the fist identity we first consider the \textit{special case} $a=Z$. We will compare the decompositions
\begin{align*}
  [x,Z) & = [\Omega_Z^n(x),Z) + [x^+,Z) + [-x^-,Z^{-1})  \quad\text{and}\\
  [(x,Z),Z) & = [(\Omega_Z^n(x),Z),Z) + [(x^+,Z),Z) + [(-x^-,Z^{-1}),Z)
\end{align*}
term by term. By Lemma \ref{omega} the two first terms coincide and the remaining terms in the second decomposition vanish. The last term in the first decomposition vanishes by Prop.\ \ref{aka}. Hence it remains to show that $[x^+,Z) = 0$. For this it suffices to check that $W_n(Zk[[Z]])_L \subseteq \wp(W_n(K)_L)$. Indeed, we claim that for $y \in W_n(Zk[[Z]])_L$ the series $\sum_{i=0}^\infty \phi_q^i(y)$ converges in $W_n(k[[Z]])_L$ (componentwise in the $Z$-adic topology). We observe that, for $x \in W_n(k[[Z]])_L$ and $z \in W_n(Z^lk[[Z]])_L$ with $l\geq 0$, one has, by \eqref{f:exact-ideal}, the congruence
\begin{equation*}
  (x+z)_i\equiv x_i \mod Z^lk[[Z]]
\end{equation*}
for the components of the respective Witt vectors. It follows that each component of the sequence of partial sums $\sum_{i=0}^m\phi_q^i(y)$ forms a Cauchy sequence. Since $\phi_q$, being the componentwise $q$th power map, obviously is continuous for the topology under consideration we obtain $\wp(-\sum_{i=0}^\infty\phi_q^i(y)) = y$.

For a \textit{general} $a$ we find a $\nu \in \mathbb{Z}$ and another prime element $Z' \in K$ such that $a = Z^\nu Z'$. Using bilinearity and the special case (for $Z$ as well as $Z'$) we compute
\begin{align*}
   [x,a) & = \nu[x,Z)+[x,Z') = \nu [(x,Z),Z) + [(x,Z'),Z') \\
         & = [(x,Z),Z^\nu)+[(x,Z'),Z) = [(x,Z^\nu Z'),Z)   \\
         & = [(x,a),Z) \ ;
\end{align*}
the third equality uses Lemma \ref{constant}.
\end{proof}

\begin{theorem}[Schmid-Witt formula]\label{schmid-witt}
The pairings $[\ ,\ )$ and $(\ ,\ )$ coincide.
\end{theorem}
\begin{proof}
This now is an immediate consequence of Lemma \ref{constant} and Prop.\ \ref{constants}.
\end{proof}

\begin{corollary}\label{Witt}
For all $z\in W_n(K)_L$ and $\hat{u} \in o_L((Z))^\times$ any lift of $u \in K^\times$ we have
\begin{equation*}
  Res(w_{n-1}(z)\frac{d\hat{u}}{\hat{u}}) = w_{n-1}(\partial(z)(rec_K(u))).
\end{equation*}
\end{corollary}
\begin{proof}
Let $f \in W_n(\rA_L)_L$ be a lift of $z$. Thm.\ \ref{schmid-witt} implies that
\begin{equation*}
  \partial(z)(rec_K(u)) = \alpha_n(\mathrm{Res}(\Phi_{n-1}(f)\frac{d\hat{u}}{\hat{u}}))
\end{equation*}
holds true. Applying $w_{n-1}$ (for $o_L$) we obtain
\begin{equation*}
  w_{n-1}(\partial(z)(rec_K(u))) = \mathrm{Res}(\Phi_{n-1}(f)\frac{d\hat{u}}{\hat{u}})  \mod \pi_L^n o_L \ .
\end{equation*}
On the other hand, by \eqref{f:A-omega}, the element $\Phi_{n-1}(f)$ module $\pi_L^n$ is equal to $w_{n-1}(z)$ (with $w_{n-1}$ for { $\mathscr{A}_L \cong \mathbf{A}_L$}).
\end{proof}

We finally are able to establish the congruence {\eqref{f:key-identity}. First note that since $\bA \subseteq W(\bE_L^{sep})_L$  we obtain the commutative diagram
 \begin{equation*}
 \xymatrix{
   0 \ar[r] & o_L/\pi_L^no_L \ar[d]^{w_{n-1}^{-1}=\overline{\alpha_n}}_{\cong}\ar[rr]^{ } && {\bA/\pi_L^n\bA} \ar[d]\ar[rr]^{\phi_q -1} && {\bA/\pi_L^n\bA}\ar[d] \ar[r] & 0  \\
   0\ar[r]^{} & W_n(k)_L \ar[rr]^{ } & & W_n(\bE_L^{sep})_L \ar[rr]^{\phi_q - 1 } && W_n(\bE_L^{sep})_L\ar[r]^{ } & 0.   }
\end{equation*}
We recall that $\partial_\varphi$ and $\partial$ denote the connecting homomorphisms arising from the upper and lower exact sequence, respectively. We obviously have the identity $(\overline{\alpha}_n)_*\circ\partial_\varphi=\partial \circ \overline{\alpha}_n$ for the map $\overline{\alpha}_n$ which was defined in \eqref{f:alpha-n}.

\begin{lemma}\label{lemma-Witt}
For any $z\in \bA_L$ and $\hat{u} \in o_L((\omega_{LT}))^\times \subseteq \bA_L^\times$ any lift of $u\in \bE_L^\times$
%\footnote{For $z = 1$ we have seen already in the proof of Remark \ref{Res-variable}.ii that the left hand side of the above identity only depends on $u$ and not on the lift $\hat{u}$.}
we have
\begin{equation}
  Res(\varphi_L^{n-1}(z)\frac{d\hat{u}}{\hat{u}}) \equiv \partial_\vp(z)(rec_{\bE_L}(u))  \mod \pi_L^n o_L \ .
\end{equation}
\end{lemma}
\begin{proof}
We use the identity in Cor.\ \ref{Witt} for the element $z' := \alpha_n(z)$ with respect to $K = \bE_L$. Since $w_{n-1} \circ \overline{\alpha}_n = \varphi_L^{n-1}$ by \eqref{f:alphaphiq} its left hand side becomes the left hand side of the assertion. For the right hand sides we compute
\begin{equation*}
 w_{n-1}(\partial(\alpha_n(z))(rec_{\bE_L}(u))) = w_{n-1} \circ \alpha_n (\partial_\varphi(z)(rec_{\bE_L}(u))) = \partial_\varphi(z)(rec_{\bE_L}(u)) \ .
\end{equation*}
%{\color{red} In der Behauptung sind manche Terme f\"ur $\mathscr{A}_L$ und andere f\"ur $\mathbf{A}_L$ definiert $\longrightarrow$ basic inconsistency.}
\end{proof}

As explained at the end of section \ref{sec:Kummer} this last lemma implies Prop.\ \ref{reduction-to-char-p}. The proof of Thm.\ \ref{Kummer-commutative} therefore now is complete.

\section{Bloch and Kato's as well as Kato's explicit reciprocity law revisited}\label{sec:CoatesWiles}

In this section we give a proof of a generalization of the explicit reciprocity law of Bloch and Kato (\cite{BK} Thm. 2.1) as well as of (a special case of) Kato's explicit reciprocity law (\cite{Kat} Thm.\ II.2.1.7) replacing his method of syntomic cohomology by generalizing the method of Fontaine in \cite{Fo2} from the cyclotomic to the general Lubin-Tate case.

First we recall some definitions and facts from \cite{Co1}. (This reference assumes that the power series $[\pi_L](Z)$ is a polynomial. But, by some additional convergence considerations, the results can be seen to hold in general (cf.\ \cite{GAL} \S\ref{GAL-sec:coeff} for more details).) The ideal $\mathbb{I}_L \subseteq W(\widetilde{\mathbf{E}}^+)_L$ is defined to be the preimage of $\pi_L o_{\mathbb{C}_p}$ under the surjective homomorphism of $o_L$-algebras $\theta : W(\widetilde{\mathbf{E}}^+)_L \longrightarrow o_{\mathbb{C}_p}$ (cf.\ \cite{GAL} Lemma \ref{GAL-adjunction2}). \cite{Co1} \S8.5 introduces the $\pi_L$-adic completion $A_{max,L}$ of the subalgebra $W(\widetilde{\mathbf{E}}^+)_L[\frac{1}{\pi_L}\mathbb{I}_L] \subseteq W(\widetilde{\mathbf{E}}^+)_L[\frac{1}{\pi_L}]$ as well as $B_{max,L}^+ := A_{max,L}[\frac{1}{\pi_L}] \subseteq B_{dR}^+$. The important point is that the Frobenius $\phi_q$ naturally extends to $A_{max,L} \subseteq B_{max,L}^+$ (but not to $B_{dR}^+$). Let $\overline{\bE}_L$ denote the algebraic closure of $\bE_L$ in $\widetilde{\bE}$ and $\overline{\bE}_L^+$ its ring of integers. Setting $\omega_L := \omega_{LT}$, $\omega_1 := \phi_q^{-1}(\omega_{LT}) \in
W(\overline{\mathbf{E}}^+)_L$, $\xi_L := \omega_{L}\omega_1^{-1} \in W(\overline{\bE}_L^+)_L$ and $t_L := \log_{LT}(\omega_L) \in A_{max,L}$ we have the following properties by Prop.s 9.10 and 9.6 in loc.\  cit.:
\begin{equation}\label{f:invBmax}
     (B_{max,L}^+)^{G_L}=L,
\end{equation}
\begin{equation}\label{f:unitAmax}
     \phi_q(t_L) = \pi_L t_L \quad\text{and}\quad \frac{t_L}{\omega_L} = 1 + \sum_{k\geq 2} e_k\omega_L^{k-1} \in A_{max,L}^\times \quad\text{with $e_k \in \pi_L^{-l_q(k)}o_L$},
\end{equation}
where $l_q(k)$ denotes the maximal integer $l$ such that $q^l\leq k$,
\begin{equation}\label{f:xi}
      W(\widetilde{\mathbf{E}}^+)_L\xi_L = \ker\left(W(\widetilde{\mathbf{E}}^+)_L \xrightarrow{\;\theta\;} o_{\mathbb{C}_p}\right),
\end{equation}
\begin{equation*}
      W(\widetilde{\mathbf{E}}^+)_L\omega_L = \{x \in W(\widetilde{\mathbf{E}}^+)_L :  \theta(\phi_q^i(x))=0\ \text{for all $i \geq 0$}\},
\end{equation*}
\begin{equation}\label{f:tLBmax}
      B_{max,L}^+t_L = \{x \in B_{max,L}^+ : \theta(\phi_q^i(x))=0\ \text{for all $i\geq 0$}\}.
\end{equation}
Since $\omega_L$ is a unit in $W(\mathbf{E}_L)_L$ we may define a $\phi_q$- and $\Gamma_L$-stable ring $A_{suf} \subseteq W(\overline{\mathbf{E}}_L)_L$ as the localization $\Sigma^{-1}W(\overline{\bE}_L^+)_L$ where $\Sigma$ denotes the multiplicatively closed subset generated by $\{\phi_q^m(\omega_L)\}_{m\geq 0}$.  Setting $B_{max,L} := B_{max,L}^+[\frac{1}{t_L}] \subseteq B_{dR}$ we obtain the following fact.

\begin{lemma}\label{suf}
$A_{suf} \subseteq B_{max,L}$.
\end{lemma}

\begin{proof}
Since $W(\overline{\bE}_L^+)_L \subseteq W(\widetilde{\bE}_L^+)_L \subseteq A_{max,L}$ it suffices to prove that $\phi_q^m(\omega_L)$ is invertible in $B_{max,L}$. As $\omega_L = t_L a$ for some $a \in A_{max,L}^\times$ by \eqref{f:unitAmax} this is clear for $m=0$ and follows for $m\geq 1$ because
\begin{equation*}
  \phi_q^m(\omega_L) = \phi_q^m(t_L)\phi_q^m(a) = \pi_L^m t_L \phi_q^m(a) \ .
\end{equation*}
\end{proof}

By \eqref{f:unitAmax} we see that
\begin{equation*}
  \left(\frac{t_L}{\omega_L}\right)^{-r }=: \sum_{m\geq 0} \lambda_{m,r} \omega_L^m
\end{equation*}
belongs to $L[[\omega_L]] \subseteq B_{dR}^+$ for $r \geq 0$.

\begin{lemma}\phantomsection\label{convergenceBmax+}
\begin{itemize}
  \item[i.] $\sum_{m\geq 0} \lambda_{m,r} \omega_L^m$ converges in $A_{max,L}$.
  \item[ii.] $\sum_{m\geq 1} \lambda_{m+r,r} \omega_L^{m-1}$ converges in $B^+_{max,L}$.
\end{itemize}
\end{lemma}
\begin{proof}
First of all we note that the $\pi_L$-adic completion $R$ of the polynomial ring $o_L[Z]$ is the subring of all power series in $o_L[[Z]]$ whose coefficients tend to zero. Using the geometric series we see that $1 + \pi_L R \subseteq R^\times$.

According to \eqref{f:unitAmax} and the proof of \cite{Co1} Prop.\ 9.10 there exists a $g(Z) \in R$ such that
\begin{equation*}
  \left(\frac{t_L}{\omega_L}\right)^r = 1 - \pi_L g(u) \ ,
\end{equation*}
where $u = \frac{\omega_L}{\pi_L}\in \pi_L^{-1}\mathbb{I}_L$ for $q \neq 2$ and $u=\frac{\omega_L}{\pi_L^2}\in \pi_L^{-2}\mathbb{I}_L^2$ for $q=2$, respectively. By the initial observation we have
\begin{equation*}
  (1 - \pi_L g(Z))^{-1} = \sum_{m \geq 0} b_m Z^m \in R \ .
\end{equation*}
Thus
\begin{equation*}
  \left(\frac{t_L}{\omega_L}\right)^{-r} = \sum_{m \geq 0} b_m u^m = \sum_{m\geq 0} \lambda_{m,r} \omega_L^m
\end{equation*}
converges in $A_{max,L}$. For the second part of the assertion it remains to note that
\begin{equation*}
  \sum_{m\geq 1} \lambda_{m+r,r} \omega_L^{m-1} = \pi_L^{-(r+1)} \sum_{m \geq 0} b_{m+r+1} u^m \ .
\end{equation*}
\end{proof}

Setting $\tau_r' := \sum_{m=0}^r \lambda_{m,r}\omega_L^m$ and $\tau_r := \omega_L^{-r}\tau_r' \in L[\frac{1}{\omega_L}] \subseteq W(\overline{\bE}_L)_L[\tfrac{1}{\pi_L}]$ we have
\begin{equation}\label{f:taur}
      \tau_r - t_L^{-r} \in L[[\omega_L]] \subseteq B_{dR}^+ \ .
\end{equation}
By \cite{Co1} Prop.\ 9.25 (SEF 3E) we have the exact sequence
\begin{equation}\label{f:BmaxBdR}
  0 \longrightarrow L \longrightarrow (B_{max,L})^{\phi_q=1} \longrightarrow B_{dR}/B_{dR}^+ \longrightarrow 0 \ .
\end{equation}
We define
\begin{equation*}
  Fil^r B_{max,L}^+ := B_{max,L}^+ \cap t_L^rB_{dR}^+ \qquad\text{for $r \geq 0$}.
\end{equation*}

\begin{lemma}\phantomsection\label{FilBmaxexact}
\begin{itemize}
  \item[i.] For $r \geq 1$ the sequence
\begin{equation*}
  0  \longrightarrow Lt_L^r \longrightarrow Fil^r B_{max,L}^+ \xrightarrow{\pi_L^{-r}\phi_q-1} B_{max,L}^+ \longrightarrow 0
\end{equation*}
is exact.
  \item[ii.] For $r = 0$ the sequence
\begin{equation*}
  0  \longrightarrow L \longrightarrow B_{max,L}^+ \xrightarrow{\phi_q-1} (\phi_q - 1)B_{max,L}^+ \longrightarrow 0
\end{equation*}
is exact, and $((\phi_q - 1)B_{max,L}^+)^{G_L} = L$.
  \item[iii.] $\phi_q - 1$ is bijective on $\omega_L B_{max,L}^+$.
\end{itemize}
\end{lemma}
\begin{proof}
i. and ii. By \cite{Co1} Prop.\ 9.22 we have, for any $r \geq 0$, the exact sequence
\begin{equation*}
  0 \longrightarrow Lt_L^r \longrightarrow (B^+_{max,L})^{\phi_q=\pi_L^r} \longrightarrow B^+_{dR}/t_L^rB^+_{dR} \longrightarrow 0 \ .
\end{equation*}
First we deduce that $Lt_L^r = (B^+_{max,L})^{\phi_q=\pi_L^r} \cap t_L^rB^+_{dR} = (Fil^r B_{max,L}^+)^{\phi_q=\pi_L^r}$. Secondly it implies that $B^+_{dR} = (B^+_{max,L})^{\phi_q=\pi_L^r} + t_L^rB^+_{dR}$ and hence $B_{max,L}^+ = (B^+_{max,L})^{\phi_q=\pi_L^r} + Fil^r B_{max,L}^+$. In the proof of \cite{Co1} Prop.\ 9.25 it is shown that, for $r \geq 1$, the map $B_{max,L}^+ \xrightarrow{\pi_L^{-r}\phi_q-1} B_{max,L}^+$ is surjective. It follows that $B_{max,L}^+ = (\pi_L^{-r}\phi_q-1) Fil^r B_{max,L}^+$ for $r \geq 1$.

It remains to verify the second part of ii. By \eqref{f:invBmax} we have $((\phi_q - 1)B_{max,L}^+)^{G_L} \subseteq L$. For the reverse inclusion it suffices to consider any $a \in o_L \subseteq W(\overline{\mathbf{E}}_L^+)_L$. Since $\overline{\mathbf{E}}_L^+$ is integrally closed the map $\phi_q - 1$ on $W(\overline{\mathbf{E}}_L^+)_L$ is surjective. Hence we find a $y \in W(\overline{\mathbf{E}}_L^+)_L \subseteq B_{max,L}^+$ such that $(\phi_q - 1)y = a$.

iii. First of all we note that $\phi_q(\omega_L B_{max,L}^+) \subseteq \phi_q(\omega_1 \xi) B_{max,L}^+ \subseteq \omega_L B_{max,L}^+$, so that, indeed, $\phi_q - 1$ restricts to an endomorphism of $\omega_L B_{max,L}^+$. By ii. we have $(\omega_L B_{max,L}^+)^{\phi_q = 1} \subseteq (B_{max,L}^+)^{\phi_q = 1} = L$. But $\omega_L B_{max,L}^+ \cap L = 0$ by \eqref{f:xi}. This proves the injectivity. It suffices to establish surjectivity on $\omega_L A_{max,L}$. Let $\omega_L a \in \omega_L A_{max,L}$. Similarly as in the proof of Lemma \ref{suf} we let $\omega_L = t_L u$ with $u \in A_{max,L}^\times$ and compute
\begin{equation*}
  \phi_q^n(\omega_L a) = \phi_q^n(t_L a u) \in \pi_L^n t_L A_{max,L} = \omega_L \pi_L^n A_{max,L} \ .
\end{equation*}
It follows that the series $- \sum_{n\geq 0} \phi_q^n(\omega_L a)$ converges ($\pi_L$-adically) to some element $\omega_L c \in \omega_L A_{max,L}$ such that $(\phi_q-1)(\omega_L c) = \omega_L a$.
\end{proof}

The sequences in Lemma \ref{FilBmaxexact}.i/ii induce, for any $r \geq 0$, the connecting homomorphism in continuous Galois cohomology
\begin{equation*}
  L = (B_{max,L}^+)^{G_L} \xrightarrow{\;\partial^r\;} H^1(L, L t_L^r) \ ,
\footnote{Setting $L^r_{adm} := L \cap (\pi_L^{-r}\phi_q-1)(Fil^rB_{max,L}^+ )$ we still may define
\begin{equation*}
  L^r_{adm} \xrightarrow{\partial^r} H^1(L, L t_L^r)
\end{equation*}
without knowing the right hand surjectivity in Lemma \ref{FilBmaxexact}.i and define $\partial^r$ with source $L^r_{adm}$ instead. In the course of the next Proposition one can then shown that $L^r_{adm} = L$.}
\end{equation*}
 Note that as a $G_L$-representation $L t_L^r$ is isomorphic to $V := L \otimes_{o_L} T$ (cf.\ Lemma \ref{iota}.c and \eqref{f:dlog}). We introduce the composite homomorphism
\begin{equation*}
  \delta^r : L \xrightarrow{\;\partial^r\;} H^1(L, L t_L^r) \xrightarrow{\;\mathrm{res}\;} \Hom_{\Gamma_L}(H_L, L t_L^r) \ .
\end{equation*}
By Lemma \ref{Artin-Schreier} we also have the connecting homomorphism
\begin{equation*}
  \partial_{\varphi} : W(\overline{\bE}_L^{H_L})_L[\tfrac{1}{\pi_L}] \longrightarrow H^1(H_L,L) \ .
\end{equation*}

\begin{proposition}
$\delta^r(a) = \partial_{\varphi}(\tau_r a)  t_L^r$ for any $a \in L$.
\end{proposition}
\begin{proof}
Let $n\geq 0$ such that $\pi_L^n\tau_r\in o_L[\frac{1}{\omega_L}]$, and assume without loss of generality that $a$ belongs to $\pi_L^n o_L$, i.e., that $\tau_r' a \in o_L[\omega_L] \subseteq W(\overline{\bE}_L^+)_L$. In order to compute $\delta^r(a)$ we choose any $\alpha \in Fil^rB_{max,L}^+$ such that $(\pi_L^{-r} \phi_q -1)(\alpha) = a$. Then
\begin{equation*}
  \delta^r(a)(g) = (g-1)\alpha \quad\text{for all $g\in H_L$}.
\end{equation*}
In fact, choosing $\alpha$ is equivalent to choosing $\beta := t_L^{-r}\alpha \in t_L^{-r}B_{max,L}^+ \cap B_{dR}^+$ such that
\begin{equation*}
  (\phi_q -1)(\beta) = b := t_L^{-r} a \ .
\end{equation*}
On the other hand, to compute $\partial_{\varphi}(\tau_r a)$ we note that $\tau_r'a$ and $\xi_L^r$ belong to $W(\overline{\bE}_L^+)_L$ and that, since $\overline{\bE}_L^+$ is integrally closed, the map $\phi_q - \xi_L^r : W(\overline{\bE}_L^+)_L \rightarrow W(\overline{\bE}_L^+)_L$ is surjective (argue inductively with respect to the length of Witt vectors). Hence we find a $y \in W(\overline{\bE}_L^+)_L$ such that
\begin{equation*}
  \varphi_L(y) - \xi_L^r y = \tau'_ra \ .
\end{equation*}
By using \eqref{f:unitAmax} we see that $\omega_1^{-r} = \xi_L^r \omega_L^{-r} \in  t_L^{-r}B_{max,L}^+$. It follows that the element $\beta_0 := \omega_1^{-r}y$ belongs to $A_{suf} \cap t_L^{-r}B_{max,L}^+ \subseteq W(\overline{\bE}_L)_L$ and satisfies
\begin{align*}
    (\phi_q-1)(\beta_0) & = \phi_q(\omega_{1})^{-r}\phi_q(y) - \omega_1^{-r}y  \\
        & = \omega_L^{-r}(\phi_q(y )- \xi_L^ry)   \\
        & = \omega_L^{-r}\tau'_r a = \tau_r a \ .
\end{align*}
Hence
\begin{equation*}
  \partial_\varphi(\tau_r a)(g) = (g-1)\beta_0 \quad\text{for all $g\in H_L$}.
\end{equation*}
We observe that $y \in W(\overline{\bE}_L^+)_L \subseteq B_{dR}^+$, that $\omega_1$ is a unit in $B_{dR}^+$ (since $\theta(\omega_1) \neq 0$ by the first sentence in the proof of \cite{Co1} Prop.\ 9.6), and hence that $\beta_0 \in t_L^{-r}B_{max,L}^+ \cap B_{dR}^+$. At this point we are reduced to finding an element $\gamma \in (\omega_L B_{max,L}^+)^{H_L} \subseteq B_{max,L}^+ \subseteq t_L^{-r}B_{max,L}^+ \cap B_{dR}^+$ such that $(\phi_q -1)(\gamma) = (t_L^{-r} - \tau_r)a$. We then put $\beta := \beta_0 + \gamma$ and obtain
\begin{equation*}
  \delta^r(a)(g) = (g-1)\beta \otimes t_L^r  = (g-1)(\beta_0+\gamma) \otimes t_L^r
   = (g-1)\beta_0 \otimes t_L^r = \partial_\varphi(\tau_r a)(g) \otimes t_L^r
\end{equation*}
for any $g \in H_L$. In order to find $\gamma$ it suffices, because of Lemma \ref{FilBmaxexact}.iii, to observe that
\begin{equation*}
  t_L^{-r} - \tau_r = \omega_L (\sum_{m\geq 1} \lambda_{m+r,r} \omega_L^{m-1}) \in (\omega_LB_{max,L}^+)^{H_L}
\end{equation*}
by Lemma \ref{convergenceBmax+}.ii.
\end{proof}

Now we define the Coates-Wiles homomorphisms in this context for $r \geq 1$ and $m \geq 0$ by\footnote{For $m>0$  one can extend the definition to $\varprojlim_n L_n^\times $ while for $m=0$ one cannot evaluate at $t_{0,0}=0$!}
\begin{align*}
   \psi_{CW,m}^r :  \varprojlim_n o_{L_n}^\times & \longrightarrow L_m \\
                                   u & \longmapsto
    \frac{1}{r!\pi_L^{rm}}\left(\partial_{\mathrm{inv}}^{r-1}\Delta_{LT} g_{u,t_0}\right)_{|Z=t_{0,m}} \ .
\end{align*}
Then the map
\begin{align*}
   \Psi_{CW,m}^r :  \varprojlim_n o_{L_n}^\times & \longrightarrow L_mt_L^r \\
                                   u & \longmapsto \psi_{CW,m}^r(u)t_L^r\ .
\end{align*}
is $G_L$-equivariant (it depends on the choice of $t_0$). In the following we abbreviate $\psi_{CW}^r := \psi_{CW,0}^r$ and $\Psi_{CW}^r := \Psi_{CW,0}^r$ . One might think about these maps in terms of the formal identity
\begin{equation*}
  \log g_{u,t_0} (\omega_{LT}) = \sum_r \psi_{CW}^r(u)t_L^r = \sum_r \Psi_{CW}^r(u)  \qquad\text{in $L[[t_L]] \subseteq B_{dR}$}.
\end{equation*}
But instead of justifying in which sense me may insert $g_{u,t_0} (\omega_{LT}(t_L))$ \footnote{This power series has a constant term: see \cite{Fo2} for a technical solution.} into the logarithm series, we shall only explain (and below  use) the following identity
\begin{equation*}d\log g_{u,t_0}{(\omega_{LT})}=\frac{dg_{u,t_0}{(\omega_{LT})}}{g_{u,t_0}{(\omega_{LT})}}= \sum_{r\geq 1}r\psi_{CW}^r(u) t_L^{r-1}dt_L \ .
\end{equation*}
Indeed, $t_L = \log_{LT}(\omega_{LT})$ implies $\frac{d}{dt_L} \omega_{LT} = g_{LT}(\omega_{LT})^{-1}$ and hence
\begin{equation*}
  \frac{d}{dt_L}f(\omega_{LT}) = g_{LT}(\omega_{LT})^{-1}\frac{d}{d\omega_{LT}}f(\omega_{LT}) = \partial_{\mathrm{inv}}(f)(Z)_{|Z=\omega_{LT}} \ .
\end{equation*}
We calculate
\begin{align*}
  \tfrac{1}{(r-1)!} \left((\tfrac{d}{dt_L} )^{r-1}   \frac{1}{g_{u,t_0}{(\omega_{LT})}}\frac{dg_{u,t_0}{(\omega_{LT})}}{dt_L}\right)_{|''t_L = 0''} & = \tfrac{1}{(r-1)!} ((\partial_{\mathrm{inv}}^{r-1}\Delta_{LT} g_{u,t_0}(Z))_{|Z=\omega_{LT}})_{|''t_L = 0''}   \\
  & = \tfrac{1}{(r-1)!} (\partial_{\mathrm{inv}}^{r-1}\Delta_{LT} g_{u,t_0}(Z))_{|Z=0} \\
  & = r\psi_{CW}^{r}(u) \ .
\end{align*}

\begin{proposition}
For all $a\in L$, $r\geq 1$,   and $u \in \varprojlim_n o_{L_n}^\times$ we have
\begin{equation*}
  ar \psi_{CW}^r(u) = \partial_\varphi(\tau_ra)(rec(u)) \ .
\end{equation*}
\end{proposition}
\begin{proof}
Using $L[[t_L]] { =} L[[\omega_L]] { \subseteq B_{dR}^+}$ we obtain from \eqref{f:taur} that $\tau_r - t_L^{-r} \in L[[t_L]]$. By the discussion before Prop.\  \ref{reduction-to-char-p} we therefore obtain
\begin{align*}
   \partial_\varphi(\tau_ra)(rec(u)) & = \mathrm{Res}_{\omega_L}(\tau_r a d\log g_{u,t_0}{(\omega_{LT})})
    = \mathrm{Res}_{t_L}(\tau_r a d\log g_{u,t_0}{(\omega_{LT})})   \\
    & = \mathrm{Res}_{t_L}( a t_L^{-r}d\log g_{u,t_0}{(\omega_{LT})})
    = \mathrm{Res}_{t_L}( a t_L^{-r}\sum_{n\geq 1}n\psi_{CW}^n(u) t_L^{n-1}dt_L)  \\
    & = ar\psi_{CW}^r(u) \ .
\end{align*}
\end{proof}

With \eqref{f:BmaxBdR} also the sequence
\begin{equation}\label{f:BmaxplusBdR}
   0 \longrightarrow L \xrightarrow{\;diag\;} B_{max,L}^{\phi_q =1} \oplus B_{dR}^+ \xrightarrow{(x,y) \mapsto x-y} B_{dR} \longrightarrow 0
\end{equation}
is exact. Tensoring with $V = L \otimes_{o_L} T$ over $L$ gives the upper exact sequence in the commutative diagram
\begin{equation}\label{f:Ltr}
  \xymatrix{
    0  \ar[r] & V^{\otimes r} \ar[r] & (B_{max,L}^{\phi_q = 1} \oplus B_{dR}^+) \otimes_L V^{\otimes r} \ar[rrr] & & & B_{dR} \, \otimes_L \, V^{\otimes r} \ar[r] & 0 \\
    0 \ar[r] & Lt_L^r \ar[r]^-{diag} \ar[u]_{\cong}^{j(at_L^r) := a t_0^{\otimes r}} & B_{max,L}^{\phi_q =\pi_L^r} \oplus t_L^rB_{dR}^+ \ar[u]^{\cong}_{(x,y) \mapsto (x t_L^{-r}, y t_L^{-r}) \otimes t_0^{\otimes r}} \ar[rrr]^-{(x,y) \mapsto (x-y)t_L^{-r} \otimes t_0^{\otimes r}} & & &  B_{dR} \, \otimes_L \, V^{\otimes r} \ar[r] \ar@{=}[u] & 0.   }
\end{equation}
Passing to continuous $G_L$-cohomology gives rise to the connecting isomorphism
\begin{equation*}
   \xymatrix@R=0.5cm{
                &           H^1(L,V^{\otimes r})      \\
  (B_{dR} \otimes_L V^{\otimes r})^{G_L} = tan_L(V^{\otimes r}) \ar[ur]^{\exp := \exp_{L,V^{\otimes r},\id} \qquad} \ar[dr]                 \\
                &         H^1(L,L t_L^r), \ar[uu]^{\cong}                }
\end{equation*}
which is the identity component (see Appendix) of the Bloch-Kato exponential map over $L$ for $V^{\otimes r}$. We introduce the composite map
\begin{equation*}
  \exp_r : L \xrightarrow[\cong]{a \mapsto a t_L^{-r} \otimes t_0^{\otimes r}} tan_L(V^{\otimes r}) \xrightarrow{\exp_{L,V^{\otimes r},\id} } H^1(L,V^{\otimes r}) \xrightarrow{\;\mathrm{res}\;} \Hom_{\Gamma_L}(H_L,V^{\otimes r}) \ .
\end{equation*}

\begin{proposition}
For all $a\in L$ we have
\begin{equation*}
   j^{-1} \circ\exp_r(a) = -\delta^r((\pi_L^{-r}-1)a) \ .
\end{equation*}

\end{proposition}
\begin{proof}
By \eqref{f:Ltr} we find $(x,y)\in B_{max,L}^{\phi_q =\pi_L^r} \oplus t_L^rB_{dR}^+$ such that  $x-y=a$. Then
\begin{equation*}
   (j^{-1} \circ \exp_r(a))(g) = (g-1)x = (g-1)y  \qquad\text{for all $g \in H_L$}.
\end{equation*}
We {\it claim} that $y$ belongs to $Fil^r B_{max,L}^+$. For this it suffices to prove that $y$ lies in $B_{max,L}^+$ (because it is contained in $t_L^{r}B_{dR}^+$ by assumption). We know that $y = x-a \in B_{max,L} = \bigcup_{s\geq 0} t_L^{-s}B_{max,L}^+$. Let $s$ be minimal with respect to the property that $y \in t_L^{-s}B_{max,L}^+,$ i.e., that $t_L^sy \in B_{max,L}^+$. We want to show that $s=0$. Assume to the contrary that $s>0$. Then $B_{max,L}^+ \ni \phi_q^i(t_L^sy) = \pi_L^{is}t_L^s\phi_q^i(y)$, for any $i \geq 0$, belongs to $Fil^s B_{max,L}^+ \subseteq \ker(\theta)$ because
\begin{equation*}
   \phi_q^i(y) = \phi_q^i(x-a) = \pi_L^{ri}x - a = \pi_L^{ri}y + \pi_L^{ri}a - a \in B_{dR}^+ \ .
\end{equation*}
By \eqref{f:tLBmax} we obtain $t_L^sy = t_Ly'$ for some $y' \in B_{max,L}^+$. Hence $t_L^{s-1}y$ already belongs to $B_{max,L}^+$, which is a contradiction. The above claim follows.

In particular, by the definition of ${\delta}^r$ and using that $y=x-a$ we see that
\begin{equation*}
  (j^{-1} \circ \exp_r(a))(g) = (g-1)y = { \delta}^r ((\pi_L^{-r}\phi_q -1)(y))(g)
 = - { \delta}^r ((\pi_L^{-r} -1)a)(g)  \quad\text{for $g \in H_L$},
\end{equation*}
because $(\pi_L^{-r}\phi_q -1)(x) = 0$ as $x$ belongs to $B_{max,L}^{\phi_q =\pi_L^r}$.
\end{proof}

Putting the previous results together we obtain the following generalization of the explicit reciprocity law of Bloch and Kato (\cite{BK} Thm.\ 2.1) from the cyclotomic to the general Lubin-Tate case. In particular, this confirms partly the speculations in \cite{dS} \S 11: de Shalit had suggested to find a replacement for $B_{max,\mathbb{Q}_p}$ (or rather $B_{cris}$ which was used at that time) in the context of general Lubin-Tate formal groups and it is precisely Colmez' $B_{max,L}$ which has this function (although the path in (loc.\ cit.) is slightly different from the one chosen here).

\begin{theorem}\label{BK}
For all $u \in \varprojlim_n o_{L_n}^\times$, $a \in L$, and $r \geq 1$ we have the identities
\begin{equation*}
  {\delta}^r(a)(rec(u)) = ar \Psi_{CW}^r(u)
\end{equation*}
and
\begin{align*}
   (j^{-1} \circ \exp_r(a))(rec(u)) & = -(\pi_L^{{ -r}} - 1)ar \Psi_{CW}^r(u) \\
   & = \tfrac{1}{(r-1)!}(1 - \pi_L^{{ -r}})a \partial_{\mathrm{inv}}^r \log g_{u,t_0}(Z)_{|Z=0} t_L^r \ .
\end{align*}
\end{theorem}

 Finally we consider the following commutative diagram
\begin{equation*}
\xymatrix{
   H_L^{ab}(p)\otimes_{\mathbb{Z}_p} V^{\otimes -r}   \ar@{}[r]|{\times} & \Hom^c(H_L,V^{\otimes r}) \ar[r] & L \ar@{=}[dd] \\
   {\varprojlim_n L_n^\times \otimes_{\mathbb{Z}_p}V^{\otimes -r}} \ar[u]^{rec\otimes \id} \ar[d]_{\mathrm{cores}(\kappa\otimes \id)} & &\\
   H^1(L, V^{\otimes -r}(1)) \ar[d]_{\exp^*} \ar@{}[r]|{\times} & H^1(L, V^{\otimes r})  \ar[uu]_{\mathrm{res}} \ar[r]^-{\cup} & H^2(L,L(1)) = L \ar[d]_{\cong}^{c \mapsto c t_{\mathbb{Q}_p}^{-1} \otimes t_0^{cyc}} \\
   D^0_{dR,L}(V^{\otimes -r}(1)) \ar@{}[r]|{\times} & tan_L(V^{\otimes r})  \ar[r] \ar[u]_{\exp} & D_{dR,L}(L(1))   \\
  L \ar[u]_{\cong}^{a \mapsto a t_L^r t_{\mathbb{Q}_p}^{-1} \otimes (t_0^{\otimes -r} \otimes t_0^{cyc})} \ar@{}[r]|{\times} & L \ar[u]^{\cong}_{b \mapsto b t_L^{-r} \otimes t_0^{\otimes r}} \ar[r]^{(a,b) \mapsto ab} &  L. \ar[u]^{\cong}_{c \mapsto c t_{\mathbb{Q}_p}^{-1} \otimes t_0^{cyc}}           }
\end{equation*}
Here $t_0^{cyc}$ is a generator of the cyclotomic Tate module $\mathbb{Z}_p(1)$, and $t_{\mathbb{Q}_p} := \log_{\mathbb{G}_m}([\iota (t_0^{cyc})+1]-1)$. The commutativity of the upper part can be shown by taking inverse limits (on both sides) of a similar diagram with appropriate torsion coefficients and afterwards tensoring with $L$ over $o_L$. Its middle part is the definition of the dual exponential map $\exp^*$. The commutativity of the lower part is easily checked. Note also that the composite of the middle maps going up is nothing else than $\exp_r$ by definition. Thus setting $\mathbf{d}_r := t_L^rt_{\mathbb{Q}_p}^{-1} \otimes (t_0^{\otimes -r} \otimes t_0^{cyc})$  we obtain the following consequence.

\begin{corollary}[A special case of Kato's explicit reciprocity law] \label{Kato}
For $r\geq 1$ the diagram
\begin{equation*}
  \xymatrix{
  {\varprojlim_n  o_{L_n}^\times} \otimes_{\mathbb{Z}} T^{\otimes -r} \ar[d]_{\kappa\otimes \id} \ar[rrdd]^{\qquad ''(1-\pi_L^{-r})r\psi_{CW}^r({_-}) \mathbf{d}_r\, ''} &   \\
  H^1_{Iw}(L_\infty,T^{\otimes -r}(1)) \ar[d]_{\mathrm{cores}}  &   \\
  H^1(L, T^{\otimes -r}(1)) \ar[rr]^-{\exp^*} & & D^0_{dR,L}(V^{\otimes -r}(1)) = L \mathbf{d}_r ,  }
\end{equation*}
commutes, i.e., the diagonal map sends $u\otimes a t_0^{\otimes -r}$ to
\begin{equation*}
  a(1-\pi_L^{-r})r\psi_{CW}^r(u) \mathbf{d}_r = a \frac{1-\pi_L^{-r}}{(r-1)!} \partial_{\mathrm{inv}}^r \log g_{u,t_0}(Z)_{| Z=0} \mathbf{d}_r \ .
\end{equation*}
\end{corollary}

\appendix

\section{$p$-adic Hodge theory}

For a continuous representation of $G_K$ on a finite dimensional $\mathbb{Q}_p$-vector space $V$ we write as usual
\begin{align*}
   & D_{dR,K}(V) := (B_{dR} \otimes_{\mathbb{Q}_p}V)^{G_K} \supseteq D_{dR,K}^0(V):= (B_{dR}^+ \otimes_{\mathbb{Q}_p}V)^{G_K} \quad\text{and}  \\
   & D_{cris,K}(V) := (B_{max,\mathbb{Q}_p}\otimes_{\mathbb{Q}_p}V)^{G_K} \ .
\end{align*}
The quotient $tan_K(V) := D_{dR,K}(V)/D_{dR,K}^0(V)$ is called the tangent space of $V$.

Henceforth we assume that $V$ is de Rham. Then the usual Bloch-Kato exponential map $\exp_{K,V}:tan_K(V)\to H^1(K,V)$ can  be defined as follows. Apply the tensor functor $-\otimes_{\mathbb{Q}_p} V$ to the  exact sequence
\begin{equation}\label{f:FESQp}
    0 \to \mathbb{Q}_p \to B_{max,\mathbb{Q}_p}^{\phi_p = 1} \to B_{dR}/B^+_{dR} \to 0
\end{equation}
and take the (first) connecting homomorphism in the associated $G_K$-cohomology sequence.\footnote{\label{ctssection} It follows from \cite[Prop.\ III.3.1]{Co4} that this sequence splits in the category of topological $\bQ_p$-vector spaces. Since the $p$-adic topology on $\bQ_p$ coincides with the induced topology from $B_{max,\bQ_p}$ the existence of the transition map is granted by \cite[Lem.\ 2.7.2]{NSW}.} Note that by \cite{BK} Lemma 3.8.1 we have $tan_K(V)=(B_{dR} /B_{dR}^+ \otimes_{\mathbb{Q}_p}V)^{G_K} .$ } Furthermore, the dual exponential map $exp^*_{K,V}$ is defined by the commutativity of the following diagram
\begin{equation}\label{f:dualexp}
\xymatrix{
  H^1(K,V) \ar[d]_{\cong} \ar[rrr]^{\exp^*_{K,V}} & & & D_{dR,K}^0(V) \ar[d]^{\cong} \\
  H^1(K,V^*(1))^* \ar[rrr]^-{(\exp_{K,V^*(1)})^*} & & & (D_{dR,K}(V^*(1))/D_{dR,K}^0(V^*(1)))^* ,  }
\end{equation}
where the left, resp.\ right, perpendicular isomorphism comes local Tate duality, resp.\ from the perfect pairing
\begin{equation}\label{f:pairingDdR}
  D_{dR,K}(V) \times D_{dR,K}(\Hom_{\mathbb{Q}_p}(V,\mathbb{Q}_p(1)))  \longrightarrow D_{dR,K}(\mathbb{Q}_p(1)) \cong K ,
\end{equation}
in which the $D^0_{dR,K}$-subspaces are orthogonal to each other. Note that the isomorphism $K\cong D_{dR,K}(\mathbb{Q}_p(1))$ sends $a$ to $at^{-1}_{\mathbb{Q}_p} \otimes t_0^{cyc}$. Also, $(-)^*$ here means the $\mathbb{Q}_p$-dual.

Now assume that $V$ is in $Rep_L(G_K)$ and consider $K = L$ in the following. Tensoring \eqref{f:different} with $\mathbb{Q}_p$ gives the isomorphism of $L$-vector spaces
\begin{equation*}
  \tilde{\Xi}: L \cong \Hom_{\mathbb{Z}_p}(o_L,\mathbb{Z}_p)\otimes_{\mathbb{Z}_p}\mathbb{Q}_p \cong \Hom_{\mathbb{Q}_p}(L,\mathbb{Q}_p ) \ .
\end{equation*}
Since $\Hom_{\mathbb{Q}_p}(L,-)$ is right adjoint to scalar restriction from $L$ to $\mathbb{Q}_p$, and by using $\tilde{\Xi}^{-1}$ in the second step, we have a natural isomorphism
\begin{equation}\label{f:Ldual}
  \Hom_{\mathbb{Q}_p}(V,\mathbb{Q}_p) \cong \Hom_{L}(V,\Hom_{\mathbb{Q}_p}(L,\mathbb{Q}_p) \cong \Hom_{L}(V,L) \ .
\end{equation}
Combined with \eqref{f:pairingDdR} we obtain the perfect pairing
\begin{equation}\label{f:pairingDdRoverL}
   D_{dR,L}(V) \times D_{dR,L}(\Hom_{L}(V,L(1))) \longrightarrow L
\end{equation}
with an analogous orthogonality property. Furthermore, similarly as in Prop.\ \ref{Tate-local} local Tate duality can be seen as a perfect pairing  of finite dimensional $L$-vector spaces
\begin{equation}\label{f:Tate-local}
  H^i(K,V) \times H^{2-i}(K,\Hom_{L}(V,L(1)) ) \longrightarrow H^2(K, L(1)) = L \ .
\end{equation}
Altogether we see that, for such a $V$, the dual Bloch-Kato exponential map can also be defined by an analogous diagram as \eqref{f:dualexp} involving the pairings \eqref{f:pairingDdRoverL} and \eqref{f:Tate-local} and in which $(-)^*$ means taking the $L$-dual.

Since $B_{dR}$ contains the algebraic closure $\overline{L}$ of $L$ we have the isomorphism
\begin{equation*}
  B_{dR} \otimes _{\mathbb{Q}_p} V = (B_{dR} \otimes _{\mathbb{Q}_p} L) \otimes_L V  \xrightarrow{\;\cong\;} \prod_{\sigma \in G_{\mathbb{Q}_p}/G_L} B_{dR} \otimes_{\sigma,L} V
\end{equation*}
which sends $b \otimes v$ to $(b \otimes v)_\sigma$. The tensor product in the factor $B_{dR} \otimes_{\sigma,L} V$ is formed with respect to $L$ acting on $B_{dR}$ through $\sigma$. With respect to the $G_L$-action the right hand side decomposes according to the double cosets in $G_L \backslash G_{\mathbb{Q}_p}/G_L$. It follows, in particular, that $D_{dR,L}^{\id}(V) := ( B_{dR} \otimes_L V)^{G_L}$ is a direct summand of $D_{dR,L}(V)$. Similarly, $tan_{L,\id}(V) := ( B_{dR}/B^+_{dR} \otimes_L V)^{G_L}  $ is a direct summand of $tan_L(V)$. We then have the composite map
\begin{equation*}
  \widetilde{\exp}_{L,V,\id} :  tan_{L,\id}(V)  \xrightarrow{\subseteq}  tan_L(V) \xrightarrow{\exp_{L,V}} H^1(L,V), \
\end{equation*}
the identity component of $\exp_{L,V}$. On the other hand, applying the tensor functor $-\otimes_{L} V$ to the  exact sequence \eqref{f:BmaxBdR}
\begin{equation*}
  0 \to L \to B_{max,L}^{\phi_q=1} \to B_{dR}/B^+_{dR} \to 0
\end{equation*}
and taking the (first) connecting homomorphism\footnote{Analogous arguments as in footnote \ref{ctssection} grant the existence of this connecting homomorphism. } in the associated $G_L$-cohomology sequence gives rise to a map
\begin{equation*}
  \exp_{L,V,\id} : tan_{L,\id}(V) \to H^1(L,V) \ .
\end{equation*}
Suppose that $V$ is even $L$-analytic, i.e., that the Hodge-Tate weights of $V$ at all embeddings $\id \neq \sigma : L \rightarrow \overline{L}$ are zero. We then have $tan_L(V) = tan_{L,\id}(V)$ and the following fact.

\begin{proposition}\label{id-component}
If $V$ is $L$-analytic, the maps $ {\exp}_{L,V},$ $\widetilde{\exp}_{L,V,\id} $ and $\exp_{L,V,\id} $ coincide. %with the composed map
%\begin{equation*}
%  tan_{L}(V) \xrightarrow{\pr} tan_{L,id}(V) \xrightarrow{\exp_{L,V,id}} H^1(L,V) \ .
%\end{equation*}
\end{proposition}

Because of this fact we call also $\exp_{L,V,\id}$ the identity component of $\exp_{L,V}$ in the situation of the Proposition. We remark that by \cite{Se0} III.A4 Prop.\ 4 and Lemma 2(b) the character $V = L(\chi_{LT})$ is $L$-analytic.\\

\noindent\textit{Proof of Prop.\ \ref{id-component}:} Let $L_0 \subseteq L$ be the maximal unramified subextension and let $f := [L_0 : \mathbb{Q}_p]$. As explained at the beginning of \S 9.7 in \cite{Co1} we have $B_{max,L} \subseteq B_{max,\mathbb{Q}_p} \otimes_{L_0} L$ and hence $B_{max,L}^{\phi_q = 1} \subseteq B_{max,\mathbb{Q}_p}^{\phi_p^f =1} \otimes_{L_0} L$. We claim that
\begin{equation*}
 B_{max,\mathbb{Q}_p}^{\phi_q =1} = B_{max,\mathbb{Q}_p}^{\phi_p^f =1} = B_{max,\mathbb{Q}_p}^{\phi_p =1} \otimes_{\mathbb{Q}_p} L_0 \ .
\end{equation*}
The left hand side contains $L_0$. Let $\Delta := \Gal(L_0/\mathbb{Q}_p)$ with Frobenius generator $\delta$. For any $x \in B_{max,\mathbb{Q}_p}^{\phi_p^f =1}$ we have the finite dimensional $L_0$-vector space $V_x := \sum_{i=0}^{f-1} L_0 \phi_p^i(x)$ on which the Galois group $\Delta$ acts semilinearly by sending $\delta$ to $\phi_p$. Hilbert 90 therefore implies that $V_x = L_0 \otimes_{\mathbb{Q}_p} V_x^\Delta$. This proves the claim.

It follows that we have the commutative diagram
\begin{equation*}
  \xymatrix{
    0 \ar[r] & L \ar[d] \ar[r] & B_{max,\mathbb{Q}_p}^{\phi_p = 1} \otimes_{\mathbb{Q}_p} L \ar[d]^{=} \ar[r] & B_{dR}/B^+_{dR} \otimes_{\mathbb{Q}_p} L \ar[d]^{mult} \ar[r] & 0  \\
    0 \ar[r] & C  \ar[r] & B_{max,\mathbb{Q}_p}^{\phi_q = 1} \otimes_{L_0} L \ar[r] & B_{dR}/B^+_{dR} \ar[r] & 0 \\
    0 \ar[r] & L \ar[u] \ar[r] & B_{max,L}^{\phi_q = 1} \ar[u]_{\subseteq} \ar[r] & B_{dR}/B^+_{dR} \ar[u]_{=} \ar[r] & 0,   }
\end{equation*}
%\begin{equation*}
%  \xymatrix{
%  0 \ar[r]^{ } & L \ar@{=}[d]_{} \ar[r]^{ } & B_{max,L}^{\phi_q = 1} \ar@{^(->}[d] \ar[r]^{ } & B_{dR}/B_{dR}^+ \ar@{^(->}[d]_{ } \ar[r]^{ } & 0 \\
% 0 \ar[r]& L \ar[r]^{ } & B_{max,\bQ_p}^{\phi_p = 1}\otimes_{\bQ_p}L \ar[r]^{ } & B_{dR}/B_{dR}^+ \otimes_{\bQ_p}L \ar[r]^{ } & 0     }
%\end{equation*}
in which the upper exact sequence is induced by tensoring \eqref{f:FESQp}  by $L$ over $\bQ_p$ while the lower one is \eqref{f:BmaxBdR}. $C$ is defined to be the kernel in the middle horizontal sequence which is therefore also exact. Note that the two vertical maps $L\to C$ both coincide as their composites into the middle term each sends $l\in L$ to $1\otimes l.$ By tensoring this diagram with $V$ over $L$ and forming the cohomology sequences we conclude that the composites of $ {\exp}_{L,V }$ and $\exp_{L,V,\id} $  with $ H^1(G_L,V) \to H^1(G_L,C\otimes_L V)$ coincide whence the claim shall follow from the injectivity of the latter map. The snake lemma applied to the upper part of the diagram (tensored with $V$) leads to the exact sequence\footnote{Using the facts from footnote \ref{ctssection} one checks that this sequence again satisfies the conditions of \cite[Lem.\ 2.7.2]{NSW} whence the existence of the long exact cohomology sequence below is granted.}
\begin{equation*}
  0\to V\to C\otimes_L V \to \prod_{\sigma\neq \id} B_{dR}/  B_{dR}^+\otimes_{\sigma,L} V\to 0 \ ,
\end{equation*}
which in turn, by forming continuous $G_L$-cohomology, induces the exact sequence
\begin{equation*}
  0 = \ker( tan_L(V) \rightarrow tan_{L,\id}(V)) \to H^1(G_L,V) \to H^1(G_L,C\otimes_L V) \ ,
\end{equation*}
i.e., we obtain the desired injectivity.

\end{document}